\documentclass{article}

\usepackage{vmargin,epsfig}
\usepackage[francais,english]{babel}
\usepackage[latin1]{inputenc}
\usepackage[tbtags]{amsmath}
\usepackage{amsthm,amssymb}
\usepackage{enumerate}
\usepackage[all]{xy}
\usepackage{times}
\usepackage{microtype}
\usepackage{url}
\usepackage[algoruled,vlined,french,onelanguage,linesnumbered]{algorithm2e}
\SetKwComment{tcp}{{\it //}\,}{}%

\newtheorem{introtheo}{Théorème}
\newtheorem{theo}{Théorème}[section]
\newtheorem{lemme}[theo]{Lemme}
\newtheorem{prop}[theo]{Proposition}

\theoremstyle{definition}
\newtheorem{deftn}[theo]{Définition}

\theoremstyle{remark}
\newtheorem{rem}[theo]{Remarque}

\def\leq{\leqslant}
\def\geq{\geqslant}

\def\undef{\cdot}

\def\N{\mathbb{N}}
\def\Z{\mathbb{Z}}
\def\Q{\mathbb{Q}}

\def\R{\mathbb{R}}
\def\C{\mathbb{C}}
\def\F{\mathbb{F}}
\def\O{\mathcal{O}}
\def\E{\mathcal{E}}
\def\df{\mathfrak{d}}

\def\epsilon{\varepsilon}

\def\id{\text{\rm id}}
\def\card{\text{\rm Card}}
\def\Diag{\text{\rm Diag}}
\def\Gal{\text{\rm Gal}\:}
\def\GL{\text{\rm GL}\:}
\def\im{\text{\rm im}\:}

\def\Frac{\text{\rm Frac}\:}
\def\Frob{\text{\rm Frob}}

\def\val{\text{\rm val}}
\def\nr{\text{\rm nr}}
\def\mr{\text{\rm mr}}

\def\Sk{\mathfrak{S}}
\def\Mk{\mathfrak{M}}
\def\Dk{\mathfrak{D}}
\def\ss{\text{\rm ss}}
\def\st{\text{\rm st}}

\def\Tst{T_\st^\star}

\def\Mod{\text{\rm Mod}}
\def\Fil{\text{\rm Fil}}

\def\Max{\text{\rm Max}}
\def\HP{\text{\rm HP}}
\def\GeLa{\text{\rm G-L}}

\def\code#1{\textrm{\tt #1}}

\title{Semi-simplifiée modulo $p$ des représentations semi-stables :\\
une approche algorithmique}
\author{Xavier Caruso et David Lubicz\\
  \multicolumn{1}{p{.7\textwidth}}{\centering\emph{IRMAR\\
    Campus de Beaulieu, bâtiments 22 et 23 \\
    263 avenue du Général Leclerc, CS 74205\\
    35042  RENNES Cédex
  }}}
\date\today

\begin{document}

\maketitle
\selectlanguage{english}
\begin{abstract}
The aim of this paper is to present an algorithm the complexity of
which is polynomial to compute the semi-simplified modulo $p$ of a
semi-stable $\Q_p$-representation of the absolute Galois group of a
$p$-adic field (\emph{i.e.} a finite extension of $\Q_p$). 
In order to do so, we use abundantly the $p$-adic Hodge theory and, in
particular, the Breuil-Kisin modules theory.
\end{abstract}
\selectlanguage{francais}

\begin{abstract}
Le but de cet article est de présenter un algorithme de complexité 
polynômiale pour calculer la semi-simplifiée modulo $p$ d'une 
$\Q_p$-représentation semi-stable du groupe de Galois absolu d'un corps 
$p$-adique (\emph{i.e.} une extension finie de $\Q_p$). Pour ce faire, 
nous utilisons abondamment la théorie de Hodge $p$-adique et, en 
particulier, la théorie des modules de Breuil-Kisin.
\end{abstract}

\setcounter{tocdepth}{2}
\tableofcontents

\bigskip

\noindent
\null \hfill \hrulefill \hfill \null

\bigskip

Soient $p$ un nombre premier et $k$ un corps fini de caractéristique 
$p$. Soient $\O_{K_0}$ l'anneau des vecteurs de Witt à coefficients dans 
$k$ et $K_0$ son corps des fractions ; c'est une extension finie non 
ramifiée de $\Q_p$. On appelle $\sigma$ l'endomorphisme de Frobenius 
agissant sur $K_0$. Soit $K$ une extension totalement ramifiée de $K_0$. 
Les corps $K_0$ et $K$ sont munis d'une valuation discrète que l'on note 
$\val$ et que l'on suppose normalisée par $\val(p) = 1$. Soient enfin 
$\bar K$ une clôture algébrique de $K$ et $G_K = \Gal(\bar K/K)$ le 
groupe de Galois absolu de $K$. La valuation $\val$ s'étend de façon 
unique à $\bar K$ et on note encore $\val$ ce prolongement.

Dans de nombreux problèmes, les représentations $p$-adiques du groupe 
$G_K$ jouent un rôle essentiel et il paraît, de nos jours, de plus en 
plus important de mettre au point des outils efficaces pour les 
manipuler sur ordinateur. En effet, la complexité des 
objets est telle qu'il n'y a guère qu'en dimension $2$ que l'on arrive à 
mener à terme des calculs explicites sur ces représentations. Or, le 
besoin se fait de plus en plus sentir de travailler avec des 
représentations de plus grande dimension pour lesquelles, au vu de ce
qui se passe déjà en dimension $2$, il semble quasiment impossible 
d'éviter d'avoir recours à une aide informatique.

Dans cet article, nous nous intéressons au calcul de la semi-simplifiée 
modulo $p$ d'une $\Q_p$-représentation de $G_K$. Rappelons brièvement 
qu'une telle représentation $V$ admet toujours un $\Z_p$-réseau stable 
$T$ par l'action de $G_K$ et que la semi-simplifiée (c'est-à-dire la somme 
directe des constituants simples de Jordan-Hölder) du quotient $T/pT$ ne 
dépend pas, à isomorphisme près, du choix de $T$ ; c'est cette
représentation semi-stable que l'on appelle la \emph{semi-simplifiée
modulo $p$} de $V$. Le résultat principal de cet article s'énonce comme 
suit.

\begin{introtheo}
\label{introtheo:algo}
Il existe un algorithme de complexité polynômiale qui calcule la
semi-simplifiée modulo $p$ d'une représentation semi-stable.
\end{introtheo}

La formulation précédente est volontairement imprécise : par exemple, 
elle ne stipule pas comment sont représentées les entrées et les sorties 
de l'algorithme, ni en quels paramètres la complexité est polynômiale. 
Pour des compléments, on renvoie au corps du texte et notamment au
début du \S \ref{sec:algo}, ainsi qu'au \S \ref{subsec:complexite} 
pour ce qui concerne la complexité.

\medskip

La démonstration du théorème \ref{introtheo:algo} repose de façon 
essentielle sur la théorie de Hodge $p$-adique et, plus précisément, 
sur la théorie des modules de Breuil-Kisin pressentie par Breuil dans 
\cite{breuil2} puis développée par Kisin dans \cite{kisin}. Rappelons
brièvement dans cette introduction que les modules de Breuil-Kisin
sont des objets d'algèbre linéaire d'apparence simple --- à savoir des
modules libres sur l'anneau $\Sk = \O_{K_0}[[u]]$ munis d'un opérateur
$\phi$ satisfaisant à un certain nombre de conditions --- qui classifient 
les réseaux (stables par un certain sous-groupe $G_\infty$ de $G_K$) à 
l'intérieur d'une représentation semi-stable. En utilisant cette 
correspondance, on ramène le calcul que l'on souhaite mener (1)~au 
calcul d'un réseau stable par l'opérateur $\phi$ à l'intérieur d'un 
certain $\phi$-module libre sur $\Sk[1/p]$ puis (2)~au calcul de la 
semi-simplifiée de la réduction modulo $p$ de celui-ci. Grâce aux
travaux de Le Borgne \cite{leborgne}, des algorithmes efficaces ont
déjà été mis au point pour le calcul de la semi-simplifiée. Le calcul du 
réseau, quant à lui, est fondé sur une idée simple : on part d'un réseau
quelconque et on itère l'opérateur $\phi$ jusqu'à obtenir un module
stable.

Malheureusement, derrière ces apparences simples, se cachent un certain 
nombre de complications techniques dues, pour l'essentiel, au fait qu'il 
est impossible de représenter sur machine un élément de $\Sk$ dans son 
intégralité (il y a une infinité de coefficients à donner et, pour 
chaque coefficient, une infinité de \og chiffres \fg) ; à cause de cela,
calculer avec des $\Sk$-modules n'est pas anodin d'un point de vue 
algorithmique ! Toutefois, nous avons montré dans un travail antérieur
\cite{carlub}, que ces problèmes peuvent être résolus (dans une certaine 
mesure) en introduisant de nouveaux anneaux $\Sk_\nu$ (pour un paramètre
$\nu$ variant dans $\Q^+$) définis ainsi :
$$\Sk_\nu = \Big\{ \, \sum_{i \in \N} a_i u^i \quad \big| \quad
a_i \in K_0 \, \text{ et } \, 
\val(a_i) + \nu i \geq 0, \, \forall i \geq 0\, \Big\}$$
et en étendant les scalaires à un nouvel $\Sk_\nu$ (pour un $\nu$ de
plus en plus grand) après chaque opération élémentaire. Ceci nous
conduit à développer une
théorie de Breuil-Kisin sur les anneaux $\Sk_\nu$ qui viennent d'être
introduits. Notre résultat principal, à ce sujet, est le théorème 
\ref{introtheo:surconv} ci-après qui dit informellement que, tant que 
$\nu$ reste suffisamment petit, remplacer l'anneau $\Sk$ par $\Sk_\nu$ 
ne porte pas à conséquence.

\begin{introtheo}[Surconvergence des modules de Breuil-Kisin]
\label{introtheo:surconv}
On se donne $r > 0$ un entier, ainsi que $\nu$ un nombre rationnel
vérifiant $0 \leq \nu < \frac{p-1}{per}$. Alors le foncteur
$$\begin{array}{rcl}
\left\{
\begin{array}{c}
\text{Modules de Breuil-Kisin sur } \Sk \\
\text{de hauteur } \leq r
\end{array}
\right\}
& \longrightarrow &
\left\{
\begin{array}{c}
\text{Modules de Breuil-Kisin sur } \Sk_\nu \\
\text{de hauteur } \leq r
\end{array}
\right\} \medskip \\
\Mk & \mapsto & \Sk_\nu \otimes_\Sk \Mk
\end{array}$$
est une équialence de catégories.
\end{introtheo}

\noindent
Le théorème précédent a fait apparaître la notion de hauteur d'un module 
de Breuil-Kisin. Nous renvoyons le lecteur au \S \ref{subsec:breuilkisin}
pour la définition de cette notion ; pour cette introduction, on pourra 
se contenter de retenir que les modules de Breuil-Kisin de hauteur $\leq 
r$ sont exactement ceux qui correspondent aux représentations 
semi-stables $V$ dont les poids de Hodge-Tate sont dans $\{0, \ldots, 
r\}$, c'est-à-dire pour lesquelles le produit tensoriel $\C_p \otimes 
_{\Q_p} V$ (où $\C_p$ est le complété de $\bar K$) se décompose comme 
une somme directe de $\C_p(h_i)$ pour des entiers $h_i \in \{0, \ldots, 
r\}$.

On déduit du théorème de surconvergence des modules de Breuil-Kisin que, 
si l'on est capable de contrôler la croissance du paramètre $\nu$ au fur 
et à mesure de l'exécution de notre algorithme, la méthode esquissée 
précédemment pour le calcul de la semi-simplifiée modulo $p$ d'une 
représentation semi-stable fonctionne bel et bien. C'est cette voie que 
nous allons suivre tout au long de cet article.

\medskip

La première partie de l'article est destinée à la mise au point des 
aspects théoriques : après quelques rappels, nous introduisons les 
modules de Breuil-Kisin sur les anneaux $\Sk_\nu$ et démontrons le 
théorème \ref{introtheo:surconv}. Les aspects algorithmiques, quant à 
eux, sont discutés dans la seconde partie de l'article, dont l'objectif 
est de présenter et d'étudier en détails (notamment en ce qui concerne 
les problèmes de précision) l'algorithme de calcul de la semi-simplifiée 
modulo $p$ d'une représentation semi-stable promis par le théorème 
\ref{introtheo:algo}.

\bigskip

Ce travail a bénéficié du soutien de l'Agence Nationale de la Recherche 
(ANR) par l'intermédiaire du projet CETHop (Calculs Effectifs en Théorie 
de Hodge $p$-adique), référence ANR-09-JCJC-0048-01.

\section{La théorie de Breuil-Kisin : rappels et compléments}

On conserve les notations de l'introduction : la lettre $p$ désigne un 
nombre premier, $k$ est un corps parfait de caractéristique $p$, on pose 
$\O_{K_0} = W(k)$, $K_0 = \Frac \O_{K_0} = \O_{K_0}[1/p]$, on note 
$\sigma$ l'endomorphisme de Frobenius agissant sur $K_0$ et on considère 
$K$ une extension finie totalement ramifiée de $K_0$ de degré $e$. On 
désigne par $\bar K$ une clôture algébrique de $K$ et par $\C_p$ le 
complété $p$-adique de $\bar K$. Il s'agit à nouveau d'un corps 
algébriquement clos. On pose $G_K = \Gal(\bar K/K)$.

On fixe en outre une uniformisante $\pi$ de $K$, ainsi qu'une suite 
compatible $(\pi_s)_{s \geq 0}$ de racines $p^s$-ièmes de $\pi$, 
c'est-à-dire telle que l'on ait $\pi_0 = \pi$ et $\pi_{s+1}^p = \pi_s$ 
pour tout $s \geq 0$. Enfin, on appelle $K_\infty$ la plus petite 
sous-extension de $\bar K$ contenant tous les $\pi_s$ et on pose 
$G_\infty = \Gal(\bar K/K_\infty) \subset G_K$.

\subsection{Quelques objets de théorie de Hodge $p$-adique}

On introduit dans ce numéro les objets de la théorie de Hodge $p$-adique 
qui seront utilisés constamment dans la suite de cet article : ce sont, 
d'une part, les $(\phi,N)$-modules filtrés de Fontaine qui permettent de 
décrire les représentations semi-stables et, d'autre part, les modules
de Breuil-Kisin qui classifient les réseaux stables par $G_\infty$ à
l'intérieur de celles-ci.

\subsubsection{Brefs rappels sur les $(\phi,N)$-modules filtrés}
\label{subsec:phiN}

Un théorème célèbre de Colmez et Fontaine \cite{colmez-fontaine} affirme 
qu'il existe une équivalence de catégories notée traditionnellement 
$D_{\st}$ ou $D_{\st,\star}$, entre la catégorie des représentations 
semi-stables\footnote{On renvoie à \cite{ast-fontaine} pour la 
définition des représentations semi-stables.} et la catégorie des 
$(\phi,N)$-modules admissibles dont voici la définition :

\begin{deftn}[Fontaine]
\label{def:phiN}
Un $(\phi,N)$-module filtré sur $K$ est la donnée de
\begin{itemize}
\item un $K_0$-espace vectoriel $D$ de dimension finie,
\item un opérateur $\sigma$-semi-linéaire (appelé Frobenius) $\phi : 
D \to D$ bijectif,
\item un opérateur linéaire (appelé opérateur de monodromie) $N : D \to 
D$ vérifiant $N \phi = p \phi N$ et
\item une filtration décroissante $(\Fil^h D_K)_{h \in \Z}$ de $D_K = D 
\otimes_{K_0} K$ telle que $\Fil^{-r} D_K = D_K$ et $\Fil^r D_K = 0$ pour
$r$ suffisamment grand.
\end{itemize}
Si $D$ est un $(\phi,N)$-module filtré, on appelle 
\begin{itemize}
\item \emph{nombre de Newton} de $D$, noté $t_N(D)$, la valuation 
$p$-adique du déterminant de $\phi$\footnote{Le déterminant de $\phi$
dépend du choix d'une base mais sa valuation, elle, n'en dépend pas.} 
et
\item \emph{nombre de Hodge} de $D$, noté $t_H(D)$, la somme des
dimensions de $\Fil^h D_K$ pour $h$ variant dans $\N$.
\end{itemize}
Un $(\phi,N)$-module filtré $D$ est dit \emph{admissible} si $t_H(D)
= t_N(D)$ et si, pour tout $D' \subset D$ stable par $\phi$ et $N$ et 
muni de la filtration induite, on a $t_H(D') \leq t_N(D')$.
\end{deftn}

Dans la suite de cet article, nous travaillerons non pas avec $D_\st$ 
mais avec une version contravariante $D_\st^\star$ de ce foncteur 
définie par $D_\st^\star(V) = D_\st(V^\vee)$ où $V^\vee$ désigne la 
représentation contragrédiente de $V$. Le foncteur réciproque de 
$D_\st^\star$ est noté $V_\st^\star$ ; il associe une représentation
galoisienne semi-stable à un $(\phi,N)$-module filtré admissible.

Avec notre convention, si $D$ est un $(\phi,N)$-module filtré 
admissible, les sauts de la filtration $\Fil^h D_K$ --- c'est-à-dire les 
entiers $h$ tels que $\Fil^{h+1} D_K \subsetneq \Fil^h D_K$ comptées 
avec multiplicité $\dim_K \frac{\Fil^h D_K}{\Fil^{h+1} D_K}$ --- 
sont exactement les poids de Hodge-Tate de la représentation 
$V_\st^\star(D)$\footnote{C'est-à-dire les entiers $h_i$ tels que $\C_p 
\otimes_{\Q_p} V_\st^\star(D) \simeq \bigoplus_{i=1}^{\dim V} 
\C_p(h_i)$.}. En particulier, les poids de Hodge-Tate de $V_\st^\star(D)$
sont positifs ou nuls si, et seulement si $\Fil^0 D_K = D_K$ ; on dit
dans ce cas que $D$ est \emph{effectif}.

\subsubsection{Les modules de Breuil-Kisin}
\label{subsec:breuilkisin}

Après avoir classifié les représentations semi-stables à l'aide de 
des $(\phi,N)$-modules filtrés, il est naturel, 
pour le problème que l'on a en vue, de chercher à décrire les réseaux à 
l'intérieur de ces représentations à l'aide d'objets de même type. La 
théorie de Breuil-Kisin, initiée par Breuil dans \cite{breuil, breuil2}, 
puis complétée par Kisin dans \cite{kisin}, donne une réponse partielle 
--- mais suffisante pour les applications qui nous intéressent --- à cette 
question. Plus précisément, elle permet de décrire, à l'aide de 
$\phi$-modules définis sur l'anneau $\Sk = \O_{K_0}[[u]]$, les 
$\Z_p$-réseaux vivant à l'intérieur d'une représentation semi-stables 
stables par l'action de $G_\infty$ (et non celle de $G_K$)\footnote{Notez
que ceci sera suffisant pour le calcul de la semi-simplifiée modulo $p$ 
qui nous intéresse car les $\F_p$-représentations semi-simples de
$G_\infty$ se trouvent être en bijection avec les $\F_p$-représentations 
semi-simples de $G_K$. On renvoie au \S \ref{subsec:etape3} pour plus de 
détails à ce sujet.}.

Pour décrire cette théorie, on a besoin de deux notations 
supplémentaires. \emph{Primo}, on note $E(u)$ le polynôme minimal de 
$\pi$ sur $K_0$ ; du fait que $\pi$ est une uniformisante de $K$, on 
déduit que $E(u)$ est un polynôme d'Eisenstein à coefficients dans 
$\O_{K_0}$. \emph{Secundo}, on pose $\Sk = \O_{K_0}[[u]]$ et on munit 
cet anneau de l'opérateur $\phi : \sum_{i \in \N} a_i u^i \mapsto 
\sum_{i \in \N} \sigma(a_i) u^{pi}$.

\begin{deftn}[Breuil]
\label{def:modKisin}
Soit $r$ un nombre entier positif. Un \emph{module de Breuil-Kisin} sur 
$\Sk$ de hauteur $\leq r$ est la donnée d'un $\Sk$-module libre $\Mk$ de 
rang fini muni d'une application $\phi : \Mk \to \Mk$ telle que :
\begin{enumerate}
\item pour tout $s \in \Sk$ et tout $x \in \Mk$, on a $\phi(sx) = 
\phi(s) \phi(x)$ ;
\item le sous-$\Sk$-module engendré par l'image de $\phi$ contient
$E(u)^r \Mk$.
\end{enumerate}
\end{deftn}

On dispose en outre d'un foncteur contravariant $\Tst$ qui, à un module 
de Breuil-Kisin de hauteur $\leq r$, associe une $\Z_p$-représentation 
libre de $G_\infty$. Ce foncteur a d'importantes propriétés qui ont été, 
pour la plupart, conjecturées par Breuil puis démontrées par Kisin dans 
\cite{kisin}. 
Le théorème ci-après en donne deux qui nous seront particulièrement 
utiles dans la suite.

\begin{theo}
\label{theo:kisin}
Soit $V$ une représentation semi-stable à poids dans $\{0, \ldots, r\}$.
Alors :
\begin{enumerate}[(1)]
\item il existe un module de Breuil-Kisin $\Mk_0$ de hauteur $\leq r$
tel que $\Q_p \otimes_{\Z_p} \Tst(\Mk_0)$ soit isomorphe à $V$ comme 
$G_\infty$-représentation
\item le foncteur $\Tst$ induit une bijection décroissante entre
l'ensemble des modules de Breuil-Kisin $\Mk \subset \Mk_0[1/p]$ 
tels que $\Mk[1/p] = \Mk_0[1/p]$ et l'ensemble des 
$\Z_p$-réseaux de $V$ stables par $G_\infty$.
\end{enumerate}
\end{theo}

Notons pour terminer ce numéro, qu'en plus de cela, si $\Mk$ est un 
module de Kisin et si $T = \Tst(\Mk)$, alors la représentation quotient 
$T/pT$ peut se retrouver à partir du quotient $\Mk/p\Mk$ grâce à une 
recette explicite que nous ne détaillons par davantage ici car elle ne
nous sera pas utile.

\subsection{Des $(\phi,N)$-modules filtrés aux modules de Breuil-Kisin}

D'après les rappels que nous venons de faire, à une représentation
semi-stable $V$ à poids de Hodge-Tate dans $\{0, \ldots, r\}$ sont
canoniquement associés deux objets, à savoir :
\begin{itemize}
\item le $(\phi,N)$-module filtré admissible effectif $D_\st^\star(V)$ ;
\item l'espace $\Dk(V) = \Mk_0[1/p]$ (muni de son action de $\phi$) si 
$\Mk_0$ est un module de Breuil-Kisin tel que $\Q_p \otimes_{\Z_p} 
\Tst(\Mk_0)$ soit isomorphe à $V$ comme $G_\infty$-représentation.
\end{itemize}
Étant donné que $D_\st^\star(V)$ détermine $V$, il détermine aussi 
$\Dk(V)$. Le but de ce numéro est d'expliquer comment il est possible 
d'obtenir $\Dk(V)$ (ou, du moins, une certaine approximation de 
celui-ci) directement à partir de $D_\st^\star(V)$. Nous suivrons pour 
cela les travaux de Génestier et Lafforgue exposés dans \cite{genlaf}.
Les anneaux $\Sk_\nu$, qui sont déjà apparus dans l'introduction, 
intervenant de façon cruciale dans la 
recette de Génestier et Lafforgue, 
nous consacrons le numéro suivant à rappeler quelques unes de leurs
propriétés essentielles.

\subsubsection{Généralités sur les anneaux $\Sk_\nu$}

Soit $\nu$ un nombre \emph{rationnel} positif ou nul. On rappelle, tout 
d'abord, que $\Sk_\nu$ est défini comme suit :
$$\Sk_\nu = \Big\{ \, \sum_{i \in \N} a_i u^i \quad \big| \quad
a_i \in K_0 \, \text{ et } \,
\val(a_i) + \nu i \geq 0, \, \forall i \geq 0\, \Big\}.$$
Clairement, pour $\nu = 0$, on a $\Sk_\nu = \Sk$ et, si $\nu' \geq \nu$, 
on a une inclusion $\Sk_\nu \subset \Sk_{\nu'}$. En particulier, on a 
toujours $\Sk \subset \Sk_\nu$ ou, autrement dit, $\Sk_\nu$ est 
naturellement une $\Sk$-algèbre. D'un point de vue analytique, l'anneau 
$\Sk_\nu$ apparaît comme l'anneau des fonctions analytiques convergentes 
et bornées par $1$ sur le disque $D_\nu$ de centre $0$ et de rayon 
$|p|^\nu$.
L'anneau $\Sk_\nu[1/p]$ jouera un rôle particulier dans la suite ; on le 
note $\E^+_\nu$ et, lorsque $\nu = 0$, on s'affranchira de l'indice et 
notera simplement $\E^+$. D'un point de vue analytique, $\E^+_\nu$ 
s'identifie à l'anneau des fonctions analytiques convergentes bornées 
sur $D_\nu$.

Le Frobenius $\phi$ défini par $\phi(\sum a_i u^i) = \sum \sigma(a_i) 
u^{pi}$ induit des morphismes d'anneaux $\phi : \Sk_\nu \to \Sk_{\nu/p}$ 
et $\phi : \E^+_\nu \to \E^+_{\nu/p}$ et donc, en particulier, puisque 
$\nu \geq 0$, il induit des endomorphismes de $\Sk_\nu$ et $\E^+_\nu$.

\medskip

On introduit la \emph{valuation de Gauss} $v_\nu$ définie par :
$$v_\nu (f) = \inf_{i \in \N} (\val(a_i) + i \nu)$$
pour une série $f = \sum_{i \in \N} a_i u^i \in \E^+_\nu$. 
On vérifie sans difficulté que $v_\nu$ vérifie les propriétés 
suivantes : pour tous $f,g \in \E_\nu$, on a $v_\nu(fg) = v_\nu(f) + 
v_\nu(g)$ et $v_\nu(f+g) \geq \min (v_\nu(f), v_\nu(g))$. De plus, une 
série $f$ comme précédemment est dans $\Sk_\nu$ si, et seulement si 
$v_\nu(f) \geq 0$.
Comme $\nu$ est supposé rationnel, la quantité $\val(a_i) + i \nu$ 
varie dans le sous-groupe discret $\Z + \nu \Z$ de $\R$. Il en résulte que 
$v_\nu$ prend également ses valeurs dans ce sous-groupe discret, et est 
donc une valuation discrète. Concrètement, si $\frac a b$ est une 
écriture de $\nu$ sous forme irréductible, alors $\Z + \nu \Z = \frac 1 
b \Z$ et, si $s$ et $t$ désignent deux entiers tels que $as-bt = 1$, 
l'élement $\frac{u^s}{p^t}$ a pour valuation $\frac 1 b$.
On définit le \emph{degré de Weierstrass} $\deg_\nu(f)$ d'une série $f 
\in \E^+_\nu$ non nulle comme le plus petit entier $i$ tel que $v_\nu(f) 
= \val(a_i) + i \nu$. On vérifie que, pour $f, g \in \E^+_\nu$ tous les 
deux non nuls, on a $\deg_\nu(fg) = \deg_\nu(f) + \deg_\nu(g)$. On a 
alors la proposition suivante qui implique, en particulier, que 
$\E^+_\nu$ est un anneau euclidien (pour le stathme $\deg_\nu$), ce qui 
s'avèrera fort utile au \S \ref{sec:algo} pour les applications 
algorithmiques.

\begin{prop}[Division euclidienne]
Soient $f$ et $g$ deux éléments de $\Sk_\nu$ avec $v_\nu(g) \geq 
v_\nu(f)$. Alors, il existe un unique couple $(q,r) \in \Sk_\nu^2$ 
tel que (1)~$r$ est un \emph{polynôme} de degré $< \deg_\nu(f)$ et
(2)~on a la relation $g = fq + r$.
\end{prop}

Comme corollaire de cette proposition, on démontre le théorème de 
préparation de Weierstrass qui affirme que tout élément $f \in \E^+_\nu$ 
s'écrit comme un produit $f_1 f_2$ où $f_1$ est un \emph{polynôme} et 
$f_2$ un élément inversible de $\E^+_\nu$. Si, en outre, $f$ est pris 
dans $\Sk_\nu$, alors on peut choisir $f_1$ et $f_2$ de sorte qu'ils 
appartiennent eux aussi à $\Sk_\nu$ et, de surcroît, que $f_2$ soit 
inversible dans cet anneau (et pas uniquement dans $\E^+_\nu$).

\medskip

On rappelle enfin qu'à chaque élément $f = \sum a_i u^i \in \E^+_\nu$, 
on peut associer un polygone de Newton $F$ défini comme l'enveloppe 
convexe des points de coordonnées $(i,\val(a_i))$ et d'un point à 
l'infini de coordonnées $(0,\infty)$ ; il s'agit donc d'un sous-ensemble 
convexe de $\R^2$. Le fait que $f \in \E^+_\nu$ entraîne que les pentes 
de ce polygone qui sont strictement inférieures à $-\nu$, comptées avec 
multiplicité\footnote{La multiplicité d'une pente est, par définition, 
sa longueur sur l'axe des abscisses.}, sont en nombre fini. De surcroît, 
elles correspondent aux valuations des racines de $f$ (vue comme 
fonction analytique) dans le disque $D_\nu$, comptées également avec
multiplicité.

On en déduit que les pentes $< - \nu$ du produit $fg$ sont exactement 
les réunion des pentes $< - \nu$ de $f$ et de $g$, comptées avec 
multiplicité. Attention, ceci n'est plus vrai pour les pentes $\geq - 
\nu$. Un contre-exemple très simple avec $\nu = 0$ est donné par $f = 
1-u$ et $g = \frac 1{1-u} = 1 + u + u^2 + \cdots$. En effet, $f$ a alors 
une unique pente $0$ de multiplicité $1$, $g$ a une unique pente $0$ de 
multiplicité $+\infty$, mais pourtant le produit $fg = 1$ n'a aucune 
pente. On a toutefois le lemme suivant qui nous sera utile dans la 
suite.

\begin{lemme}
\label{lem:prodNewton}
Soient $f$ et $g$ deux éléments de $\E^+_\nu$. Soit $\mu$ un nombre
rationnel. On note $m_f$ (resp. $m_g$) la multiplicité (éventuellement
nulle) de la pente $\mu$ dans le polygone de Newton de $f$ (resp. de
$g$). On suppose que l'on est dans un des deux cas suivants (non
exclusifs) :
\begin{enumerate}[(i)]
\item les deux multiplicités $m_f$ et $m_g$ sont finies ;
\item l'une des multiplicités parmi $m_f$ et $m_g$ est nulle.
\end{enumerate}
Alors la multiplicité de la pente $\mu$ dans le polygone de Newton de 
$fg$ est $m_f + m_g$.
\end{lemme}

\begin{proof}
Quitte à remplacer, dans la définition de $\Sk_\nu$, le corps des 
coefficients $K_0$ par une extension finie totalement ramifiée, puis à 
effectuer un changement de variables et à multiplier $f$ et $g$ par des 
puissances adéquates de l'uniformisante de $K_0$, on peut supposer que 
$\mu = 0$, que $f$ et $g$ sont à coefficients dans $\Sk$ et que 
$v_0(f) = v_0(g) = 0$.
Soient $\bar f$ et $\bar g$ les réductions respectives de $f$ et $g$ 
modulo l'idéal maximal de $\O_{K_0}$ ; ce sont \emph{a priori} des 
éléments de l'anneau $k[[u]]$ des séries formelles à coefficients dans 
$k$, qui ne sont pas nuls d'après la condition sur $v_0$ qui a été 
supposée. Soit $v_f$ (resp. $v_g$) la valuation de $\bar f$ (resp. 
$\bar g$).

Dans le cas (i), les séries $\bar f$ et $\bar g$ sont des polynômes
et on a $\deg(\bar f) = v_f + m_f$ et $\deg(\bar g) = v_g + m_g$. 
Ainsi $\deg(\bar f \bar g) = (v_f + v_g) + (m_f + m_g)$. Étant donné
que $v_f + v_g$ est la valuation du produit $\bar f \bar g$, on en
déduit que la multiplicité de la pente $\mu = 0$ dans le polygone de 
Newton de $fg$ est $m_f + m_g$, comme souhaité.

Passons maintenant au cas (ii). On suppose $m_f = 0$ pour fixer les idées. 
On peut supposer de surcroît que $m_g = +\infty$ car sinon la situation 
relève du cas précédent. La série $\bar f$ est alors un monôme, tandis 
que $\bar g$ n'est pas un polynôme. On en déduit que $\bar f \bar g$ 
n'est pas non plus un polynôme et donc que la multiplicité de la pente 
$\mu = 0$ dans le polygone de Newton de $fg$ est infinie. C'est bien ce 
que l'on voulait démontrer.
\end{proof}

\subsubsection{La méthode de Génestier et Lafforgue}
\label{subsec:GL}

Soit $D$ un $(\phi,N)$-module filtré effectif admissible dont la 
représentation galoisienne semi-stable correspondante est notée $V$. Le 
premier alinéa du théorème \ref{theo:kisin} affirme qu'il existe un 
module de Breuil-Kisin $\Mk_0$ tel que $\Tst(\Mk_0)$ soit un réseau
pour
$G_\infty$ à l'intérieur de $V$, tandis que le second alinéa de ce même 
théorème montre que $\Dk = \Mk_0[1/p]$ ne dépend pas du choix de 
$\Mk_0$. Ainsi, l'espace $\Dk$ muni de l'action de $\phi$ est 
canoniquement associé à $D$. Le but de ce numéro est d'expliquer, en 
suivant \cite{genlaf}, comment construire $\Dk$ directement à partir de 
$D$, sans passer par les représentations galoisiennes.

\medskip

En suivant la terminologie de Génestier et Lafforgue, la première étape 
consiste à associer à $D$ une structure de Hodge-Pink. Soit $\hat \Sk$ 
le complété de $\E^+$ pour la topologie $E(u)$-adique (\emph{i.e.} celle 
relative à l'idéal principal engendré par $E(u)$). La flèche $u \mapsto 
\pi + u_\pi$ définit un isomorphisme canonique $\E^+/E(u)^n \to 
K[u_\pi]/u_\pi^n$ pour tout entier $n$ et donc, par passage à la limite, 
un isomorphisme canonique entre $\hat \Sk$ et l'anneau $K[[u_\pi]]$ des 
séries formelles à coefficients dans $K$. En particulier, cette
identification permet de définir une structure canonique de $K$-algèbre
sur $\hat \Sk$. D'autre part, on remarque que l'idéal principal $(E(u))$ 
correspond, dans $K[[u_\pi]]$, à l'idéal maximal $u_\pi \:
K[[u_\pi]]$. Ainsi l'identification $\hat \Sk \simeq K[[u_\pi]]$ se prolonge
en un isomorphisme canonique $\hat \Sk[\frac 1{E(u)}] \simeq K((u_\pi))$.
La dérivation $u \: \frac d{du}$ envoie $E(u)^h$ sur un multiple de 
$E(u)^{h-1}$ pour tout $h$ et définit de ce fait un endomorphisme 
$K_0$-linéaire de $\hat \Sk$ ; on le note $\hat N$. \emph{Via} 
l'identification $\hat \Sk \simeq K[[u_\pi]]$, on a $\hat N = (u_\pi + 
\pi) \frac d {du_\pi}$. Ceci montre en particulier que $\hat N$ est 
$K$-linéaire (et pas seulement $K_0$-linéaire). Bien sûr, on dispose 
également de la formule de Leibniz habituelle $\hat N(ab) = a \hat N(b) 
+ b \hat N(a)$ pour tous $a,b \in \hat \Sk$.

Étant donné un $(\phi,N)$-module filtré $D$, on s'intéresse à l'espace 
$\hat \Dk = \hat \Sk[\frac 1{E(u)}] \otimes_{\phi, K_0} D$ où le $\phi$ 
en indice dans le produit tensoriel signifie que $\Sk[\frac 1{E(u)}]$ 
est vu comme une $K_0$-algèbre \emph{via} le Frobenius $\phi : K_0 \to 
K_0 \to \Sk[\frac 1{E(u)}]$ où la première flèche est le Frobenius 
$\sigma$ et la seconde est l'inclusion canonique. En \og développant 
\fg\ le produit tensoriel, on obtient une identification canonique
$\hat \Dk \simeq \hat \Sk[\frac 1{E(u)}] \otimes_K D^\phi_K$ avec 
$D^\phi_K = K \otimes_{\phi,K_0} D$.
Le Frobenius $\phi$ définit un isomorphisme $K_0$-linéaire $K_0 
\otimes_{\phi,K_0} D$ qui s'étend de façon unique en un isomorphisme
$K$-linéaire $\phi_K : D^\phi_K \to D_K$. L'opérateur de monodromie
$N$ s'étend lui aussi par linéarité en un endomorphisme $N_K$ de $D_K$. 
On appelle également $N^\phi_K$ l'endomorphisme de 
$D^\phi_K$ défini par $N^\phi_K = p \otimes N$. La relation $N \phi = p 
\phi N$ se réécrit alors $N_K \phi_K = \phi_K N^\phi_K$. On définit, 
à présent, un opérateur différentiel $\hat N$ agissant sur $\hat \Dk
\simeq \hat \Sk[\frac 1{E(u)}] \otimes_K D^\phi_K$
par la formule 
$$\hat N = \hat N \otimes \id + \id \otimes N^\phi_K.$$
La relation de Leibniz est à nouveau vérifiée pour ce dernier $\hat
N$ : pour tout $s \in \hat \Sk$ et tout $x \in \hat \Dk$, on a $\hat
N(sx) = \hat N(s) \: x + s \: \hat N(x)$.

\medskip

On est enfin prêt à introduire les structures de Hodge-Pink de
Génestier et Lafforgue :

\begin{deftn}[Génestier-Lafforgue]
Soit $D$ un $(\phi,N)$-module filtré.
Une \emph{structure de Hodge-Pink} $V_D$ sur $D$ est un $\hat 
\Sk$-module libre de rang $d$ inclus dans $\hat \Sk[\frac 1{E(u)}] 
\otimes_{\phi, K_0} D$.

On dit que $V_D$ satisfait à la \emph{transversalité de Griffiths} 
si $\hat N(V_D) \subset \frac 1{E(u)} V_D$.
\end{deftn}

Une structure de Hodge-Pink sur $D$ définit une filtration 
$\Fil^h_\HP D_K$ indexée par $\Z$ sur l'espace $D_K = K \otimes_{K_0} 
D$ déterminée par la relation :
$$\phi_K^{-1}(\Fil^h_\HP D_K) = \text{image de } E(u)^h V_D \cap (\hat 
\Sk \otimes_{K_0} D) \text{ dans } \Sk/E(u)\Sk \otimes_{K_0} D$$
ce dernier espace s'identifiant à $D^\phi_K$ par l'intermédiaire de 
l'isomorphisme canonique $\Sk/E(u)\Sk \to K$, $u \mapsto \pi$. Le fait 
que $\Fil^h_\HP D_K = D_K$ pour tout $h \leq 0$ est équivalent à $\hat 
\Sk \otimes_{K_0} D \subset V_D$. De même $\Fil^{r+1}_\HP D_K = 0$ si, 
et seulement si $V_D \subset E(u)^{-r} \hat \Sk \otimes_{K_0} D$.

\begin{lemme}
\label{lem:uniqueHP}
Soit $D$ un $(\phi,N)$-module filtré.
Alors, il existe une unique structure de Hodge-Pink $V_D$ sur $D$
satisfaisant à la transversalité de Griffiths et dont la filtration
associée s'identifie à la filtration $\Fil^h D_K$ donnée.
\end{lemme}

\begin{proof}
Voir lemme 1.3 de \cite{genlaf}.
\end{proof}

\medskip

Pour pouvoir continuer à appliquer la méthode de Génestier et Lafforgue, 
on a besoin d'une structure de Hodge-Pink incluse dans $\hat \Sk 
\otimes_{\phi,K_0} D$. Or cela n'est manifestement pas le cas de $V_D$. 
Pour se ramener au cas où cette hypothèse est satisfaite, on fixe un 
entier $r$ supérieur ou égal à tous les poids de Hodge-Tate de la 
représentation associée au $(\phi,N)$-module filtré $D$ et on 
\emph{twiste} $D$ comme ceci : on pose $D' = D$ que l'on munit de 
$\phi' = \frac{\phi}{p^r}$ et de la filtration décalée définie par 
$\Fil^h D' = \Fil^{h+r} D$ pour tout $h \in \Z$. Le structure de 
Hodge-Pink associée à $D'$ est alors $V_{D'} = E(u)^r V_D$ et elle 
satisfait à l'hypothèse de Génestier et Lafforgue.
Toujours suivant ces deux auteurs, étant donné $L$ un 
sous-$\O_{K_0}$-module de $D$, on définit à présent une suite 
$(\beta_n(L))_{n \geq 0}$ de sous-$\Sk$-modules de $\E^+ \otimes_{K_0} 
D' = \E^+ \otimes_{K_0} D$ par la formule récurrente suivante :
\begin{eqnarray}
\beta_0(L) & = & \Sk \otimes_{\O_{K_0}} L, \nonumber\\
\beta_{n+1}(L) &=& \phi'\big((\Sk \otimes_{\phi, \Sk} \beta_n(L)) \cap
V_{D'}\big) 
\,\,=\,\, \frac{\phi}{p^r}\big((\Sk \otimes_{\phi, \Sk} \beta_n(L)) \cap
E(u)^r V_D\big).\label{eq:betan}
\end{eqnarray}
La formule que l'on vient d'écrire présente deux petites subtilités : 
\emph{primo}, l'intersection $(\Sk \otimes_{\phi, \Sk} \beta_n(L)) 
\cap E(u)^r V_D$ est calculée dans l'espace $\hat \Sk[\frac 1 {E(u)}] 
\otimes_{\phi, K_0} D$ et \emph{secundo}, la lettre $\phi$ (resp. 
$\phi'$) désigne ici l'application \emph{linéaire} $\E^+ \otimes_{\phi, 
K_0} D \to \E^+ \otimes D$ déduite du Frobenius $\phi$ (resp. $\phi' = 
\frac{\phi}{p^r}$) agissant sur $D$. On prendra garde au fait que l'on 
ne peut pas définir un Frobenius sur $\hat \Sk[\frac 1 {E(u)}]$ étant 
donné qu'aucune puissance de $\phi(E(u))$ n'est divisible par $E(u)$ ; 
ainsi on ne peut pas non plus définir de Frobenius sur le gros espace 
$\hat \Sk[\frac 1 {E(u)}] \otimes_{\phi, K_0} D$ ; toutefois, cela n'est 
aucunement nécessaire car l'intersection que l'on considère est incluse 
dans l'espace $\E^+ \otimes_{\phi, K_0} D$ sur lequel les opérateurs 
$\phi$ et $\phi'$ sont bien définis.

Lorsque $L = D$, on note simplement $\beta_n$ à la place de 
$\beta_n(D)$. Les $\beta_n$ sont des sous-$\E^+$-modules libres de $\E^+ 
\otimes_{\phi,K_0} D$ et on démontre sans difficulté, par récurrence sur 
$n$, que $\E^+ \otimes_\Sk \beta_n(L) = \beta_n$ pour tout entier 
$n$ et tout sous-$\O_{K_0}$-module $L$ de $D$.
Grâce à cette remarque, le théorème suivant résulte des travaux de 
Génestier et Lafforgue.

\begin{theo}[Génestier-Lafforgue]
\label{theo:GL}
Pour tout entier $n$, l'espace $\E^+_{1/(ep^n)} \otimes_{\E^+} \beta_n 
\subset \E^+_{1/(ep^n)} \otimes_{K_0} D$ est stable par l'opérateur 
$(\frac{E(u)}{E(0)})^r \phi \otimes \phi$. De plus, on a un isomorphisme 
canonique $\xi_n : \E^+_{1/(ep^n)} \otimes_{\E^+} \beta_n \to 
\E^+_{1/(ep^n)} \otimes_{\E^+} \Dk$ rendant le diagramme suivant 
commutatif :
$$\xymatrix @C=50pt {
\E^+_{1/(ep^n)} \otimes_{\E^+} \beta_n \ar[r]^-{\sim}_-{\xi_n} 
\ar[d]_-{(\frac{E(u)}{E(0)})^r \phi \otimes \phi} &
\E^+_{1/(ep^n)} \otimes_{\E^+} \Dk \ar[d]^-{\phi \otimes \phi} \\
\E^+_{1/(ep^n)} \otimes_{\E^+} \beta_n \ar[r]^-{\sim}_-{\xi_n} &
\E^+_{1/(ep^n)} \otimes_{\E^+} \Dk }$$
où la flèche de droite provient du Frobenius $\phi$ agissant sur 
$\Dk$ donné par la théorie de Breuil-Kisin.
\end{theo}

\begin{rem}
\label{rem:GL}
Il suit du résultat de Génestier et Lafforgue que l'on vient d'énoncer 
que si $L$ est un $\O_{K_0}$-réseau de $D$, il existe une constante
$C$ et un module de Breuil-Kisin $\Mk_{1/(ep^n)}$ tels que :
$$\Sk_\nu \otimes_\Sk \beta_n(L) \subset
\Mk_{1/(ep^n)} \subset p^{-C} \cdot (\Sk_\nu \otimes_\Sk \beta_n(L)).$$
En réalité, en examinant de près la démonstration de Génestier et 
Lafforgue, on se rend compte que la constante $C$ ci-dessus dépend de 
façon explicite et \emph{polynômiale}\footnote{Cela sera très important 
dans la suite lorsque l'on étudiera la complexité de l'algorithme du 
calcul de la semi-simplifiée modulo $p$.} de la dimension $d$, de 
l'entier $n$ et du plus petit entier $v$ tel que $p^v \phi(L) \subset 
L$.
\end{rem}

\subsection{Surconvergence des modules de Breuil-Kisin}
\label{sec:sousconv}

Dans le \S \ref{subsec:GL}, nous avons été amené à considérer 
l'extension des scalaires à certains anneaux $\Sk_\nu$ de modules de 
Breuil-Kisin. L'objectif de cette partie est de montrer que ce foncteur 
d'extension des scalaires possède d'excellentes propriétés lorsque $\nu$ 
reste petit.

\medskip

On introduit pour cela la définition suivante, copie conforme de la 
définition des modules de Breuil-Kisin usuels.

\begin{deftn}
\label{def:modKisinnu}
Soient $\nu$ un nombre rationnel positif ou nul et $r$ un nombre entier 
strictement positif. Un \emph{module de Breuil-Kisin} sur 
$\Sk_\nu$ de hauteur $\leq r$ est la donnée d'un $\Sk_\nu$-module libre 
$\Mk_\nu$ de rang fini muni d'une application $\phi : \Mk_\nu \to
\Mk_\nu$ telle 
que :
\begin{enumerate}
\item pour tout $s \in \Sk_\nu$ et tout $x \in \Mk_\nu$, on a $\phi(sx) = 
\phi(s) \phi(x)$ ;
\item le sous-$\Sk_\nu$-module engendré par l'image de $\phi$ contient
$E(u)^r \Mk_\nu$.
\end{enumerate}
\end{deftn}

Pour un nombre rationnel $\nu \geq 0$, on note $\Mod^{r,\phi}_{/\Sk_\nu}$ la 
catégorie des modules de Kisin de hauteur $\leq r$ sur $\Sk_\nu$ et, si 
$0 \leq \nu \leq \nu'$, on note $F_{\nu \to \nu'} : 
\Mod^{r,\phi}_{/\Sk_\nu} \to \Mod^{r,\phi}_{/\Sk_{\nu'}}$ le foncteur 
d'extension des scalaires de $\Sk_\nu$ à $\Sk_{\nu'}$.

\begin{theo}
\label{theo:surconv}
Le foncteur $F_{0 \to \nu}$ est une équivalence de catégories dès que 
$\nu < \frac{p-1}{per}$.
\end{theo}

La démonstration de ce théorème va occuper toute la fin de cette partie. 
Elle repose sur les deux lemmes suivants que nous démontrerons 
respectivement dans le \S \ref{subsec:pleinfid} et le \S 
\ref{subsec:lemapprox}.

\begin{lemme}
\label{lem:pleinfid}
Si $\nu$ et $\nu'$ sont des nombres rationnels positifs ou nuls
vérifiant $\nu' \leq \nu < \frac{p-1}{er}$ et $\nu' < \frac{p-1}{per}$, 
alors le foncteur $F_{\nu' \to \nu}$ est pleinement fidèle.
\end{lemme}

\begin{lemme}
\label{lem:approx}
Pour tout $\nu < \frac{p-1}{per}$ 
et tout
objet $\Mk_\nu$ de $\Mod^{r,\phi}_ {/\Sk_\nu}$, il existe un objet $\Mk
\in \Mod^{r,\phi}_{/\Sk}$ et un isomorphisme $F_{0 \to \nu'} (\Mk) \to
F_{\nu \to \nu'}(\Mk_\nu)$ dans la catégorie
$\Mod^{r,\phi}_{/\Sk_{\nu'}}$ où $\nu'$ est défini comme le rationnel
solution de l'équation $\frac 1 {\nu'} = \frac 1 \nu - er$.
\end{lemme}

Expliquons à présent comment le théorème \ref{theo:surconv} se déduit
des deux lemmes précédents. Clairement la pleine fidélité de 
$F_{0 \to \nu}$ suit du lemme \ref{lem:pleinfid} en prenant $\nu'
= 0$. Il ne reste donc plus qu'à démontrer l'essentielle surjectivité.
Pour cela, on remarque qu'en termes plus concis mais plus abstraits, le 
lemme \ref{lem:approx} affirme que, si $\nu < \frac{p-1}{per}$ et si 
$\nu'$ est défini par l'égalité $\frac 1 {\nu'} = \frac 1 {\nu} - er$ 
alors, dans le diagramme suivant : 
$$\xymatrix @R=15pt @C=5pt { & 
\Mod^{r,\phi}_{/\Sk} \ar[ld]_-{F_{0 \to \nu}} \ar[rd]^-{F_{0 \to \nu'}} 
\\ \
\Mod^{r,\phi}_{/\Sk_{\nu}} \ar[rr]_-{F_{\nu \to \nu'}} & & 
\Mod^{r,\phi}_{/\Sk_{\nu'}} }$$ 
les foncteurs $F_{0 \to \nu'}$ et $F_{\nu \to \nu'}$ ont même image 
essentielle. L'essentielle surjectivité de $F_{0 \to \nu}$ suit alors
directement de la pleine fidélité des $F_{0 \to \nu'}$ et $F_{\nu \to 
\nu'}$, résultats que l'on connaît grâce au lemme \ref{lem:pleinfid}.

\subsubsection{Démonstration du lemme \ref{lem:pleinfid}}
\label{subsec:pleinfid}

Soient $\nu$ et $\nu'$ deux nombres rationnels tels que $0 \leq \nu' 
\leq \nu < \frac{p-1}{er}$ et $\nu' < \frac{p-1}{per}$. On se donne 
$\Mk_{\nu'}$ et $\Mk'_{\nu'}$ deux modules de Breuil-Kisin sur 
$\Sk_{\nu'}$, ainsi qu'un morphisme $f : \Sk_\nu \otimes_{\Sk_{\nu'}} 
\Mk_{\nu'} \to \Sk_\nu \otimes_{\Sk_{\nu'}} \Mk'_{\nu'}$ dans la 
catégorie $\Mod^{r,\phi}_{/\Sk_\nu}$. On souhaite démontrer que $f$ 
envoie $\Mk_{\nu'}$ sur $\Mk'_{\nu'}$.
L'argument de base se développe comme suit. Si on sait que 
$f(\Mk_{\nu'}) \subset A \otimes_{\Sk_{\nu'}} \Mk'_{\nu'}$, pour un certain 
$\Sk_{\nu'}$-module $A \subset \Sk_\nu$, alors 
$$\begin{array}{l} E(u)^r f(\Mk_{\nu'}) = f(E(u)^r \Mk_{\nu'}) \subset 
f(\left< \phi(\Mk_{\nu'}) \right>_{\Sk_{\nu'}})
= \left< f \circ \phi(\Mk_{\nu'}) \right>_{\Sk_{\nu'}} \medskip \\
\hphantom{E(u)^r f(\Mk_{\nu'})} = \left< \phi \circ f(\Mk_{\nu'}) 
\right>_{\Sk_{\nu'}}
\subset \left< \phi (A \otimes_{\Sk_{\nu'}} \Mk'_{\nu'}) \right>_{\Sk_{\nu'}}
\subset \left<\phi(A)\right>_{\Sk_{\nu'}} \otimes_{\Sk_{\nu'}} \Mk'_{\nu'}.
\end{array}$$
Ainsi, si l'on note $\frac\phi{E(u)^r}(A)$ le sous-$\Sk_{\nu'}$-module de 
$\Sk_\nu$ formé des $x$ tels que $E(u)^r x \in 
\left<\phi(A)\right>_{\Sk_{\nu'}}$, on a démontré que $f(\Mk_{\nu'}) 
\subset \frac\phi{E(u)^r}(A) \otimes_{\Sk_{\nu'}} \Mk'_{\nu'}$.
En appliquant le même argument avec $\frac\phi{E(u)^r}(A)$ à la place
de $A$, on 
obtient alors $f(\Mk_{\nu'}) \subset (\frac\phi{E(u)^r})^2(A)
\otimes_{\Sk_{\nu'}} \Mk'_{\nu'}$ 
et, ainsi de suite, $f(\Mk_{\nu'}) \subset (\frac\phi{E(u)^r})^n (A)
\otimes_{\Sk_{\nu'}} 
\Mk'_{\nu'}$ pour tout entier $n$.
Pour conclure, il reste donc à montrer que :
\begin{equation}
\label{eq:inclplfid}
\bigcap_{n \geq 0} \Big(\frac\phi{E(u)^r}\Big)^n (\Sk_\nu) = \Sk_{\nu'}.
\end{equation}

\begin{lemme}
\label{lem:phiSk}
Pour tout réel $\mu \in [0, \frac{p-1}{er}[$ on a 
$\frac{\phi}{E(u)^r}(\Sk_{1/\mu}) \subset \Sk_{1/\mu'}$ avec
$\mu' = \min \big( \frac 1 {\nu'}, \: p \mu - er \big)$.
\end{lemme}
\begin{proof}
Il suffit de montrer que si $f$ est un élément de $\Sk_{\nu'}$ tel que 
$E(u)^r f \in \phi(\Sk_{1/\mu})$, alors $f \in \Sk_{1/\mu'}$. Supposons 
donc qu'il existe $g \in \Sk_{1/\mu}$ tel que $E(u)^r f = \phi(g)$. Soit 
$F$ (resp. $G$) le polygone du Newton de $f$ (resp. de $g$). Si on 
appelle $\Phi$ la fonction de $\R^2$ dans $\R^2$ qui à $(x,y)$ associe 
$(px,y)$, l'épigraphe du polygone de Newton de $\phi(g)$ est $\Phi(G)$. 
On note encore $E$ l'épigraphe du polygone de Newton de $E(u)$.
Comme $\nu' < \frac{p-1}{per}$, la pente $- \frac 1 e$ ne peut avoir
une multiplicité infinie dans le polygone $F$. On déduit alors du lemme 
\ref{lem:prodNewton} et de l'égalité $E(u)^r f = \phi(g)$, que $E + F = 
\Phi(G)$.

Or, par définition, $G$ est inclus dans le demi-espace défini par 
l'inégalité $y + \frac x \mu \geq 0$. Ainsi les couples $(x,y)$ dans 
$\Phi(G)$ vérifie $y + \frac x {p\mu} \geq 0$. Si maintenant $(x,y) \in F$, 
on déduit de l'égalité $E + F = \Phi(G)$ que $(x+er,y) \in \Phi(G)$ et 
donc que $y + \frac {x+er}{p\mu} \geq 0$.
En d'autres termes, si on écrit $f = \sum_{i \in \N} a_i u^i$, on vient 
de démontrer que $\val(a_i) \geq -\frac {i + er}{p\mu}$ pour tout $i$.
Comme $\val(a_i)$ est un entier, on en déduit que $\val(a_i) \geq \lceil
-\frac {i + er}{p\mu} \rceil$ pour tout $i$ (où $\lceil x \rceil$ dénote
la partie entière supérieure du réel $x$). En particulier, si $i < 
p \mu - er$, on obtient $\val(a_i) \geq 0$. Si, au contraire, $i 
\geq p \mu - er$, on peut écrire :
$$\val(a_i) \geq -\frac {i + er}{p\mu} \geq - \frac i {p \mu - er}
\geq - \frac i {\mu'}$$
l'inégalité du milieu résultant de l'hypothèse faite sur $\mu$. Ainsi,
dans tous les cas (\emph{i.e.} quelle que soit la valeur de $i$), on a 
$\val(a_i) + \frac i {\mu'} \geq 0$ pour tout $i$. Cela signifie que $f 
\in \Sk_{1/\mu'}$ comme souhaité.
\end{proof}

Il suit du lemme \ref{lem:phiSk} que $(\frac{\phi}{E(u)^r})^n(\Sk_\nu) 
\subset \Sk_{1/\mu_n}$ où la suite $(\mu_n)_{n \geq 0}$ est définie par 
récurrence par $\mu_0 = \frac 1 \nu$ et $\mu_{n+1} = \min \big( \frac 1 
{\nu'}, \: p \mu_n - er \big)$ pour tout entier $n \geq 0$. Puisque $\nu 
< \frac {p-1}{er}$, on a $\mu_0 > \frac{er}{p-1}$ et il suit de là que 
la suite des $\mu_n$ converge vers $\frac 1 {\nu'}$ (s'il n'y avait pas 
le $\min$, elle tendrait par $+\infty$). On en déduit l'égalité 
\eqref{eq:inclplfid} de laquelle résulte, comme cela a déjà été 
expliqué, la pleine fidélité du foncteur $F_{\nu' \to \nu}$.

\subsubsection{Démonstration du lemme \ref{lem:approx}}
\label{subsec:lemapprox}

On en vient maintenant à la démonstration du lemme \ref{lem:approx}. On 
se donne pour cela $\nu$ et $\Mk_\nu$ comme dans l'énoncé. On fixe une 
$\Sk_\nu$-base de $\Mk_\nu$ et on note $G_0$ la matrice du Frobenius 
$\phi$ dans cette base. Par hypothèse, il existe une matrice $H_0$ telle 
que $H_0 G_0 = E(u)^r$.

\begin{lemme}
\label{lem:Snup}
Tout élément $g \in \Sk_\nu$ se décompose sous la forme $E(u)^r x + y$ 
avec $x = \sum_{i \geq 1/\nu'} x_i u^i \in \Sk_{\nu'}$ et $y \in \Sk$.
\end{lemme}

\begin{proof}
Comme $E(u)$ est un polynôme d'Eisenstein de degré $e$, on peut 
décomposer $E(u)^r$ sous la forme $E(u)^r = u^{er} - p F(u)$ où $F(u)$ 
est un polynôme à coefficients dans $\O_{K_0}$.
Soit $g = \sum_{i \in \N} a_i u^i$ un élément de $\Sk_\nu$. Il se 
décompose sous la forme $g = E(u)^r x_1 + y_1 + g_1$ avec
$$x_1 = \sum_{i \geq 1/\nu} a_i u^{i-er} 
= \sum_{i \geq 1/\nu'} a_{i+er} u^i \quad ; \quad
y_1 = \sum_{0 \leq i < 1/\nu} a_i u^i \quad ; \quad
g_1 = p F(u) \cdot x_1.$$
Un calcul facile montre, d'une part, que $x_1 \in \Sk_{\nu'}$, $y_1 \in 
\Sk$ et $g_1 \in \Sk_\nu$ et, d'autre part, que $v_{\nu'}(x_1) \geq 
v_\nu(g_1)$, $v_\nu(y_1) \geq v_\nu(g_1)$ et $v_\nu(g_1) \geq v_\nu(x_1) 
+ (1 - er\nu) > v_\nu(x_1) + \frac 1 {p-1}$.
Répétant maintenant la construction précédente à partir de $g_1$, puis
itérant le procédé, on obtient des suites $(x_n)$, $(y_n)$ et 
$(g_n)$ prenant leurs valeurs respectivement dans $\Sk_{\nu'}$, $\Sk$ et 
$\Sk_\nu$ telles que, pour tout entier $n$, on ait $g = E(u)^r (x_1 + 
x_2 + \cdots + x_n) + (y_1 + y_2 + \cdots + y_n) + g_n$. 
On dispose en outre des estimations suivantes sur les valuations :
$$v_{\nu'}(x_n) \geq v_\nu(g_1) + \frac{n-1}{p-1}
\quad ; \quad 
v_{\nu}(y_n) \geq v_\nu(g_1) + \frac{n-1}{p-1}
\quad ; \quad 
v_{\nu}(g_n) \geq v_\nu(g_1) + \frac{n}{p-1}$$
valables pour tout entier $n \geq 1$.
Il en résulte que les suites $(x_n)$, $(y_n)$ et $(g_n)$ tendent toutes 
les trois vers $0$. On peut donc définir $x = \sum_{n \in \N} x_n \in 
\Sk_{\nu'}$ et $y = \sum_{n \in \N} y_n \in \Sk$. Ces éléments vérifient
l'égalité $g = E(u)^r x + y$ ; on a donc bien établi la décomposition
annoncée.
\end{proof}

Par le lemme, la matrice $G_0$ se décompose sous la forme $G_0 = Y_0 - 
E(u)^r X_0$ où $X_0$ est une matrice à coefficients dans $\Sk_{\nu'}$ 
\og multiple de $u^{1/\nu'}$ \fg\ et $Y_0$ est à coefficients dans $\Sk$. 
On pose $P_0 = I + X_0 H_0$, où $I$ désigne la matrice identité (de 
taille adéquate). Manifestement, $P_0$ est à coefficients dans 
$\Sk_{\nu'}$ et la condition sur $X_0$ entraîne qu'elle est inversible.
En outre, la matrice produit $G_1 = P_0 G_0 \phi(P_0)^{-1} = (G_0 + 
E(u)^r X_0) \phi(P_0)^{-1} = Y_0 \phi(P_0)^{-1}$ est, elle, à 
coefficients dans $\Sk_{\nu'/p}$ puisque le premier facteur est à 
coefficients dans $\Sk$ et que le second est à coefficients dans 
$\Sk_{\nu'/p}$.

Comme la matrice $P_0 G_0 = Y_0$ est à coefficients dans 
$\Sk$, son déterminant $\delta$ appartient aussi à $\Sk$. Par ailleurs, 
celui-ci s'écrit comme le produit de $\det P_0$, qui est un élément 
inversible dans $\Sk_{\nu'}$, et de $\det G_0$ qui s'écrit, à son tour, 
comme le produit d'une certaine puissance $E(u)^N$ de $E(u)$ et d'une 
unité de $\Sk_\nu$. 
Ainsi $\delta = a E(u)^N$ où $a$ est un élément 
inversible de $\Sk_{\nu'}$. En particulier, le coefficient constant de 
$a$ est inversible dans $\O_{K_0}$. Grâce au lemme \ref{lem:prodNewton},
on en déduit que le polygone de Newton de 
$\delta$ passe par le point de coordonnées $(0,N)$ et contient un segment 
de pente $-\frac 1 e$ qui a une longueur $eN$ sur l'axe des abscisses. 
Comme ce polygone est entièrement situé au-dessus de l'axe des abscisses 
(puisque $\delta \in \Sk$), tous ses autres segments ont nécessairement 
une pente $\geq 0$. Comme ces autres segments sont exactement ceux qui 
forment le polygone de Newton de $a$, il s'ensuit que $a$ est en fait 
élément de $\Sk$ et, mieux encore, qu'il est inversible dans cet anneau.

Les diviseurs élémentaires de la matrice $P_0 G_0$ vue comme matrice à 
coefficients dans l'anneau principal $\E^+ = \Sk[1/p]$ sont donc tous (à 
multiplication par des unités près) des puissances de $E(u)$, qui est 
premier dans cet anneau. Lorsque l'on étend les scalaires à 
$\E^+_{\nu'}$, ces diviseurs élémentaires restent les mêmes et sont 
également ceux de $G_0$ puisque $P_0 G_0$ s'obtient manifestement à 
partir de $G_0$ en multipliant à gauche par une matrice inversible à 
coefficients dans $\Sk_{\nu'}$, et donc \emph{a fortioti} dans
$\E^+_{\nu'}$. Ceci entraîne que les exposants sur ces 
diviseurs élémentaires sont tous $\leq r$ et donc qu'il existe une 
matrice $H'$ à coefficients dans $\E^+$ vérifiant $H' P_0 G_0 
= E(u)^r$. Par ailleurs, vue la forme de $\delta$, il existe également
une matrice $H''$ à coefficients dans $\Sk$ telle que $H'' P_0 G_0 = 
E(u)^N$ pour un entier $N$ que l'on peut bien sûr choisir supérieur
à $r$. On en déduit que $E(u)^{N-r} H' = H''$ et donc que $E(u)^{N-r}
H'$ est à coefficients dans $\Sk$. De $v_0(E(u)) = 0$, on déduit que
$H'$ est également à coefficients dans $\Sk$.
La matrice produit $H_1 = \phi(P_0) H'$ est alors à coefficients dans 
$\Sk_{\nu'/p}$ et vérifie $H_1 G_1 = \phi(P_0) H' P_0 G_0 \phi(P_0)^{-1} 
= E(u)^r$. On en déduit que la matrice $G_1$ définit un module de 
Breuil-Kisin $\Mk_{\nu'/p}$ sur $\Sk_{\nu'/p}$ qui est isomorphe à 
$\Mk_\nu$ après extension des scalaires à $\Sk_{\nu'}$. Or un calcul 
immédiat montre que $\frac{\nu'} p = \frac {\nu} {p - p er \nu}$ et 
donc, par l'hypothèse faite sur $\nu$, que $\frac{\nu'} p < \nu$.

On itère maintenant le processus précédent. On construit comme ceci des 
matrices $G_n$ à coefficients dans $\Mk_{\nu_n}$ définissant des modules 
de Breuil-Kisin sur ces anneaux qui deviennent isomorphes à $\Mk_{\nu}$ 
après extension des scalaires à $\Sk_{\nu'}$. Ici, la suite réelle 
$(\nu_n)_{n \geq 0}$ est définie récursivement par $\nu_0 = \nu$ et 
$\nu_{n+1} = \frac {\nu_n} {p- p er\nu_n}$. De la décroissance de la 
suite des $\nu_n$ (déjà mentionnée précédemment), on déduit que $\nu_n 
\leq \nu \cdot (p - per\nu)^{-n}$. Comme le facteur élevée à la 
puissance $(-n)$ dans l'expression précédente est par hypothèse $> 1$, 
on obtient la convergence vers $0$ de la suite des $\nu_n$. Il en 
résulte que, si $\nu'_n$ est le réel associé à $\nu_n$, la suite des 
$\nu'_n$ tend aussi vers $0$ quand $n$ tend vers l'infini. La condition 
de \og divisibilité par $u^{1/\nu'}$ \fg\ qui apparaît dans le lemme 
\ref{lem:Snup} assure ainsi que le produit infini $\prod_{n \geq 0} (I + 
X_n H_n)$ converge dans $\Sk_{\nu'}$ vers une limite $P_\infty$. La 
matrice $G_\infty = P_\infty G_0 \phi(P_\infty)^{-1}$ est alors la 
limite des $G_n$, et elle est donc à coefficients dans $\Sk$. Comme 
précédemment, on montre qu'il existe $H_\infty$ tel que $H_\infty 
G_\infty = E(u)^r$, et donc que $G_\infty$ définit un module de 
Breuil-Kisin sur $\Sk$. Enfin, l'égalité $G_\infty = P_\infty G_0 
\phi(P_\infty)^{-1}$ montre que celui-ci est isomorphe à $\Mk_\nu$ après 
extension des scalaires à $\Sk_{\nu'}$.

\section{L'algorithme de calcul de la semi-simplifiée modulo $p$}
\label{sec:algo}

Le but de cette seconde partie est de décrire un algorithme qui prend en 
entrée une représentation semi-stable $V$ --- donnée par l'intermédiaire 
de son $(\phi,N)$-module filtré admissible $D$ --- et renvoie sa 
semi-simplifiée modulo $p$ notée $\bar V^\ss$. Quitte à twister par une 
puissance du caractère cyclotomique, on peut toujours supposer que tous 
les poids de Hodge-Tate de $V$ sont positifs ou nuls ; du point de vue 
des $(\phi,N)$-modules filtrés, cela revient à supposer que $D$ est
effectif. Cette hypothèse permet de simplifier l'exposition ; nous la
ferons systématiquement dans la suite.

De façon très schématique, l'algorithme que nous voulons décrire se 
décompose en trois étapes que voici :

\medskip

\emph{Étape 1} : Calcul du module $\Dk$ associé à $D$ (voir \S 
\ref{subsec:GL} pour la définition de $\Dk$)

\medskip

\emph{Étape 2} : Calcul d'un module de Breuil-Kisin $\Mk$ à l'intérieur 
de $\Dk$

\medskip

\emph{Étape 3} : Calcul de la représentation $\bar V^\ss$ à partir de 
$\Mk/p\Mk$

\medskip

\noindent
Chacune de ces trois étapes fait l'objet d'une sous-partie de ce numéro :
l'étape 1 est traitée au \S \ref{subsec:etape1}, l'étape 2 au \S 
\ref{subsec:etape2} et l'étape 3 au \S \ref{subsec:etape3}.
Le \S \ref{subsec:repr} regroupe un certain nombre de 
préliminaires algorithmiques concernant la représentation des objets sur 
machine ; ces précisions aboutiront notamment à une explication précise 
et rigoureuse des spécifications de l'algorithme que l'on cherche à 
écrire, ainsi que quelques hypothèses simples sur le modèle de calcul 
que l'on considère. Enfin, au \S \ref{subsec:complexite}, on détermine 
la complexité de l'algorithme.

\subsection{Représentation des objets sur machine}
\label{subsec:repr}

Avant toute chose, il est important d'expliquer comment les différents 
objets que nous serons amenés à manipuler, à savoir :
\begin{itemize}
\item les éléments des différents anneaux de nombres $p$-adiques :
$\O_{K_0}$, $K_0$, $\O_K$ et $K$,
\item les éléments des anneaux de séries $\Sk_\nu$ et $\E^+_\nu$,
\item les $(\phi,N)$-modules filtrés (c'est l'entrée de notre
algorithme !),
\item les $\F_p$-représentations semi-simples de $G_K$ (c'est la
sortie des notre algorithme !), et
\item les modules de Breuil-Kisin sur $\Sk$ et $\Sk_\nu$,
\end{itemize}
peuvent se représenter sur machine. C'est l'objet de ce numéro.

\subsubsection*{Les corps $p$-adiques et leurs éléments}

Par définition, un élément $x \in \O_{K_0} = W(k)$ est une suite 
\emph{infinie} $x = (x_1, x_2, \ldots)$ de composantes appartenant à 
$k$. Or, stocker --- et \emph{a fortiori} manipuler --- une telle suite 
sur machine est tout simplement impossible, la mémoire d'un ordinateur 
n'étant certainement pas infinie.

Traditionnellement (comme c'est d'ailleurs le cas pour les nombres 
réels), on résout ce problème en travaillant avec des approximations. 
Concrètement, un élément $x \in \O_{K_0}$ sera représenté en machine par 
un couple $(n, \tilde x)$ où $n$ est un entier positif --- la 
\emph{précision} --- et $\tilde x$ est une \emph{approximation} de $x$ 
vivant dans un anneau exact et vérifiant $x \equiv \tilde x \pmod{p^n}$. 
Si $\O_{K_0} = \Z_p[X]/F(X)$ pour un certain polynôme $P$ à coefficients 
entiers, l'anneau exact dont il vient d'être question peut, par exemple, 
être $\Z[X]/F(X)$ et, bien sûr, au lieu de considérer $\tilde x$ 
lui-même, on peut se contenter de travailler avec $\tilde x \text{ mod } 
p^n$, ce qui revient à dire que l'on peut considérer que $\tilde x$ est 
élément de $\Z[X] /(p^n, P(X)) \simeq \Z_q/p^n$ (mais le modulo $p^n$ 
dépend alors de l'élément $x$ considéré). 

Pareillement, on représente les éléments de $K_0$ par une donnée de 
précision et une approximation qui vit dans l'anneau exact $\Z[1/p][X] 
/P(X)$. En ce qui concerne $\O_K$ et $K$, on procède de même en 
considérant un nouveau polynôme définissant l'extension $K/K_0$, qui 
peut par exemple être le polynôme $E(u)$. À l'instar de ce qui se passe 
avec les nombres réels, cette représentation n'empêche nullement 
d'effectuer des opérations ; par exemple, si $x$ et $y$ sont connus à 
précision $p^n$, alors on peut calculer $x+y$ et $xy$ à précision $p^n$ 
également, simplement en effectuant ces opérations modulo $p^n$.

\subsubsection*{Les anneaux de séries $\Sk_\nu$ et leurs éléments}

Attardons-nous à présent sur les anneaux $\Sk_\nu$ et $\Sk_\nu[1/p]$. Comme 
précédemment, on est contraint de représenter leurs éléments par une 
approximation complétée par des données supplémentaires décrivant la 
nature de cette approximation. Au vu des résultats de \cite{carlub} 
(voir en particulier \S 4), nous choisissons de représenter en machine 
une série $f = \sum a_i u^i \in \Sk_\nu[1/p]$ par les trois données 
suivantes :
\begin{itemize}
\item la \emph{précision} : un entier strictement positif $N$ et
un $N$-uplet d'entiers relatifs $(N_0, \ldots, N_{N-1})$
\item l'\emph{approximation} : pour $i \in \{0, \ldots, N-1\}$, des
éléments $\tilde a_i$ (vivant dans un anneau exact approximant $K_0$)
tels que $\tilde a_i \equiv a_i \pmod{p^{N_i}}$.
\item la \emph{garantie} : un nombre rationnel $g$ tel que $v_p(a_i)
\geq g - \nu i$ pour tout $i \geq N$.
\end{itemize}
Dans la suite, on appellera \emph{représentation PAG} (pour
Précision-Approximation-Garantie) la représentation ci-dessus. Effectuer
les opérations arithmétiques élémentaires (addition, soustraction, 
multiplication) sur la représentation PAG ne pose pas de problème.
Il en va pareillement de la division euclidienne dans $\Sk_\nu[1/p]$,
comme cela est expliqué dans \cite{carlub}, \S 4.3.

\subsubsection*{Les $(\phi,N)$-modules filtrés}

Revenant à la définition, on voit qu'un $(\phi,N)$-module filtré 
admissible est la donnée de :
\begin{itemize}
\item un $K_0$-espace vectoriel de dimension finie $D$ ;
\item un opérateur semi-linéaire $\phi : D \to D$ ;
\item un opérateur linéaire $N : D \to D$ ;
\item une filtration $(Fil^h D_K)_{h \in \Z}$ sur l'espace
$D_K = K \otimes_{K_0} D$.
\end{itemize}
Les trois premières données ne sont pas difficiles à représenter :
en effet, se donner $D$ revient à s'en donner une base et les
opérateurs $\phi$ et $N$ sont alors donnés par leur matrice dans
la base retenue. En ce qui concerne la filtration, une possibilité
pour la décrire est la suivante : on se donne les sauts de la
filtrations $h_1 \geq h_2 \geq \cdots \geq h_d$ (où $d$ est la
dimension de $D$) ainsi que des vecteurs $f_1, \ldots, f_d$ tels
que, pour tout entier $h$, le cran $\Fil^h D_K$ soit engendré sur
$K$ par les vecteurs $f_i$ tels que $h_i \geq h$.

En résumé, un $(\phi,N)$-module filtré $D$ peut se décrire par
la donnée d'un quadruplet $(\code{Phi}, \code{N}, \code{H}, 
\code{F})$ où l'on a fixé une base de $D$ et 
\begin{itemize}
\item $\code{Phi}$ est la matrice de $\phi$ dans cette base ;
\item $\code{N}$ est la matrice de $N$ dans cette base ;
\item $\code{H}$ est le $d$-uplet d'entiers $h_1 \geq h_2 \geq
\cdots \geq h_d$ donnant les sauts de la filtration ;
\item $\code{F}$ est une matrice à coefficients dans $K$ dont les 
vecteurs colonne sont les vecteurs $f_1, \ldots, f_d$ définis précédemment 
(chaque vecteur étant exprimé par ses coordonnées dans la base qui a 
été fixée).
\end{itemize}
Dans la suite, nous utiliserons cette écriture pour représenter les 
$(\phi,N)$-modules filtrés sur machine. Toutefois,
conformément à ce qui a été expliqué précédemment, tous les coefficients 
des matrices \code{Phi}, \code{N} et \code{F} ne sont pas connus 
exactement mais à une certaine précision $p^n$ (ou $\pi^n$ dans le 
dernier cas) qu'il nous faudra préciser.

\subsubsection*{Les représentations semi-simples modulo $p$}

La représentation sur machine des $\F_p$-représentations semi-simples 
de $G_K$ passe par un théorème de classification. Soit $I$ 
le sous-groupe d'inertie ; on rappelle qu'il s'agit du sous-groupe 
distingué de $G_K$ formé des éléments qui agissent trivialement sur le 
corps résiduel. Le quotient $G_K/I$ s'identifie canoniquement au groupe 
de Galois absolu du corps résiduel $k$ qui, on le rappelle, est supposé
fini. Si $q$ désigne le cardinal de $k$, le groupe $G_K/I$ est un
procyclique et engendré par le Frobenius $\Frob_q : x \mapsto x^q$.
On rappelle aussi que $I$ admet un unique pro-$p$-Sylow, classiquement 
noté $I_p$ et appelé sous-groupe d'inertie sauvage. Le quotient $I/I_p = 
I_t$ est appelé le groupe d'inertie modérée ; il est isomorphe à 
$\varprojlim_n \mu_{p^n-1}(\bar k)$ où $\mu_{p^n-1}(\bar k)$ désigne le 
sous-groupe de $\bar k^\star$ des racines $(p^n-1)$-ièmes de l'unité
et s'identifie donc à $\F_{p^n}^\star$ si $\F_{p^n}$ désigne l'unique
sous-corps de cardinal $p^n$ de $\bar k$. On déduit, en particulier, 
de cet isomorphisme que $I_t$ est procyclique, c'est-à-dire 
topologiquement engendré par un unique générateur. Enfin, le quotient 
$G_K/I_p$ s'identifie au produit semi-direct $I_t \rtimes \Gal(\bar 
k/k)$ où le générateur $\Frob_q \in \Gal(\bar k/k)$ agit sur $I_t$ par 
élévation à la puissance $q$.
Le corps découpé par les sous-groupe distingué $I$ et $I_p$ est
l'extension maximale non ramifiée (resp. modérément ramifiée) de
$K$ et sera notée $K^\nr$ (resp. $K^\mr$). On a bien sûr $K \subset
K^\nr \subset K^\mr \subset \bar K$ ; de plus, le corps $K^\mr$
s'obtient à partir de $K^\nr$ en ajoutant les racines $(p^n-1)$-ièmes
de $\pi$.

\medskip

Si maintenant $V$ est une $\F_p$-représentation de $G_K$, un lemme 
classique sur les actions de groupes assure que l'ensemble $V^{I_p}$, 
constitué des éléments $x \in V$ tel que $gx = x$ pour tout $g \in I_p$, 
n'est pas réduit à $0$. Par ailleurs, du fait $I_p$ est un distingué 
dans $G_K$, on déduit que $V^{I_p}$ est stable par l'action de $G_K$ 
tout entier. Ainsi, si l'on suppose en outre que $V$ est irréductible, 
on a nécessairement $V^{I_p} = V$ ; autrement dit, le sous-groupe $I_p$ 
agit trivialement sur $V$, ce qui signifie encore que l'action de $G_K$ 
se factorise par $G_K/I_p \simeq I_t \rtimes \Gal(\bar k/k)$. 
Ainsi, pour décrire complètement une $\F_p$-représentation de $G_K$,
il suffit de se donner :
\begin{itemize}
\item la matrice (à coefficients dans $\F_p$) donnant l'action d'un
générateur topologique de $I_t$, et
\item la matrice (à coefficients dans $\F_p$) donnant l'action de
$\Frob_q$.
\end{itemize}
À vrai dire, on peut encore simplifier cette représentation car, étant 
donné sa forme particulièrement simple, on dispose d'une classification
très concrète des représentations irréductibles de $I_t$. Précisément,
en notant $\varpi_n$ une racine $(p^n-1)$-ième de $\pi$, on dispose
du caractère fondamental de Serre de niveau $n$ défini comme suit :
\begin{eqnarray*}
\omega_n : \quad I_t & \to & \mu_{p^n-1}(\bar K) \simeq 
\mu_{p^n-1}(\bar k) = \F_{p^n}^\star \\
g & \mapsto & \frac{g(\varpi_n)}{\varpi_n}.
\end{eqnarray*}
Ce caractère définit bien sûr une $\F_{p^n}$-représentation de dimension 
$1$ de $I_t$ mais, en considérant $\F_{p^n}$ comme un $\F_p$-espace 
vectoriel de dimension $n$, on s'aperçoit qu'il définit aussi une 
$\F_p$-représentation de $I_t$ de dimension $n$ qui sera notée dans la
suite $\F_{p^n}(\omega)$. Évidemment, il est possible de faire la même 
construction à partir des puissances $\omega^s$ du caractère $\omega$. 
On obtient, ce faisant, des $\F_p$-représentations de dimension $n$ que 
l'on note $\F_{p^n}(\omega^s)$. On montre que ces représentations sont 
irréductibles dès que la fraction $\frac s{p^n-1}$ ne peut s'écrire sous 
la forme $\frac{s'}{p^{n'}-1}$ pour un entier $n' < n$. Mieux encore, 
toute représentation irréductible de $I_t$ est de cette forme et on sait 
que deux représentations $\F_{p^n}(\omega^s)$ et $\F_{p^m} (\omega^t)$ 
sont isomorphes si, et seulement si $n = m$ et il existe un entier $a$ 
tel que $s \equiv p^a t \pmod {p^n-1}$.

Pour étendre cette classification aux représentations de $I_t \rtimes 
\Gal(\bar k/k)$, on introduit la définition classique suivante.

\begin{deftn}
Soit $F$ un corps de caractéristique $p$.
Un \emph{$\phi$-module} sur $(F, \Frob_q)$ est un $F$-espace vectoriel
$D$ de dimension finie muni d'une application additive $\phi : D \to
D$ vérifiant $\phi(ax) = a^q \phi(x)$ pour tout $a \in F$ et tout $x
\in D$.

Le $\phi$-module $D$ est dit \emph{étale} si $\phi$ est bijectif.
Il est dit \emph{simple} s'il n'admet aucun sous-$F$-espace vectoriel
strict stable par $\phi$.
\end{deftn}

On se donne, à présent, des entiers $s$ et $n$ comme précédemment ainsi 
qu'un $\phi$-module $D$ sur $(\F_{p^n}, \Frob_q)$ qui soit étale et simple. On 
pose $V = D$ et on munit cet espace de l'action $\F_p$-linéaire de $I_t 
\rtimes \Gal(\bar k/k)$ en faisant agir $I_t$ par multiplication par 
$\omega_n^s$ et $\Frob_q$ \emph{via} l'endomorphisme $\phi$.
(La relation $\Frob_q \cdot g \cdot \Frob_q^{-1} = g^q$ valable pour $g 
\in I_t$ montre que cette définition a bien un sens.)
En oubliant maintenant la structure de $\F_{p^n}$-espace vectoriel sur 
$V$ (pour ne retenir que celle de $\F_p$-espace vectoriel), on fait de
$V$ une $\F_p$-représentation de $I_t \rtimes \Gal(\bar k/k)$. On note
cette représentation $V(\omega_n^s, D)$.

\begin{prop}
Les représentations $V(\omega_n^s, D)$ définies ci-dessus sont
irréductibles et, réciproquement, toute $\F_p$-représentation
irréductible de $I_t \rtimes \Gal(\bar k/k)$ est de la forme
$V(\omega_n^s, D)$ pour certains paramètres $n$, $s$ et $D$.

Par ailleurs, $V(\omega_n^s, D) \simeq V(\omega_m^t, E)$
si, et seulement si $n = m$ et il existe un entier $a$ tel que
l'on ait simultanément $s' \equiv p^a t \pmod {p^n-1}$ et
l'existence d'un isomorphisme de $\phi$-modules
$E \simeq \F_{p^n} \otimes_{x \mapsto x^{p^a}, \F_{p^n}} D$.
\end{prop}

\begin{proof}
Exercice.
\end{proof}

On déduit de la proposition que, pour représenter une 
$\F_p$-représentation simple $V$ de $I_t \rtimes \Gal(\bar k/k)$, il 
suffit de se donner les paramètres $s$, $n$ et $D$ tels que $V = 
V(\omega_n^s, D)$. En vertu de la première réduction que l'on a 
expliquée, ceci s'applique également aux $\F_p$-représentation simples 
de $G_K$. Ainsi, pour décrire une représentation semi-simple de
ce groupe, il suffit de se donner une liste de paramètres $(s,n,D)$,
chacun de ces triplets correspondant à un facteur simple de la
représentation.

Il reste à expliquer comment on peut représenter concrètement un 
$\phi$-module $D$ sur $(F, \Frob_q)$ où $F$ est un corps fini de
caractéristique $p$. Une possibilité est de se donner la matrice de
l'opérateur $\phi$ dans une certaine base de $D$. Une autre solution
consiste à introduite l'anneau $F[X,\Frob_q]$ des polynômes 
tordus\footnote{Les éléments de $F[X,\Frob_q]$ sont les polynômes à 
coefficients dans $F$, l'addition sur $F[X,\Frob_q]$ est également
l'addition usuelle sur les polynômes, tandis que la multiplication 
découle la règle $X \cdot a = a^q \cdot X$. En particulier, l'anneau
$F[X,\Frob_q]$ n'est pas commutatif.} et à se rappeler que la donnée
d'un $\phi$-module étale simple sur $(F, \Frob_q)$ équivaut à la donnée 
d'une classe de similarité\footnote{Deux polynômes tordus $P$ et $Q$ 
dans $K[X,\Frob_q]$ sont dit similaires s'il existe $U$ et $V$ dans 
$K[X,\Frob_q]$ tels que $P$ et $V$ soient premiers entre eux à droite, 
$Q$ et $U$ soient premiers entre eux à gauche et $UP = QV$.} de 
polynômes tordus irréductibles dans $F[X,\Frob_q]$ (voir par exemple 
\cite{leborgne}, \S I). Ainsi, il est également possible de représenter 
le $\phi$-module par un représentant de la classe de similarité dans 
$F[X,\Frob_q]$ ; on obtient, comme ceci, une représentation compacte et 
également plus agréable à manipuler.

\subsubsection*{Les modules de Breuil-Kisin}

Soit $\nu$ un nombre rationnel positif ou nul. Par définition, un module 
de Breuil-Kisin sur l'anneau $\Sk_\nu$ est un $\Sk_\nu$-module libre 
$\Mk$ muni d'un endomorphisme semi-linéaire $\phi$ dont le conoyau du
linéarisé est annulé par une puissance de $E(u)$.

On peut ainsi raisonner comme dans le cas des $(\phi,N)$-modules filtrés 
et représenter un module de Breuil-Kisin sur $\Sk_\nu$ par la matrice 
$\code{PhiBK}$ ($\code{BK}$ pour Breuil-Kisin) de l'opérateur $\phi$ 
dans une certaine base de $\Mk$ fixée à l'avance. Cette matrice est
à coefficients dans $\Sk_\nu$ et la condition sur le conoyau de $\phi$ 
se traduit par l'existence d'une autre matrice $\code{PhiBK}'$ à
coefficients dans $\Sk_\nu$ telle que
$$\code{PhiSK}' \cdot \code{PhiSK} = E(u)^r$$
pour un certain entier $r$.

\subsection{Étape 1 : Des $(\phi,N)$-modules filtrés aux modules 
de Breuil-Kisin}
\label{subsec:etape1}

On se donne, dans ce numéro, un $(\phi,N)$-module filtré effectif 
admissible $D$ que l'on représente par un quadruplet $(\code{Phi}, 
\code{N}, \code{H}, \code{F})$ comme expliqué au \S \ref{subsec:repr}. 
Ceci sous-entend en particulier que l'on a fixé une base de $D$, que 
l'on notera parfois $(e_1, \ldots, e_d)$ avec $d = \dim_{K_0} D$. Dans la suite, on 
utilisera constamment de façon implicite cette base pour identifier $D$ 
à $K_0^d$. On note $L \subset D$ le réseau \og 
standard \fg, c'est-à-dire le sous-$\O_{K_0}$-module de $D$ engendré par 
les $e_i$.
Pour des questions de précision, on suppose, en outre, que la matrice 
$\code{F}$ est à coefficients dans l'anneau des entiers $\O_K$ et, de 
plus, qu'elle appartient $\GL_d(\O_K)$. Il est facile de vérifier qu'une 
matrice $\code{F}$ satisfaisant à cette hypothèse supplémentaire existe 
toujours ; de plus, s'il nous est donnée une matrice $\code{F}$ ne 
vérifiant pas cette hypothèse, il n'est pas difficile de la transformer 
--- à l'aide d'un algorithme de type \og pivot de Gauss \fg, voir lemme
\ref{lem:decomp}, page \pageref{lem:decomp} --- en une matrice 
$\code{F}$ convenable. Notons cependant que cette opération peut 
entraîner des pertes de précision $p$-adique.

\medskip

On considère un nombre entier $r$ tel que $\Fil^{r+1} D_K = 0$ ; 
par exemple, $r$ peut être pris égal au plus grand poids de Hodge-Tate 
de $V$. Par ailleurs, pour ne pas alourdir les notations lors de
l'étude de la perte de précision, nous faisons également les deux 
hypothèses simplificatrices suivantes : \emph{primo}, le polynôme
$E(u)$ est connu exactement et \emph{secundo}, les matrices 
$\code{Phi}$, $\code{N}$ et $\code{F}$ sont à coefficients entiers.

Notre objectif est de calculer le module de Breuil-Kisin $\Dk$ --- ou 
plutôt $\Dk_\nu = \Sk_\nu \otimes_\Sk \Dk$ pour $\nu > 0$ --- et nous 
suivons pour cela la méthode de Génestier et Lafforgue rappelée au \S 
\ref{subsec:GL}. Nous reprenons également les notations de ce numéro : 
on considère l'anneau $\hat \Sk = \varprojlim_n \Sk/E(u)^n$, on pose 
$\hat \Dk = \hat \Sk [\frac 1{E(u)}] \otimes_{\phi,K_0} D$ et on note 
$V_D$ l'unique structure de Hodge-Pink satisfaisant à la transversalité 
de Griffiths qui correspond à la filtration $\Fil^h D_K$ sur le 
$(\phi,N)$-module filtré $D$ (voir lemme \ref{lem:uniqueHP}).
D'après les hypothèses que l'on a faites, on a $\Fil^0 D_K = D_K$ et 
$\Fil^{r+1} D_K = 0$, ce qui se traduit par les inclusions $\hat \Sk 
\otimes_{K_0} D \subset V_D \subset E(u)^{-r} \hat \Sk \otimes_{K_0} D$. 
On pose $W_D = E(u)^r V_D$ ; on a alors $E(u)^r\hat \Sk \otimes_{K_0} D 
\subset W_D \subset \hat \Sk \otimes_{K_0} D$.

\subsubsection{Le calcul de $V_D$}
\label{subsec:VD}

La première étape consiste à calculer une famille explicite de 
générateurs de $V_D$. Pour ce faire, le lemme suivant nous sera fort
utile.

\begin{lemme}
\label{lem:hatNzero}
Le morphisme de réduction modulo $E(u)$ induit un isomorphisme :
$$(\hat \Sk \otimes_{\phi, K_0} D)^{\hat N = 0} \to D^\phi_K.$$
\end{lemme}

\begin{proof}
Il s'agit de montrer que tout $w \in D^\phi_K$ se relève de manière
unique en un élément $\hat w \in \hat \Sk \otimes_{\phi, K_0} D$ 
vérifiant $\hat N(\hat w) = 0$. 
On raisonne par approximations successives : on va démontrer par
récurrence que, pour tout entier $i \geq 0$, il existe $\hat w^{(i)} 
\in \hat \Sk \otimes_{\phi, K_0} D$ tel que $\hat w_i \equiv w \pmod
{E(u)}$ et $\hat N(\hat w^{(i)}) \equiv 0 \pmod{E(u)^i}$ et, de plus,
deux $\hat w^{(i)}$ solutions de deux congruences précédentes sont 
congrus entre eux modulo $E(u)^i$.
L'assertion est clairement vraie pour $i = 1$. On la suppose à présent
pour un certain $i$ et on cherche à la démontrer pour $i+1$. Cherchant 
$\hat w^{(i+1)}$ sous la forme $\hat w^{(i)} + E(u)^i x$, on est amené
à résoudre la congruence
$$\hat N(\hat w^{(i)}) + i u E(u)^{i-1} \: \frac {d E(u)}{du} \cdot x 
\equiv 0 \pmod{E(u)^i}.$$
Or, on sait, par hypothèse de récurrence, que $\hat N(\hat w^{(i)})$ est
multiple de $E(u)^{i-1}$. On peut donc écrire $\hat N(\hat w^{(i)}) =
E(u)^{i-1} y$ et la congruence que l'on cherche à résoudre se réduit
alors à :
$$i u \: \frac {d E(u)}{du} \cdot x + y \equiv 0 \pmod{E(u)}.$$
Étant donné que $\hat \Sk / E(u) \hat \Sk \simeq K$ est un corps dans
lequel le produit $i u \frac {d E(u)}{du}$ est non nul, la congruence 
ci-dessus admet une solution modulo $E(u)$. Ceci termine la 
récurrence et démontre le lemme.
\end{proof}

Il est à noter que la démonstration que l'on vient de donner est 
entièrement constructive (bien sûr, si l'on se limite à calculer les 
$\hat w^{(i)}$ pour un entier $i$ fini) et peut-être transformée 
aisément en algorithme (voir algorithme \ref{algo:relHP}).
\begin{algorithm}
  \SetKwInOut{Input}{Entrée}
  \SetKwInOut{Output}{Sortie}

  \Input{un vecteur $w \in D_K^\phi$}
  \Output{l'unique vecteur $\hat w \in (\hat \Sk \otimes_{\phi,K_0} D)
  ^{\hat N = 0}$ relevant $w$ calculé modulo $E(u)^i$}

  \BlankLine
  $A$ $\leftarrow$ $u \: \frac{dE(u)}{du}$, \;
  $B$ $\leftarrow$ inverse de $A$ modulo $E(u)$\;

  $\hat w$ $\leftarrow$ relevé de $w$ dans $K_0[u]$\label{lgn:relw}\;
  $y$ $\leftarrow$ $p \code{N} \cdot \hat w + u \: \frac{d\hat w}{du}$\;
  \For{j allant de $1$ à $i-1$}{
    $x$ $\leftarrow$ $- j^{-1} \: B \cdot y \mod E(u)$\label{lgn:divj}\;
    $\hat w$ $\leftarrow$ $\hat w + E(u)^i \: x$\;
    $y$ $\leftarrow$ $p \code{N} \cdot x + u \: \frac{dx}{du} + 
      \frac{Y + jAx}{E(u)}$\label{lgn:mulB}\;
  }
  \Return $\hat w$\;
\caption{\sc 
ReleveHP$(w, \code{N}, i)$\label{algo:relHP}}
\end{algorithm}
Il est, en outre, possible d'estimer les pertes de précision engendrées
par cet algorithme.

\begin{lemme}
\label{lem:precrelHP}
On suppose que le polynôme $E(u)$ est connu à précision arbitrairement
grande.
Si le vecteur $w$ est à coefficients dans $\O_K$ et que $w$ ainsi que 
$\code{N}$ sont connus à précision $O(p^N)$, alors l'algorithme 
\ref{algo:relHP} appelé sur l'entrée $(w, \code{N}, i)$ renvoie un 
vecteur $\hat w$ connu à précision $O(p^{N-\rho_1})$ avec
$$\rho_1 = (i-2) \cdot \left\lceil \frac 1 e + \val(\df_{K/K_0}) 
\right\rceil + \frac {i-1}{p-1}$$
où $\df_{K/K_0}$ désigne la différente de $K/K_0$ et $\lceil x \rceil$
désigne la partie entière supérieure du réel $x$.
\end{lemme}

\begin{proof}
D'après l'hypothèse sur $E(u)$, il est possible de calculer les 
polynômes $A$ et $B$ avec une précision arbitrairement grande. Comme, en 
outre, $A$ est à coefficients dans $\O_{K_0}$, les multiplications par
$A$ n'engendrent pas de perte de précision. De la même façon, étant
donné que $E(u)$ est unitaire, la division par $E(u)$ n'entraîne pas
non plus de perte de précision. Le polynôme $B$, quant à lui, n'a pas de 
raison d'être à coefficients dans $\O_{K_0}$ ; cependant, on connait la 
valuation de son image dans $K$ : c'est l'opposé de $\frac 1 e + 
\val(\df_{K/K_0})$. Ainsi, si l'on note $\delta$ la partie entière
supérieure de $\frac 1 e + \val(\df_{K/K_0})$, on peut choisir $B$ de
façon à ce que $v_0(B) \geq -\delta$. Ainsi les multiplications par 
$B$ font chuter la précision de $p^\delta$.

Lors de l'exécution de l'algorithme \ref{algo:relHP}, les seules pertes 
de précision interviennent :
\begin{itemize}
\item à la ligne \ref{lgn:divj} lors de la multiplication par $j^{-1}$,
\item à la ligne \ref{lgn:mulB} lors de la multiplication par $X$ et
de la multiplication par $B$.
\end{itemize}
Ainsi, si l'on note $v_{x,j}$, $v_{y,j}$ et $v_{\hat w,j}$ (resp. 
$p_{x,j}$, $p_{y,j}$ et $p_{\hat w,j}$) les valuations respectives 
(resp. les exposants des précisions respectives) des variables $x$, $y$ 
et $\hat w$ à la sortie de la $j$-ième itération de la boucle, on a les 
formules récurrentes suivantes :
\begin{eqnarray}
v_{x,j} & \geq & v_{y,j-1} - \val(j)\label{eq:vxj} \\
v_{y,j} & \geq & \min(v_{y,j-1}, v_{x,j}, v_{x,j} + \delta + \val(j))\label{eq:vyj} \\
v_{\hat w,j} & \geq & \min(v_{\hat w,j-1}, v_{x,j}) \\
p_{x,j} & \geq & p_{y,j-1} - \val(j) \label{eq:pxj}\\
p_{y,j} & \geq & \min(p_{y,j-1}, N + v_{x,j}, p_{x,j}, p_{x,j} + \delta + \val(B))\label{eq:pyj} \\
p_{\hat w,j} & \geq & \min(p_{\hat w,j-1}, p_{x,j}).
\end{eqnarray}
Par ailleurs, à la ligne \ref{lgn:relw}, il est certainement possible
de choisir un relevé $\hat w$ connu à précision $O(p^N)$ (simplement en
écrivant $w$ comme un polynôme en $\pi$). Ainsi, si l'on convient que
$v_{\hat w, 0}$ et $p_{\hat w,0}$ désignent la valuation et la précision
de $\hat w$ avant l'entrée dans la boucle, on obtient les conditions
initiales $v_{\hat w, 0} \geq 0$, $p_{\hat w,0} \geq N$. De même, on
a $v_{y,1} \geq 0$ et $p_{y,1} \geq N$.
En remplaçant dans l'inégalité \eqref{eq:vyj}, la quantité $v_{x,j}$ par 
le minorant donné par \ref{eq:vxj}, on obtient la formule de 
récurrence :
$$v_{y,j} \geq \min(v_{y,j-1} - \val(j), 
v_{y,j-1} - \delta) \geq v_{y,j-1} - \val(j) - \delta$$
qui entraîne immédiatement la formule close
$v_{y,j} \geq - j \cdot \delta - \val(j!) \geq - j \cdot \delta - 
\frac j {p-1}$. On en déduit que $v_{x,j} \geq - (j-1) \cdot \delta - 
\frac j {p-1}$ et, par suite, que le même minorant vaut pour $v_{\hat
w,j}$. De la même façon, en combinant \eqref{eq:pxj} et \eqref{eq:pyj},
on obtient $p_{y,j} \geq N - j \delta - \frac j{p-1}$, $p_{x,j} \geq N - 
j \delta - \frac j{p-1}$ et $p_{\hat w,j} \geq N - (j-1) \delta - \frac 
j{p-1}$. Comme la boucle est exécutée $j-1$ fois, on a bien le résultat
souhaité.
\end{proof}

On revient à présent au problème de calculer la structure de Hodge-Pink 
$W_D$ à partir de la donnée du quadruplet $(\code{Phi}, \code{N}, 
\code{H}, \code{F})$. On pose $W = \code{Phi}^{-1} \cdot \code{F}$. Les 
vecteurs colonne de $W$ vus comme éléments de $K \otimes_{\phi, K_0} D$ 
(et exprimées dans la base $(1 \otimes e_1, \ldots, 1 \otimes e_d)$) 
sont alors égaux aux $\phi_K^{-1}(f_i)$ où les $f_i$ sont les vecteurs 
colonnes de $\code{F}$. On rappelle le lemme classique suivant :

\begin{lemme}
\label{lem:decomp}
Toute matrice $M \in M_d(K)$ se décompose sous la forme $M = M' U$
avec $M' \in \GL_d(\O_K)$ et $U$ triangulaire supérieure.
\end{lemme}

\begin{proof}
Il s'agit de démontrer qu'en effectuant des opérations élémentaires
inversibles dans $\O_K$ sur les lignes de $M$, on peut obtenir une
matrice triangulaire supérieure. Voici comment on peut procéder : 
(1)~on sélectionne un élément de valuation minimale sur la première
colonne, (2)~on déplace cet élément sur la première ligne en
permutant les lignes correspondantes, (3)~on utilise cet élément
comme pivot pour annuler, à l'aide d'opérations élémentaires sur
les lignes de $M$, tous les autres coefficients de la première
colonne de $M$, (4)~on recommence tout le processus précédent à
partir de la sous-matrice de $M$ obtenue en retirant la première
ligne et la première colonne.
\end{proof}

Soit $W = W' U$ une décomposition satisfaisant aux conditions du lemme 
précédent. Sachant que $U$ est triangulaire supérieure et que les $h_i$ 
sont rangés par ordre croissant, on déduit qu'en notant $(w'_1, \ldots, 
w'_d)$ les vecteurs colonne de $W'$ considérés comme éléments de $K 
\otimes_{\phi, K_0} D$, pour tout $h \geq 0$, l'espace $\phi_K^{-1}( 
\Fil^h D_K)$ est engendré par les vecteurs $w'_i$ pour les indices $i$ 
tels que $h_i \geq h$. Pour tout $i$, on note $\hat w'_i$ l'unique élément
de $(\hat \Sk \otimes_{\phi, K_0} D)^{\hat N = 0}$ relevant $w'_i$ (voir
lemme \ref{lem:hatNzero} ci-dessus).

\begin{prop}
\label{prop:genwi}
Avec les notations précédentes, $V_D$ est engendré comme $\hat 
\Sk$-module par les vecteurs $E(u)^{-h_i} \hat w'_i$ ($1 \leq i \leq 
d$).
\end{prop}

\begin{proof}
Si l'on note $V'_D$ la structure de Hodge-Pink engendrée par les 
vecteurs $E(u)^{-h_i} \hat w'_i$, il suffit de vérifier que :
\begin{enumerate}[(i)]
\item $V'_D$ vérifie la transversalité de Griffiths, et
\item la filtration de $D_K$ associée à $V'_D$ s'identifie à la
filtration $\Fil^h D_K$ du $(\phi,N)$-module filtré $D$.
\end{enumerate}
Le premier point découle directement de la condition $\hat N(\hat 
w'_i) = 0$ tandis que le second est une conséquence immédiate des
congruences $\hat w'_i \equiv w'_i =  \phi_K^{-1}(v_i) \pmod{E(u)}$.
\end{proof}

Étant donné que l'on sait par avance que $\hat \Sk \otimes_{\phi, K_0} D 
\subset V_D$, la proposition \ref{prop:genwi} ci-dessus vaut encore si, 
pour chaque $i$, on remplace $\hat w_i$ par un élément qui lui est 
congru modulo $E(u)^{h_i}$. Ainsi, en reprenant les notations de la 
démonstration du lemme \ref{lem:hatNzero}, les $E(u)^{-h_i} \hat 
w_i^{(h_i)}$ forment aussi une famille de générateurs de $V_D$. Or, ces 
$\hat w_i^{(h_i)}$ ont l'énorme avantage de pouvoir être calculé 
efficacement par l'algorithme \ref{algo:relHP}. En plus de cela, les 
vecteurs $E(u)^{r-h_i} \hat w_i^{(h_i)}$ --- qui engendrent donc $W_D = 
E(u)^r V_D$ --- possèdent le second avantage d'être éléments de 
$\E^+ \otimes_{\phi,K_0} D$. Ainsi, non seulement ils forment une $\hat 
\Sk$ base de $W_D$ mais également une $\E^+$-base de l'intersection $W_D 
\cap (\E^+ \otimes_{\phi,K_0} D)$.

L'algorithme \ref{algo:HP} récapitule la construction d'une famille
génératrice de $V_D$ que nous venons de présenter.
\begin{algorithm}
  \SetKwInOut{Input}{Entrée}
  \SetKwInOut{Output}{Sortie}

  \Input{Un $(\phi,N)$-module filtré $D$}
  \Output{Une matrice $\hat W'$ telle que les vecteurs colonne de $\hat W' 
  \cdot \Diag(E(u)^{-h_1}, \ldots, E(u)^{-h_d})$ engendrent $V_D$}

  \BlankLine
  $W$ $\leftarrow$ $\code{Phi}^{-1} \cdot \code{F}$\;
  \textbf{écrire} $W = W' U$ (\emph{cf} lemme \ref{lem:decomp})\;
  \lFor{i allant de $1$ à $d$}{
    $\hat w'_i$ $\leftarrow$ $\textsc{ReleveHP}(W'[\cdot,i], \code{N}, \code{H}[i])$\;
  }
  \Return{la matrice dont les vecteurs colonne sont $\hat w'_1, 
    \ldots, \hat w'_d$\;}
\caption{\sc 
HodgePink$(\code{Phi}, \code{N}, \code{H}, \code{F})$\label{algo:HP}}
\end{algorithm}
Au niveau de la précision, l'admissibilité de $D$, combinée au fait que 
tous les $h_i$ sont dans $\{0, \ldots, r\}$, implique que tous les 
diviseurs élémentaires de $\code{Phi}$ divisent $p^r$. À partir de là, 
un examen de la méthode de calcul de la matrice $W'$ (voir démonstration 
du lemme \ref{lem:decomp}) montre que les pertes de précision pour le
calcul de $W'$ sont majorées par $p^{dr}$ ; autrement dit, si $W$ est
connu à précision $O(p^N)$, la matrice $W'$ est au pire connue avec
une précision $O(p^{N-dr})$. En combinant ceci avec le résultat du
lemme \ref{lem:decomp}, on obtient une formule explicite pour les pertes 
de précision totales de l'algorithme \ref{algo:HP}.

\subsubsection{Le calcul de $\beta_n$}
\label{subsec:betan}

On rappelle que l'on a défini au \S \ref{subsec:GL} une suite $\beta_n$ 
de sous-$\E^+$-modules de $\E^+ \otimes_{K_0} D$ (voir formule 
\ref{eq:betan}) et que ceux-ci sont reliés au module de Breuil-Kisin 
$\Dk_\nu$ que l'on souhaite calculer grâce au théorème \ref{theo:GL} : 
on a un isomorphisme canonique
$$\Dk_{1/(ep^n)} \simeq \E^+_{1/(ep^n)} \otimes_{\E^+} \beta_n$$
et l'action de $\phi$ sur $\Dk_{1/(ep^n)}$ correspond \emph{via} 
cette identification à l'action de $(\frac{E(u)}{E(0)})^r \phi
\otimes \phi$ sur le membre de droite. Dans ce paragraphe, nous
expliquons comment calculer des générateurs de $\beta_n$ pour un
entier $n$ que l'on se donne.

\medskip

On remarque pour commencer que les $\beta_n$ ne sont pas modifiés si, 
dans la formule \eqref{eq:betan}, on replace $W_D$ par son intersection 
avec $\E^+ \otimes_{\phi, K_0} D$. Or, d'après les résultats du \S 
\ref{subsec:VD}, une base de cette intersection est donnée par les 
vecteurs colonne de la matrice produit $\hat W' \Delta_1$ où $\hat W'$ 
est la matrice calculée par l'algorithme \ref{algo:HP} et $\Delta_1$ est 
la matrice diagonale suivante :
$$\Delta_1 = \left( \begin{matrix}
\lambda_1^{r-h_1} & & \\
& \ddots & \\
& & \lambda_1^{r-h_d}
\end{matrix} \right)
\quad \text{avec} \quad \lambda_1 = \frac{E(u)}{E(0)}.$$
À partir de là, on s'aperçoit qu'il est possible de calculer les 
$\beta_n$ en calculant des intersections de $\E^+$-modules complètement 
explicites. Comme $\E^+$ est un anneau euclidien, cela est possible en
utilisant les algorithmes standard de manipulation de modules sur les
anneaux euclidiens (qui reposent essentiellement sur l'existence d'une
forme normale d'Hermite). Toutefois, ces méthodes peuvent s'avérer 
lentes et très instables. Dans la suite, nous allons présenter une autre 
approche qui repose sur l'écriture d'une formule quasiment explicite pour 
$\beta_n$.

Pour tout entier $m \geq 1$, on pose $\lambda_m = \lambda_1 
\phi(\lambda_1) \cdots \phi^{m-1}(\lambda_1)$ et, de même que l'on a 
défini $\Delta_1$, on note $\Delta_m$ la matrice diagonale dont les 
termes diagonaux sont $\lambda_m^{r-h_1}, \ldots, \lambda_m^{r-h_d}$. On 
convient également que $\lambda_0 = 1$ et que $\Delta_0$ est la matrice 
identité de taille $d$.

\begin{lemme}
\label{lem:betan}
Soit $\hat W'$ la matrice calculée par l'algorithme \ref{algo:HP}. Soient $n$ 
un entier strictement positif et $X$ une matrice à coefficients dans 
$\E^+$ telle que pour tout entier $m \in \{0, \ldots, n-1\}$ :
$$X \equiv \code{Phi} \cdot \phi(\code{Phi}) \cdots \phi^m(\code{Phi})
\cdot \phi^m(\hat W' \Delta_1) \cdot P_m \pmod{\phi^m(E(u)^r)}$$
pour une certaine matrice $P_m$ à coefficients dans $\E^+$, inversible
modulo $\phi^m(E(u))$.
Alors $\beta_n$ est engendré par les vecteurs $\lambda_n^r e_i$ ($1 \leq 
i \leq d$) et les vecteurs colonne de la matrice $X$.
\end{lemme}

\begin{proof}
Pour simplifier les écritures, on pose $M = \hat W' \Delta_1$.
On raisonne par récurrence sur $n$. Lorsque $n = 1$, l'espace $\beta_1$ 
est, par définition, égal à $\frac {\phi}{p^r} (W_D)$ et est donc 
engendré par les vecteurs colonne de la matrice produit $\code{Phi} 
\cdot M$. Par ailleurs, la matrice $P_0$ étant inversible modulo $E(u)$, 
elle l'est aussi modulo $E(u)^r$. On conclut la démonstration dans le 
cas $n = 1$ en remarquant que multiplier à droite par une matrice 
inversible revient à faire une combinaison linéaire \og inversible \fg\ 
des colonnes et donc ne change pas l'espace engendré par les colonnes.

On suppose maintenant que le lemme est vrai pour l'entier $n$.
Du fait que $E(u)^r \E^+ \otimes_{\phi,K_0} D \subset W_D$, 
on déduit aisément que $\lambda_{n+1}^r \E^+ \otimes_{K_0} D \subset 
\beta_{n+1}$. Ainsi, il suffit de démontrer que les colonnes de la
matrice $X$ du lemme engendrent l'image de $\beta_{n+1}$ dans le quotient $(\E^+ / 
\lambda_{n+1}^r \E^+) \otimes_{K_0} D$. Par le lemme chinois, ce dernier 
est isomorphe à la somme directe $(A_0 \otimes_{K_0} D) \oplus \cdots 
\oplus (A_n \otimes_{K_0} D)$ où on a noté $A_m = 
\frac{\E^+}{\phi^m(E(u)^r) \E^+}$. Or, d'après l'hypothèse de 
récurrence, si $0 \leq m < n$, l'image de $\beta_n$ dans $A_m 
\otimes_{K_0} D$ est engendrée par les vecteurs colonne 
de $\code{Phi} \cdot \phi(\code{Phi}) \cdots \phi^m(\code{Phi}) \cdot 
\phi^m(M)$ (on peut à nouveau supprimer la multiplication par $P_m$, 
cela ne modifie pas l'espace engendré). Ainsi, après un twist, l'image de 
$\E^+ \otimes_{\phi, \E^+} \beta_n$ dans $A_{m+1} \otimes_{K_0} D$ est 
engendrée par les vecteurs colonne de la matrice $\phi(\code{Phi}) 
\cdots \phi^{m+1}(\code{Phi}) \cdot \phi^{m+1}(M)$. Il en est, en 
réalité, de même de l'image de l'intersection $\E^+ \otimes_{\phi, \E^+} 
\beta_n \cap W_D$ étant donné que $E(u)^r$ est inversible dans $A_{m+1}$ 
(puisque $m+1 > 0$). Par ailleurs, 
l'image de cette intersection dans $A_0 \otimes_{K_0} D$ est engendrée 
par les vecteurs colonne de $M$ puisque $\E^+ \otimes_{\phi, 
\E^+} \beta_n$ contient $\phi(\lambda_n) \E^+ \otimes_{\phi,K_0} D$ et 
que $\phi(\lambda_n)$ est inversible dans $A_0$. Il résulte de cela que, 
pour tout $m \in \{0, \ldots, n\}$, l'image de $\beta_{n+1}$ dans $A_m 
\otimes_{K_0} D$ est engendrée par les vecteurs colonne de 
$\code{Phi} \cdot \phi(\code{Phi}) \cdots \phi^m(\code{Phi}) \cdot 
\phi^m(M)$. L'assertion du lemme au rang $n+1$ en découle.
\end{proof}

À partir de maintenant, on fixe l'entier $n$. Si l'on se donne n'importe 
quelle famille de matrices $P_m$, il est possible de calculer une 
matrice $X$ vérifiant les conditions du même lemme par des applications 
successives du lemme chinois. Étant donné que $\E^+$ est un anneau 
euclidien, calculer une base de $\beta_n$ peut alors se faire en 
réduisant sous forme d'Hermite la matrice par blocs $(\begin{matrix} X_n 
& \lambda_n^r I\end{matrix})$.
Toutefois, le fait de pouvoir choisir librement les $P_m$ permet de 
limiter les calculs, comme on se propose de l'expliquer maintenant.

Pour tout entier $m$, on considère l'application $\phi_K^{(m)} : K 
\otimes_{\phi^m, K_0} D \to K \otimes_{K_0} D = D_K$ obtenue en 
linéarisant $\phi^m$ ; c'est une application $K$-linéaire bijective. En 
désignant par $f_1, \ldots, f_d$ les vecteurs colonne de $\code{F}$, on 
appelle $W_m$ la matrice dont la $i$-ième colonne est l'unique 
antécédent de $f_i$ par $\phi_K^{(m)}$. Un calcul simple montre que l'on 
a l'expression suivante :
\begin{equation}
\label{eq:Wm}
W_m = \phi^{m-1}(\code{Phi}^{-1}) \cdots \phi(\code{Phi}^{-1}) 
\cdot \code{Phi}^{-1} \cdot \code{F}.
\end{equation}
On décompose, à présent, $W_m$ sous la forme $W_m = W'_m U_m$ où $W'_m 
\in \GL_d(\O_K)$ et $U_m$ est triangulaire supérieure (voir lemme 
\ref{lem:decomp}). Pour tout $j \leq d$, les $j$ premiers vecteurs
colonne de $W'_m$ forment une base de l'image réciproque par 
$\phi_K^{(m)}$ de $\Fil^h D_K$ lorsque $h_j \geq h > h_{j+1}$. 

On considère l'espace $\hat \Sk \otimes_{\phi^m, K_0} D$. Comme 
précédemment, on dispose d'une application $\hat \phi^{(m)} : \hat \Sk 
\otimes_{\phi^m, K_0} D \to \hat \Sk \otimes_{K_0} D$ obtenue en 
linéarisant $\phi^m$. D'autre part, $\hat \Sk \otimes_{\phi^m, K_0} D$ 
est également muni de la dérivation $\hat N_m = p^m \otimes N + u \: 
\frac d{du} \otimes \id$. De la relation $N \phi = p \phi N$, on déduit 
$\hat \phi^{(m)} \circ \hat N_m = \hat N_0 \circ \hat \phi^{(m)}$. Par 
ailleurs, en copiant la démonstration du lemme \ref{lem:hatNzero}, on
montre que le morphisme de réduction modulo $E(u)$ induit une
bijection :
\begin{equation}
\label{eq:hatNm}
(\hat \Sk \otimes_{\phi^m, K_0} D)^{\hat N_m = 0} 
\stackrel{\sim}{\longrightarrow} K \otimes_{\phi^m, K_0} D.
\end{equation}
En outre, l'algorithme \ref{algo:relHP} s'adapte sans difficulté à cette 
nouvelle situation. Ainsi, on sait calculer une matrice à coefficients 
dans $K_0[u]$ dont le $i$-ième vecteur colonne est congru modulo 
$E(u)^{h_i}$ à l'unique antécédent dans $ (\hat \Sk \otimes_{\phi^m, 
K_0} D)^{\hat N_m = 0}$ du $i$-ième vecteur colonne de $W'_m$. On note
$\hat W'_m$ cette matrice. Comme précédemment, on pose $W' = W'_1$,
$U = U_1$ et $\hat W' = \hat W'_1$.

\begin{lemme}
\label{lem:congrWm}
On a la congruence :
$$\hat W' \cdot \Delta_1 \equiv 
\phi(\code{Phi}) \cdots \phi^{m-1}(\code{Phi})
\cdot \hat W'_m \cdot U_m \cdot U^{-1} \cdot \Delta_1 \pmod{E(u)^r}.$$
\end{lemme}

\begin{proof}
On note $\hat W_m$ (resp. $\hat W$) la matrice à coefficients dans $\Sk$ 
dont les colonnes relèvent les colonnes de $W_m$ (resp. de $W$) par la 
bijection \eqref{eq:hatNm} (resp. avec $m = 1$). On considère la 
composée suivante :
$$(\hat \Sk \otimes_{\phi^m, K_0} D)^{\hat N_m = 0} 
\stackrel{f}{\longrightarrow}
(\hat \Sk \otimes_{\phi, K_0} D)^{\hat N = 0} 
\stackrel{\sim}{\longrightarrow} D_K^\phi = K \otimes_{\phi, K_0} D$$
où la première flèche, notée $f$, est l'application $\hat \phi^{(m-1)}$
twistée par $\phi$. La matrice de $f$ dans les bases canoniques est
$\phi(\code{Phi}) \cdots \phi^{m-1}(\code{Phi})$.
Soit $w_i$ l'image réciproque de $f_i \in D_K$ dans $D_K^\phi$. Les 
vecteurs colonne de $\hat W_m$ (resp. $\hat W$) correspondent alors
aux éléments de $(\hat \Sk \otimes_{\phi^m, K_0} D)^{\hat N_m = 0}$
(resp. $(\hat \Sk \otimes_{\phi, K_0} D)^{\hat N = 0}$) qui ont pour
image dans $D_K$ les vecteurs $w_i$. On en déduit que 
$\hat W = \phi(\code{Phi}) \cdots \phi^{m-1}(\code{Phi}) \cdot
\hat W_m$.
D'autre part, comme la bijection \eqref{eq:hatNm} est $K$-linéaire, on 
a $\hat W_m \cdot \Delta_1 \equiv \hat W'_m \cdot U_m \cdot \Delta_1 
\pmod{E(u)^r}$ et, pareillement, $\hat W \cdot \Delta_1 \equiv \hat
W' \cdot U \cdot \Delta_1 \pmod{E(u)^r}$. Des congruences et égalité 
précédentes, on déduit :
$$\hat W' \cdot U \cdot \Delta_1 \equiv \phi(\code{Phi}) \cdots 
\phi^{m-1}(\code{Phi}) \cdot \hat W'_m \cdot U_m \cdot \Delta_1
\pmod {E(u)^r}.$$
On conclut en multipliant à droite la congruence précédente par 
$\Delta_1^{-1} U^{-1} \Delta_1$, après avoir remarqué que cette matrice 
produit est à coefficients dans $\Sk$ parce qu'elle s'obtient à partir 
de $U^{-1}$ qui est triangulaire supérieure en multipliant le 
coefficient en position $(i,j)$ par $\lambda_1^{h_i - h_j}$ et que l'on 
a bien $h_i \geq h_j$ dès que $i \leq j$.
\end{proof}

\begin{theo}[Décompositions PLU simultanées]
\label{theo:LU}
On considère un entier $v \geq \log_p(2nd)$.
Alors, une matrice aléatoire $\omega$ à coefficients dans $\Z_p$ 
vérifie les conditions suivantes avec probabilité $\geq \frac 1 2$ :
\begin{enumerate}[(i)]
\item considérée comme matrice à coefficients dans $\Q_p$, la matrice 
$\omega$ est inversible et on a $\omega^{-1} \in p^{-v} M_d(\Z_p)$.
\item pour tout entier $m \in \{1, \ldots, n\}$, il existe des 
matrices $L_m$ et $V_m$ (uniquement déterminées modulo $E(u)^r$) qui 
sont respectivement triangulaire inférieure unipotente et triangulaire 
supérieure et qui vérifient la congruence :
\begin{equation} 
\label{eq:LU} 
\omega \cdot \hat W'_m \equiv L_m V_m \pmod{E(u)^r}
\end{equation}
de plus, les matrices $L_m$ et $V_m$ sont uniquement déterminées si
on impose de plus que tous leurs coefficients sont des polynômes de 
degré $< er$.
\item \label{item:valLm}
pour tout $m$, l'image de $L_m$ dans $M_d(\Sk/E(u)^r) \simeq 
M_d(K[u_\pi] / u_\pi^r)$ appartient à $\pi^{-e \rho_2} M_d(\O_K[u_\pi]/
u_\pi^r)$ avec
$$\textstyle \rho_2 = r(\frac 1 e + v + \val(\df_{K/K_0})).$$
\item si les $\hat W'_m$ sont connus à précision $O(p^N)$, alors
on peut calculer les images de $L_m$ dans $M_d(K[u_\pi] / u_\pi^r)$
à précision $O(p^{N - 2 \lceil \rho_2 \rceil})$.
\end{enumerate}
\end{theo}

\begin{proof}
Pour $r = 1$, il s'agit du théorème 2.12 de \cite{caruso-LU} étant donné 
que, par construction, l'image de $\hat W'_m$ dans $M_d(\Sk/E(u)) \simeq 
M_d(K)$ est inversible dans l'anneau $M_d(\O_K)$.

Pour $r > 1$, on montre, en reprenant la démonstration du lemme 
\ref{lem:precrelHP}, que, pour tout $j < r$, l'image de $\hat W'_m$ 
dans $M_d(K[u_\pi]/u_\pi^j)$ appartient à $\pi^{-j w} \cdot M_d
(\O_K[u_\pi]/u_\pi^j)$ avec $w = \frac 1 e + \val(\df_{K/K_0})$.
Ainsi l'image de $\hat W'_m$ dans $M_d(K[u_\pi]/u_\pi^r)$ appartient
à $M_d(\O_K[X]/X^r)$ où $X = \frac{u_\pi}{\pi^w}$.
Donc les réductions modulo $E(u)^r$ de tous les mineurs de $\omega 
\cdot \hat W'_m$ sont dans $\O_K[X]/X^r$. Or, par ailleurs, 
on sait que :
\begin{enumerate}[(1)]
\item les coefficients de $L_m$ s'obtiennent comme quotient de deux tels 
mineurs, le dénominateur étant toujours un mineur principal (voir par 
exemple \S 1.1.1 de \cite{caruso-LU}) ;
\item qu'avec probabilité $\geq \frac 1 2$, on a $\omega^{-1} \in p^{-v} 
M_d(\Z_p)$ et les images dans $\Sk/E(u) \simeq K$ des mineurs principaux 
de $\omega \hat W'_m$ est divisible par $\pi^v$ (cela résulte de la 
démonstration du théorème 2.12 de \cite{caruso-LU}).
\end{enumerate}
On conclut en remarquant qu'un élément de $\O_K[X]/X^r$ dont le terme 
constant est divisible par $\pi^v$ a pour inverse un élément de 
$\pi^{-v} \O_K[Y]/Y^r$ avec $Y = \frac X{\pi^v}$.
\end{proof}

\begin{rem}
\label{rem:rcard}
En utilisant non pas une décomposition LU mais une décomposition LU par 
blocs (en regroupant entre eux les $h_i$ qui sont égaux), on peut 
remplacer la borne $\log_q(2 n d)$ qui apparait dans le théorème 
\ref{theo:LU} par $\log_q(2 n r_\card)$ où $r_\card$ est le cardinal de 
l'ensemble $\{h_1, \ldots, h_d\}$ (c'est-à-dire le nombre de poids de 
Hodge-Tate de la représentation semi-stable $V$, comptés \emph{sans} 
multiplicité). On a évidemment $r_\card \leq \min(r,d)$. On renvoie au 
théorème 2.12 de \cite{caruso-LU} pour plus de précisions à ce sujet.
\end{rem}

\medskip

À partir de maintenant, on fixe des matrices $\omega$, $L_m$ et $V_m$
vérifiant les conditions du théorème précédent.
On note $t_0 = 1$ et pour tout $m \geq 1$, on pose $t_m = t_{m-1} \cdot 
\phi^m(\lambda_1^r) \cdot \big(\phi^m(\lambda_1^r)^{-1} \text{ mod } 
E(u)^r\big)$ où, si $s \in \E^+$, la notation $s^{-1} \text{ mod } 
E(u)^r$ désigne un inverse de $s$ modulo $E(u)^r$. Pour tout indice $m > 
0$, on a alors $t_m \equiv 1 \pmod {E(u)^r}$ et $t_m \equiv 0 \pmod 
{\phi^{m'}(E(u)^r)}$ si $1 \leq m' \leq m$. On définit une suite de matrices 
$(Y_m)_{1 \leq m \leq n}$ par 
\begin{equation}
\label{eq:recYm}
Y_1 = L_1 \quad \text{et} \quad  Y_{m+1} = t_m L_{m+1} + (1-t_m) 
\phi(Y_m)
\end{equation}
pour $m \in \{1, \ldots, n-1\}$. On pose enfin : 
\begin{equation} 
\label{eq:defXn}
X_n = \code{Phi} \cdot \phi(\code{Phi}) \cdots \phi^{n-1}(\code{Phi})
\cdot \omega^{-1} \cdot Y_n \cdot \Delta_n
\end{equation}
où on rappelle que $\Delta_n$ est la matrice diagonale dont le 
$i$-ième coefficient diagonal est $\lambda_n^{r-h_i}$.

\begin{prop}
Avec les notations précédentes, les vecteurs colonne de $X_n$
forment une base de $\beta_n$.
\end{prop}

\begin{proof}
Il est facile de montrer par récurrence sur $s$ que si $0 \leq m < s 
\leq n$, la matrice $Y_s$ est triangulaire inférieure unipotente et congrue à 
$\phi^{m}(L_{s-m})$ modulo $\phi^{m}(E(u)^r)$. Par ailleurs, il suit de 
l'égalité \eqref{eq:LU} que la matrice $V_m$ est inversible modulo 
$E(u)^r$ : il existe une matrice $V'_m$ à coefficients dans $\E^+$ telle 
que $V_m V'_m \equiv I \pmod{E(u)^r}$. Comme, en outre, $V_m$ est 
triangulaire supérieure, on peut supposer que $V'_m$ l'est également. Il 
suit alors des congruences $Y_n \equiv \phi^m(L_{n-m}) \pmod 
{\phi^m(E(u)^r)}$ ($0 \leq m < n$) et du fait que $\omega^{-1}$ soit 
à coefficients dans $\Q_p$ (et donc fixe par $\phi$) que :
$$\begin{array}{rcl}
X_n & \equiv & \code{Phi} \cdot \phi(\code{Phi}) \cdots \phi^{n-1}(\code{Phi})
\cdot \phi^m(\hat W'_{n-m} \Delta_1) 
\smallskip \\
& & \hspace{2cm} \phi^m (\Delta_1^{-1} V'_{n-m} \Delta_1) \cdot
(\phi^m(\Delta_1)^{-1} \Delta_n) \pmod{\phi^m(E(u)^r)}
\end{array}$$
On remarque que, bien que $\Delta_1$ ne soit pas inversible, l'écriture 
précédente a bien un sens et définit une matrice à coefficients dans 
$\E^+$. En effet, \emph{primo}, le fait que $\phi^m(\lambda_1)$ divise 
$\lambda_n$ (car $m < n$) permet de donner un sens au second facteur 
$\phi^m(\Delta_1) ^{-1} \Delta_n$ (il s'agit de la matrice diagonale 
dont le $i$-ième coefficient diagonal est 
$(\frac{\lambda_n}{\phi^m(\lambda_1)})^{r-h_i}$) et \emph{secundo}, le 
fait que $V_{n-m}$ soit triangulaire supérieure permet de définir le produit 
$\Delta_1^{-1} V_{n-m} \Delta_1$ comme la matrice triangulaire supérieure 
dont le coefficient à la position $(i,j)$ (avec $i \leq j$) est obtenu 
en multipliant le coefficient $(i,j)$ de $V_{n-m}$ par $\lambda_1^{h_i-h_j}$ 
(ce qui a bien un sens car $h_i \geq h_j$ étant donné que $i \leq j$). 
En utilisant à présent le lemme \ref{lem:congrWm}, on obtient :
$$X_n \equiv \code{Phi} \cdot \phi(\code{Phi}) \cdots \phi^m(\code{Phi})
\cdot \phi^m(\hat W' \Delta_1) \cdot P_m \pmod{\phi^m(E(u)^r)}.$$
avec 
$$P_m = \phi^m (\Delta_1^{-1} U_m \Delta_1) \cdot 
\phi^m (\Delta_1^{-1} V'_{n-m} \Delta_1)
\cdot (\phi^m(\Delta_1)^{-1} \Delta_n).$$
Comme précédemment, le fait que $U_m$ soit triangulaire supérieure
montre que le facteur $\Delta_1^{-1} U_m \Delta_1$ est bien défini. De
plus, on vérifie sans mal que la matrice $P_m$ est inversible modulo
$\phi^m(E(u)^r)$.
On est ainsi dans les conditions d'application du lemme \ref{lem:betan} 
duquel on déduit que la famille formée des vecteurs colonne de $X_n$ et 
des $\lambda_n^r e_i$ ($1 \leq i \leq d$) engendre $\beta_n$. Par 
ailleurs, on sait que la matrice $Y_n$ est triangulaire inférieure 
unipotente ; elle est donc inversible. Ceci implique, en revenant à la 
définition de $X_n$ (voir formule \ref{eq:defXn}), que tous les 
$\lambda_n^r e_i$ ($1 \leq i \leq d$) sont dans l'espace engendré par 
les vecteurs colonne de $X_n$. Le lemme en découle.
\end{proof}

\subsubsection{La matrice de $\phi$ sur le module de Breuil-Kisin}
\label{subsec:phiBK}

Nous venons de déterminer une $\E^+$-base de $\beta_n$ ce qui correspond 
d'après le théorème \ref{theo:GL} de Génestier et Lafforgue à une 
$\Sk_\nu$-base du module de Breuil-Kisin $\Dk_\nu$ dès que $\nu \leq 
\frac 1{ep^n}$. Dans cette base, la matrice de l'opérateur $\phi$ est 
donnée par la formule $\lambda_1^r \cdot X_n^{-1} \cdot \code{Phi} \cdot 
\phi(X_n)$ et vaut donc :
\begin{equation}
\label{eq:phiBK}
\code{PhiBK} = \lambda_1^r \cdot \Delta_n^{-1} \cdot Y_n^{-1} \cdot
\omega \cdot \phi^n(\code{Phi}) \cdot \phi(Y_n) \cdot \omega^{-1}
\cdot \phi(\Delta_n)
\end{equation}
(on rappelle que $\omega^{-1}$ est à coefficients dans $\Q_p$ et 
donc fixe par $\phi$). On pourra noter que, dans
le produit ci-dessus, rien n'est véritablement difficile à 
calculer. En effet, clairement déjà, appliquer le Frobenius et 
multiplier des matrices ne pose aucun problème particulier au niveau 
calculatoire et, en particulier, n'entraîne aucune perte de précision. 
L'inversion de $P$ n'est pas non plus très délicate, étant donné qu'il 
s'agit d'une matrice à coefficients dans $K_0$ et que, par ailleurs, on
dispose d'un contrôle sur les dénominateurs qui peuvent apparaître (voir
théorème \ref{theo:LU}).
Enfin, les matrices $\Delta_n$ et $Y_n$ sont respectivement diagonale et 
triangulaire supérieure unipotente et s'inversent donc aisément. On 
notera toutefois que $\Delta_n$ n'est, en réalité, pas inversible comme 
matrice à coefficients dans $\E^+$. Toutefois on sait, par avance, que 
la matrice $\code{PhiBK}$ définie comme le produit \eqref{eq:phiBK} est 
à coefficients dans $\E^+_{1/(ep^n)}$. On a donc la garantie \emph{a 
priori} que $\Delta_n$ divise $\lambda_1^r \cdot Y_n^{-1} \cdot \omega
\cdot \phi^n(\code{Phi}) \cdot \phi(Y_n) \cdot \omega^{-1} \cdot 
\phi(\Delta_n)$. Ainsi, pour faire le calcul, il suffit de prendre le 
quotient de la division euclidienne de chaque entrée $(i,j)$ de la 
matrice produit précédente par $\lambda_n^{r-h_i}$.

\medskip

L'objectif de la suite de ce numéro est d'obtenir une minoration de la 
valuation $v_\nu$ de la matrice $\code{PhiBK}$. Si 
$A$ est une matrice à coefficients dans $\E^+_\nu$, on note $v_\nu(A)$ 
la plus petite valuation d'un coefficient de $A$. Étant donné que $P$ et 
$\code{Phi}$ sont à coefficients dans $\O_{K_0}$, la relation 
\eqref{eq:phiBK} implique
$$v_\nu(\Delta_n \cdot \code{PhiBK}) \geq v_\nu(\lambda_1^r) + 
v_\nu(\omega^{-1}) + v_\nu(Y_n^{-1}) + v_\nu(\phi(Y_n)) + 
v_\nu(\phi(\Delta_n)).$$
Or, du fait que $E(u)$ est un polynôme à coefficients dans $\O_{K_0}$, 
on déduit que $v_\nu(\phi^m(E(u))) \geq 0$ pour tout $m$. Ainsi, on 
trouve $v_\nu(\lambda_1) \geq -1$ et $v_\nu(\phi(\Delta_n)) \geq -nr$ 
(puisque les $h_i$ sont tous compris entre $0$ et $r$). Par ailleurs, on 
vérifie directement que $v_\nu(\phi(x)) \geq v_\nu(x)$ pour tout $x \in 
\E^+_\nu$, d'où on déduit que $v_\nu(\phi(Y_n)) \geq v_\nu(Y_n)$. 
D'autre part, en utilisant le fait que $Y_n$ est triangulaire inférieure 
unipotente, on montre facilement en exprimant son inverse que 
$v_\nu(Y_n^{-1}) \geq (d-1) v_\nu(Y_n)$. De plus, on rappelle que l'on a 
supposé que $\omega$ vérifiait les conditions du théorème \ref{theo:LU} ;
ainsi, en particulier, on a $v_\nu(\omega^{-1}) \geq -v$ où on rappelle
que $v$ est un entier $\geq \log_p(2 n r_\card)$ (voir remarque
\ref{rem:rcard}). On remarque enfin que pour $m 
\leq n$, on a $v_\nu(\phi^m(E(u))) = \nu e p^n \leq 1$. Ainsi 
$v_\nu(\lambda_n) \leq 0$ et, comme le produit $\Delta_n \cdot 
\code{PhiBK}$ s'obtient en multipliant la $i$-ième ligne de 
$\code{PhiBK}$ par $\lambda_n^{r-h_i}$, on a $v_\nu(\Delta_n \cdot 
\code{PhiBK}) \leq v_\nu(\code{PhiBK})$. En mettant ensemble tout ce 
qui précède, on trouve :
\begin{equation}
\label{eq:valPhiBK}
v_\nu(\code{PhiBK}) \geq d \cdot v_\nu(Y_n) - v - r(n+1).
\end{equation}
Il ne reste donc plus qu'à obtenir une borne inférieure pour 
$v_\nu(Y_n)$. Or, en revenant aux définitions, on obtient facilement 
que :
\begin{equation}
\label{eq:valYn}
v_\nu(Y_n) \geq \min_{0 \leq m \leq n} v_\nu(L_m) + c_m + \cdots
+ c_n
\end{equation}
où on a posé $c_i = \min(v_\nu(t_i), 0)$. On rappelle que les $t_i$ sont 
définis par récurrence par les formules $t_0 = 1$ et $t_i = t_{i-1} 
\cdot \phi^i(\lambda_1^r) \cdot t'_i$ où $t'_i \in \E^+_\nu$ désigne un 
inverse de $\phi^i(\lambda_1^r)$ modulo $E(u)^r$. Ainsi défini, $t'_i$ 
est uniquement déterminé modulo $E(u)^r$. À partir de maintenant, pour
fixer les idées, on supposera que c'est un polynôme à coefficients dans
$K_0$ de degré $< er$ ; il est ainsi uniquement déterminé.

Pour minorer la valuation de $Y_n$, on doit minorer celle de $L_m$, 
celles des $\phi^i (\lambda_1^r)$ ainsi que celles des $t'_i$. En ce qui 
concerne $\phi^i (\lambda_1^r)$, on a déjà dit que 
$v_\nu(\phi^i(\lambda_1)) \geq -1$, ce qui donne directement 
$v_\nu(\phi^i(\lambda_1^r)) \geq -r$. Pour les deux autres, on 
considère l'isomorphisme $K_0$-linéaire :
$$\begin{array}{rcl}
T \, : \quad K_0^{<er}[u] & \to & K[u_\pi]/u_\pi^r \\
f & \mapsto & 
\displaystyle T(f) = \sum_{s=0}^{r-1} \frac 1{s!} \cdot \frac{d^s 
f}{du^s} (\pi) \cdot u_\pi^s.
\end{array}$$
où la notation $K_0^{<er}[u]$ désigne l'ensemble des polynômes à 
coefficients dans $K_0$ de degré $< er$. D'après la démonstration du 
lemme 2.4.4 de \cite{carliu}, on a 
$$T^{-1}(\O_K[u_\pi]/u_\pi^r) \subset
p^{-r \lceil \val(\df_{K/K_0}) \rceil} \: \O_{K_0}[u].$$
De la condition \eqref{item:valLm} du théorème \ref{theo:LU}, on déduit
ainsi que, si parmi toutes les matrices $L_m$ possibles, on choisit
celle à coefficients à $K_0^{<er}[u]$, alors on aura 
$v_0(L_m) \geq - r \cdot \lceil \frac 1 e + v + 2 \: \val(\df_{K/K_0})
\rceil$ et donc \emph{a fortiori}
\begin{equation}
\label{eq:valLm}
v_\nu(L_m) \geq r \cdot \left\lceil \frac 1 e + v + 2 \cdot 
\val(\df_{K/K_0}) \right\rceil.
\end{equation}
De même, pour minorer $v_\nu(t_i)$, il suffit de minorer la valuation 
de $T(t'_i)$. C'est l'objet du lemme suivant.

\begin{lemme}
\label{lem:valti}
Pour tout entier $i$, on a $T(t'_i) \in p^{-m_i} \: \O_K[u_\pi]/u_\pi^r$ 
où $m_i$ est la partie entière de $\frac r {e(p^i-1)}$.
\end{lemme}

\begin{proof}
On pose $f = \phi^i(\lambda_1)$ de sorte que $T(t'_i) = T(f)^{-r}$.
Par ailleurs, comme $E(u)$ est un polynôme d'Eisenstein de degré $e$, 
l'élément $f$ s'écrit sous la forme $\frac{u^{ep^i}} p + 
\sum_{j=0}^{e-1} a_j u^{j p^i}$ pour certains $a_j \in \O_{K_0}$ avec 
$a_0 = 1$. Les coefficients $b_s$ de $T(f)$ sont alors donnés par :
$$b_s = \frac 1{s!} \cdot \frac{d^s f}{du^s} (\pi) = 
\frac 1 p \binom {e p^i} s \pi^{e p^i-s} + \sum_{j=0}^{e-1} \binom {j 
p^i} s a_j \pi^{j p^i - s}.$$
En examinant cette formule, on remarque que tous les $b_s$ sont dans
$p^{-1} \O_K$ et, mieux encore, que $b_s \in \O_K$ pour $s \leq e
(p^i-1)$. Quant à $b_0$, il est congru à $1$ modulo $\pi$ et donc, en
particulier, inversible dans $\O_K$. On en déduit que $T(f)$ est
inversible dans $K[u_\pi]/u_\pi^r$ et que son inverse appartient à
l'image de l'application $\O_K[u_\pi, \frac{u_\pi^{e(p^i-1)}} p] \to
K[u_\pi]/u_\pi^r$. L'inverse de $T(f^r) = T(f)^r$ appartient donc
également à cette image et, à plus forte raison, à $p^{-m_i} \: 
\O_K[u_\pi]/u_\pi^r$.
\end{proof}

On déduit du lemme \ref{lem:valti} que
$v_\nu(t'_i) \geq - r (\frac 1 {e(p^i-1)} + \lceil \val(\df_{K/K_0}) 
\rceil)$.
En revenant aux définitions, cela implique $c_i \geq -r i (2 + \lceil 
\val(\df_{K/K_0}) \rceil)$, ce qui donne, pour finir, en combinant 
avec \eqref{eq:valPhiBK}, \eqref{eq:valYn} et \eqref{eq:valLm},
\begin{equation}
\label{eq:ctec}
v_\nu(\code{PhiBK}) 
\geq - d r \left( \frac{n(n+1)} 2 \cdot (2 + \lceil
\val(\df_{K/K_0}) \rceil) + \left\lceil \frac 1 e + v + 2 \cdot
\val(\df_{K/K_0}) \right\rceil \right) - v - r(n+1).
\end{equation}
La formule précédente n'est pas très ragoutante, mais on peut néanmoins 
estimer simplement son ordre de grandeur en fonction uniquement de $r$,
$d$, $\nu$ et des constantes liées au corps $K$. En effet, si l'on 
souhaite uniquement calculer l'action de $\phi$ sur $\Dk_\nu$, on peut
choisir pour $n$ n'importe quel entier $\geq - \log_p(e\nu)$. De la
même façon, la seule contrainte portant sur $v$, hormis le fait qu'il
soit entier, est $v \geq \log_p(2n r_\card)$ où on rappelle que $r_\card$
est un entier inférieur ou égal à $\min(r,d)$. Ainsi, en prenant pour
$n$ et $v$ les entiers les plus proches des bornes précédentes, on 
obtient l'estimation :
$$\textstyle v_\nu(\code{PhiBK}) \geq O(dr \cdot \log_p^2(\frac 1 \nu) + 
dr \cdot \log_p(r_\card))$$
où la constante (négative) cachée dans le $O$ ne dépend que du corps $K$. 
Par ailleurs, afin de pouvoir utiliser le théorème \ref{theo:surconv},
nous allons devoir choisir un paramètre $\nu < \frac{per}{p-1}$. Sous
cette hypothèse, on voit que le terme $dr \log_p(r_\card)$ qui apparaît 
dans le $O$ devient négligeable devant son compagnon. On obtient ainsi
simplement $v_\nu(\code{PhiBK}) \geq O(dr \cdot \log_p^2(\frac 1 \nu))$.

\subsubsection{L'algorithme sous forme synthétique}
\label{subsec:etape1synthese}

Rappelons pour commencer que l'objectif de cette première étape est, 
étant donné
\begin{itemize}
\item un $(\phi,N)$-module filtré admissible effectif $D$ donné par 
le quadruplet $(\code{Phi}, \code{N}, \code{H}, \code{F})$ (voir \S 
\ref{subsec:repr}),
\item un nombre rationnel $\nu > 0$ et
\item un nombre entier $N$
\end{itemize}
de calculer le module de Breuil-Kisin $\Dk_\nu$ sur $\Sk_\nu[1/p]$ avec 
précision $O(u^N)$. Comme $\Dk_\nu$ est un $\Sk_\nu[1/p]$-module libre de 
rang $d = \dim_{K_0} D$, se donner $\Dk_\nu$ revient à se donner la 
matrice \code{PhiBK} donnant l'action de $\phi$ sur $\Dk_\nu$ dans une 
certaine base. En mettant ensemble tout ce qui a été dit dans les \S\S 
\ref{subsec:VD}, \ref{subsec:betan}, \ref{subsec:phiBK}, on obtient 
l'algorithme suivant pour calculer \code{PhiBK}.

\begin{enumerate}[(1)]
\item Déterminer un entier $n$ tel que $\frac 1 {ep^n} \leq \nu$
\item Calculer les matrices $W_m$ ($0 \leq m < n$) données par la
formule \eqref{eq:Wm}
\item Écrire $W_m$ sous la forme $W_m = W'_m \cdot U_m$ avec 
$W'_m \in \GL_d(\O_K)$ et $U_m$ triangulaire supérieure (voir
lemme \ref{lem:decomp})
\item Calculer les matrices $\hat W'_m$ ($0 \leq m < n$) par une
variante de l'algorithme \ref{algo:relHP}
\item \label{item:alea} Tirer aléatoirement une matrice $\omega \in M_d
(\Z_p)$
\item \label{item:LU} Calculer la décomposition LU des matrices $\omega \cdot
\hat W'_m \text{ mod } E(u)^r$ ($0 \leq m < n$) :
$$\omega \cdot \hat W'_m \equiv L_m V_m \pmod{E(u)^r}.$$
Si l'une des matrices $\omega \cdot \hat W'_m$ n'est pas connue avec
suffisamment de précision pour pouvoir déterminer sa décomposition LU,
revenir en \eqref{item:alea}
\item Calculer les matrices $Y_1, \ldots, Y_n$ modulo $u^N$ par la
formule de récurrence \eqref{eq:recYm}
\item Calculer et retourner la matrice \code{PhiBK} donnée par la
formule \eqref{eq:phiBK}
\end{enumerate}

\medskip

La correction de l'algorithme ci-dessus est une conséquence de la 
discussion menée dans les \S\S \ref{subsec:VD}, \ref{subsec:betan}, 
\ref{subsec:phiBK}. De plus, si tous les coefficients des matrices
\code{Phi}, \code{N} et \code{F} sont connus avec précision $O(p^M)$,
en suivant les pertes de précision au fil des calculs, on trouve
successivement que :
\begin{itemize}
\item les coefficients des matrices $W_m$ sont connus avec précision 
au moins $O(p^{M-mr})$ ;
\item les coefficients des matrices $W'_m$ sont connus avec précision 
au moins $O(p^{M-mr-dr})$ ;
\item par le lemme \ref{lem:precrelHP}, les coefficients de la matrice 
$\hat W'_m$ sont connus avec précision au moins $O(p^{M-mr-dr - \rho_1})$ 
où $\rho_1$ est la constante définie dans ce lemme où on a pris $i = r$ ;
\item par l'alinéa \eqref{item:valLm} du théorème \ref{theo:LU}, avec
probabilité $\geq \frac 1 2$, le calcul de l'étape \eqref{item:LU} n'échoue
pas et les coefficients de la matrice $L_m$ sont connus avec précision
au moins $O(p^{M-mr-dr - \rho_1 - 2 \lceil \rho_2 \rceil})$ où $\rho_2$ 
est la constante définie dans ce théorème ;
\item d'après l'estimation 
$v_\nu(t'_i) \geq - r (\frac 1 {e(p^i-1)} + \lceil \val(\df_{K/K_0})
\rceil)$ qui se déduit du lemme \ref{lem:valti}, les coefficients des
matrices $Y_m$ puis de la matrice \code{PhiBK} que l'on renvoie sont 
connus avec précision au moins $O(p^{M - M_0})$ où $M_0$ est une 
constante que l'on peut exprimer explicitement.
\end{itemize}
De même qu'à la fin du \S \ref{subsec:phiBK}, on peut déduire de
l'expression explicite de $M_0$ que, sous l'hypothèse supplémentaire
$\nu < \frac{per}{p-1}$, on a :
\begin{equation}
\label{eq:precreletape1}
\textstyle M_0 = O(dr \cdot \log_p^2(\frac 1 \nu)).
\end{equation}

\subsection{Étape 2 : Calcul d'un réseau dans un module de Breuil-Kisin}
\label{subsec:etape2}

Maintenant que nous avons calculé $\Dk_\nu$,
l'étape suivante consiste à déterminer un $\Sk_\nu$-réseau $\Mk_\nu$ à 
l'intérieur de $\Dk_\nu$ qui soit un module de Kisin de hauteur $\leq 
r$. Remarquons que l'on connaît déjà un réseau à l'intérieur de 
$\Dk_\nu$ : c'est celui donné par les vecteurs d'une base de la matrice 
\code{PhiBK} que l'on a calculée lors de la première étape et que l'on
notera à partir de maintenant $\Mk_\nu^\GeLa$. En outre, de même que
précédemment on a obtenu une estimation de la constante $M_0$, on peut 
déterminer, par 
des arguments analogues, une constante $C'$ qui dépend de façon 
polynômiale de $d$, $r$ et $\log_p(\frac 1 \nu)$ et qui vérifie :
$$p^{C'} \Mk_\nu^\GeLa \subset \Sk_\nu \otimes_\Sk
\beta_n(L) \subset p^{-C'} \Mk_\nu^\GeLa$$
où $L \subset D$ est le $\O_{K_0}$-réseau standard engendré par les 
vecteurs de la base de $D$ qui a été fixée au départ\footnote{Il s'agit
de la base de $D$ dans lesquelles les matrices \code{Phi}, \code{N} et
\code{F} ont été écrites.}.
Par ailleurs, d'après la remarque \ref{rem:GL}, il existe une constante 
$C$ qui s'exprime de façon polynômiale en $d$, $r$ et $-\log_p(\nu)$ et
un module de Breuil-Kisin $\Mk_\nu$ sur $\Sk_\nu$ tels que
$$p^{-C'} \Sk_\nu \otimes_\Sk \beta_n(L) \subset
\Mk_\nu \subset p^{C'-C} \cdot (\Sk_\nu \otimes_\Sk \beta_n(L)).$$
En posant $c_\nu = C + 2 C'$, on obtient $p^{c_\nu} \cdot \Mk_\nu^\GeLa 
\subset \Mk_\nu \subset \Mk_\nu^\GeLa$. De plus, par les résultats du 
\S \ref{subsec:phiBK}, on sait calculer une constante explicite $c$
telle que $\phi(\Mk_\nu^\GeLa) \subset p^{-c} \cdot \Mk_\nu^\GeLa$.
On a en outre $c = O(dr \cdot \log_p^2(\frac 1 \nu))$ où la constante
dans le $O(\cdot)$ dépend du corps $K$ et, à vrai dire, plus précisément,
de la valuation de la différente de $K$ à $K_0$. De plus, en revenant au
calcul de \S \ref{subsec:phiBK}, on se rend compte que cette dépendance
est, au pire, polynômiale.

\medskip

Dans cette partie, nous expliquons comment déterminer $\Mk_\nu$ --- 
ou plutôt un $\Mk_\nu$ convenable --- connaissant $\Mk_\nu^\GeLa$ et 
$c_0$. Le \S \ref{subsec:algomax} est consacré à quelques rappels, 
extraits de \cite{carlub}, concernant la manipulation algorithmique des 
$\Sk_\nu$-modules. On entre dans le c\oe{}ur du sujet avec les numéros 
suivants. Dans les \S\S \ref{subsec:iteridee} et \ref{subsec:iterfrob}, 
on montre comment, en itérant le Frobenius, on parvient à construire un
module de Breuil-Kisin $\Mk_{\nu,D}$ défini non pas sur $\Sk_\nu$ mais
sur un anneau de séries formelles à coefficients dans $K_0[\sqrt[D] p]$
(pour un certain entier $D$ bien choisi).
Dans le \S \ref{subsec:liberte}, on explique comment, quitte à faire 
quelques petites modifications, on peut redescendre $\Mk_{\nu,D}$ en un 
authentique module de Breuil-Kisin défini (et libre) sur $\Sk_\nu$ et,
enfin, dans le \S \ref{subsec:etape2synthese}, on présente l'algorithme 
obtenu sous forme synthétique.

\subsubsection{L'algorithmique des $\Sk_\nu$-modules : rappels et 
compléments}
\label{subsec:algomax}

On rappelle dans ce numéro quelques algorithmes tirés de \cite{carlub} 
pour la manipulation des $\Sk_\nu$-modules. Ceux-ci seront constamment 
utilisés dans la suite.

D'un point de vue algorithmique, le fait que $\Sk$ ne soit pas un anneau 
principal est un souci majeur. En effet, sur un anneau $A$ non 
principal, il existe par définition des sous-modules de $A^d$ qui ne 
sont pas libres, et travailler sur machine avec des modules non libres
est nettement plus difficile.
Par chance, la structure de $\Sk$ est suffisamment simple (il s'agit 
d'un anneau local régulier de dimension $2$) pour que l'on puisse
quand même systématiquement se ramener à des modules libres quitte à
ajouter de temps en temps un morceau fini.
Plus précisément, si $\Mk$ est un sous-$\Sk$-module de type fini de
$(\E^+)^d$, on définit
$$\Max(\Mk) = \Big\{ \, x \in (\E^+)^d \,\, \big| \,\, \exists n \in 
\N, \, p^n x \in \Mk \text{ et } u^n x \in \Mk \, \Big\}.$$
Il est clair que le quotient $\Max(\Mk) / \Mk$ est annulé, à la fois,
par une puissance de $p$ et une puissance de $u$. Comme $\Mk$ est
finiment engendré, ceci implique que $\Max(\Mk) / \Mk$ est de 
longueur fini comme $\Sk$-module. En particulier, c'est un ensemble
fini si le corps résiduel $k$ est lui-même fini. Il fait ainsi sens
de dire que l'on passe de $\Mk$ à $\Max(\Mk)$ en \og ajoutant un
morceau fini \fg.

\begin{theo}[Iwasawa]
\label{theo:iwasawa}
Pour tout $\Sk$-module de type fini $\Mk \subset (\E^+)^d$, le module
$\Max(\Mk)$ est le plus petit sous-module libre de $(\E^+)^d$ qui
contient $\Mk$.
\end{theo}

Si l'on suppose que l'on sait manipuler dans leur intégralité les 
éléments de $\Sk$, le $\Max$ que l'on a défini ci-dessus a, de surcroît, 
la propriété intéressante de pouvoir être calculé explicitement ; un \og 
algorithme \fg, pour ce faire, est décrit dans le \S 3.3.1 de 
\cite{carlub}.
Pour le propos de cet article, on aura besoin de l'appliquer uniquement 
dans un cas particulier plus simple qui est celui du calcul de $\Max(\Mk 
+ x\Sk)$ où $\Mk$ est un sous-module libre de $(\E^+)^d$ de rang maximal 
donné par l'intermédiaire d'une base et $x$ est un vecteur de $\Sk^d$. 
Dans ce cas, si on note $v_0$ (resp. $\deg_0$) la valuation de Gauss 
(resp. le degré de Weierstrass) sur $\Sk$, l'algorithme fonctionne comme 
suit :

\medskip

\begin{algorithm}[H]
  \SetKwInOut{Input}{Entrée}
  \SetKwInOut{Output}{Sortie}

  \Input{une base $(e_1, \ldots, e_d)$ de $\Mk \subset (\E^+)^d$ 
    et un vecteur $X \in (\E^+)^d$}
  \Output{une base de $\Max(\Mk + X \Sk)$}
  \BlankLine
  $(x_1, \ldots, x_d)$ $\leftarrow$ coordonnées de $x$ sur la base
  $(e_1, \ldots, e_d)$\label{lgn:coord}\;
  \While{l'un des $x_i$ n'est pas dans $\Sk$}{
    $i \leftarrow$ indice tel que $v_0(x_i)$ est minimal\;
    $j \leftarrow$ indice distinct de $i$ tel que $v_0(x_j)$ est minimal\;
    $\delta \leftarrow \min(0,v_0(x_j)) - v_0(x_i);$\,
    $x_i \leftarrow p^\delta x_i;$\, 
    $e_i \leftarrow p^{-\delta} e_i$\;
    \If{$v_0(x_j) < 0$}{
      \lIf{$\deg_0(x_i) < \deg_0(x_j)$}{échanger $i$ et $j$}\;
      $(q,r) \leftarrow $ quotient et reste de la division euclidienne
      de $x_i$ par $x_j$\;
      $x_i \leftarrow r;$\,
      $e_i \leftarrow e_i + q e_j$\;
    }
  }
  \Return $(e_1, \ldots, e_d)$\;
\caption{\sc AjoutVecteur$(\Mk,x)$\label{algo:ajoutvecteur}}
\end{algorithm}

\medskip

\noindent
La démonstration de la correction de l'algorithme précédent n'est pas 
très difficile. On commence par remarquer qu'après chaque itération, la 
famille des $e_i$ reste toujours libre, que l'espace qu'elle engendre 
est inclus dans $\Max(\Mk + x \Sk)$ et, finalement, que la relation 
$x = x_1 e_1 + \cdots + x_d e_d$ continue d'être vérifiée.
Ainsi, lorsque l'on quitte la boucle, $x$ s'écrit comme combinaison 
linéaire à coefficients dans $\Sk$ des $e_i$ ; autrement dit, si on 
appelle $\Mk'$ l'espace engendré par ces $e_i$, on a $x \in \Mk'$
et donc $\Mk + x\Sk \subset \Mk'$. Par ailleurs, comme les $e_i$
forment une famille libre, le module $\Mk'$ est libre et l'inclusion
précédente donne alors $\Max(\Mk + x \Sk) \subset \Mk'$. L'égalité
en résulte puisque l'on a déjà dit que l'inclusion précédente était
vraie.

\bigskip

Dans la suite, on aura besoin de travailler non seulement avec des 
$\Sk$-modules, mais aussi avec $\Sk_\nu$-modules. Or, malheureusement, 
le théorème \ref{theo:iwasawa} ne vaut plus lorsque l'on remplace $\Sk$ 
par $\Sk_\nu$. Pour résoudre ce problème, on suppose que $\nu$ est un
nombre rationnel s'écrivant sous la forme $\nu = \frac a b$, on fixe
un élément $\varpi_b \in \bar K$ tel que $\varpi_b^b = p$ et on 
introduit l'anneau
$$\Sk_{\nu,b} = \Bigg\{ \, \sum_{i \in \N} a_i u^i \quad \Big| \quad
\begin{array}{l}
a_i \in K_0[\varpi_b] \smallskip \\
\val(a_i) + \nu i \geq 0, \, \forall i \smallskip 
\end{array} \, \Bigg\}.$$
Celui-ci s'identifie à l'anneau des séries formelles à coefficients
dans $\O_{K_0}[\varpi_b]$ --- qui est l'anneau des entiers de 
$K_0[\varpi_b]$ --- en la variable $\frac u {\varpi_b^a}$, d'où on
déduit que le théorème d'Iwasawa s'applique à cet anneau. Pour éviter
les confusions, on notera $\Max_{\nu,b}$ le foncteur $\Max$ 
correspondant ; en posant $\E^+_{\nu,b} = \Sk_{\nu,b}[1/p]$, on a donc
$$\textstyle \Max_{\nu,b}(\Mk_{\nu,b}) =
\Big\{ \, x \in (\E^+_{\nu,b})^d \,\, \big| \,\, \exists n \in
\N, \, p^n x \in \Mk_{\nu,b} \text{ et } (\frac u {\varpi_b^a})^n x 
\in \Mk_{\nu,b} \, \Big\}$$
pour un $\Sk_{\nu,b}$-module de type fini $\Mk_{\nu,b} \subset
(\E^+_{\nu,b})^d$.
Lorsqu'en outre $\Mk_{\nu,b}$ est défini sur $\Sk_\nu$, on peut 
démontrer (voir proposition 3.9 et lemme 3.18 de \cite{carlub}) que 
$\Max_{\nu,b}(\Mk_{\nu,b})$ admet une base formée des vecteurs de la 
forme $e_i = \varpi_b^{-b_i} B_i$ avec $b_i \in \N$ et $B_i \in 
(\E^+_\nu)^d$. Mieux encore, si l'on applique l'algorithme 
\ref{algo:ajoutvecteur} avec pour entrée des vecteurs $e_i$ et $x$ 
prenant la forme précédente, alors la sortie a également cette 
forme\footnote{On peut utiliser cette remarque pour démontrer 
que $\Max_{\nu,b}(\Mk_{\nu,b})$ prend la forme annoncée lorsque 
$\Mk_{\nu,b}$ est défini sur $\Sk_\nu$ ; il s'agit d'ailleurs de 
l'approche qui est suivie dans \cite{carlub}.}.

Enfin, dans le \S 3.3.4 de \cite{carlub}, il est présenté un algorithme 
qui, étant donné un $\Sk_{\nu,b}$-module libre $\Mk_{\nu,b}$ engendré 
par des vecteurs $e_i = \varpi_b^{-b_i} B_i$ avec $b_i \in \N$ et $B_i 
\in (\E^+_\nu)^d$, calcule l'intersection $\Mk_{\nu,b} \cap 
(\E^+_\nu)^b$. Dans le cas général, la description de cet algorithme 
n'est pas triviale\footnote{Elle fait appel à la théorie des fractions 
continues, au moins si l'on souhaite une version efficace.} mais elle
prend une forme particulièrement simple lorsque $\nu = \frac 1 b$ ;
en effet, dans ce cas, on montre que, si $q_i$ et $r_i$ désignent 
respectivement le quotient et le reste de la division euclidienne de
$-b_i$ par $b$, l'intersection $\Mk_{\nu,b} \cap (\E^+_\nu)^b$ est 
engendrée par les vecteurs
$$p^{q_i+1} B_i \quad \text{et} \quad p^{q_i} u^{r_i} B_i$$
pour $i$ variant de $1$ à $d$.

\bigskip

Tel qu'il est écrit, l'algorithme \ref{algo:ajoutvecteur} souffre d'un 
défaut majeur lié à la précision : lorsqu'un élément $x \in \E^+$ n'est 
connu qu'à précision $u$-adique fini, il n'est pas possible de calculer 
$v_0(x)$ ! En effet, si $x = \sum_{i \in \N} a_i u^i$ (avec $a_i \in 
K_0$) et si l'on ne connaît que les $a_i$ pour $i < N$, il est 
impossible de dire avec certitude quel est le minimum des $\val(a_i)$. 
Or, manifestement, savoir calculer les valuations d'éléments de $\E^+$ 
est essentiel au bon déroulement de l'algorithme 
\ref{algo:ajoutvecteur}.

Résoudre ce problème sans hypothèse supplémentaire paraît très délicat. 
Il est toutefois possible d'y apporter une réponse satisfaisante si l'on 
suppose en outre que $X \in p^{-c} \: \Mk$ pour un certain entier $c$ 
connu, ce qui revient encore à dire que tous les $x_i$ sont dans $p^{-c} 
\: \Sk$. Or, si on travaille sur machine avec la représentation PAG des 
éléments de $\Sk$ (voir \S \ref{subsec:repr}), un entier $c$ vérifiant 
la condition précédente est facile à obtenir : c'est le minimum des 
garanties des éléments $x_i$.
Toutefois, dans un souci de simplification, plutôt que de recourir aux 
informations contenues dans les représentations PAG, nous supposons 
simplement par la suite qu'un entier $c$ convenable est passé en 
argument à l'algorithme \ref{algo:ajoutvecteur}. Ceci ne portera pas à 
préjudice car, dans la situation qui nous occupe, un tel entier $c$ sera 
connu \emph{a priori} grâce au calcul du \S \ref{subsec:phiBK}.

Expliquons à présent comment utiliser cette hypothèse supplémentaire que 
nous venons de formuler. Cela ne coule pas de source car elle n'est, de 
fait, pas suffisante pour garantir que l'on puisse calculer les 
$v_0(x_i)$. Cependant, si les calculs se font à la précision $u$-adique 
$O(u^N)$, elle assure que l'on peut déterminer sans ambiguïté les 
$v_{c/N}(x_i)$ au moins si ceux-ci sont négatifs ou nuls : si un élément 
$x = \sum_{i \in \N} a_i u^i$ est dans $p^{-c} \: \Sk$, tous les 
coefficients $a_i$ ont une valuation $\geq -c$, d'où on tire que 
$\val(a_i) + \frac c N i \geq 0$ dès que $i \geq N$. Or, un examen 
rapide de l'algorithme \ref{algo:ajoutvecteur} montre que son 
comportement n'est pas modifié si l'on commet une erreur sur le calcul 
d'une valuation positive (dans la mesure, bien sûr, où l'on obtient 
quand même un résultat positif).
Ainsi, sous l'hypothèse supplémentaire que l'on a faite, afin de faire 
fonctionner l'algorithme \ref{algo:ajoutvecteur} avec des calculs à 
précision finie $u^N$, il suffit de remplacer partout $v_0$ par 
$v_{c/N}$ et $\E^+$ par $\E^+_{\nu,N}$. Le résultat renvoyé est une base 
de $\Max_{2c/N,N}(\Sk_{2c/N,N} \otimes_{\Sk_\nu} \Mk + x \: 
\Sk_{2c/N,N})$ si l'on souhaite un résultat exact. Si l'on peut se 
contenter d'un résultat approché à $u^N$, l'algorithme renvoie une base 
de $\Max_{c/N,N}(\Sk_{c/N,N} \otimes_{\Sk_\nu} \Mk + x \: \Sk_{c/N,N})$ 
avec $\frac c N$ à la place de $\frac{2c} N$ ; exactement, cela signifie 
que, quitte à ajouter des multiples de $u^N$ aux vecteurs renvoyés par 
l'algorithme, ceux-ci forment une base du dernier espace que l'on a 
écrit. Tout ceci s'étend au cas où l'on part d'un module $\Mk$ et d'un 
vecteur $x$ définis sur $\Sk_{\nu,b}$ et que l'on prend pour $N$ un 
multiple de $b$.

\paragraph{Utilisation de l'algorithme \ref{algo:ajoutvecteur} 
pour la suite}

Dans la suite de cet article, on manipulera essentiellement des modules 
sur les anneaux $\Sk_{\nu,b}$ munis d'un opérateur semi-linéaire $\phi$ 
et, plutôt que de calculer une base d'un tel module, il sera souvent 
plus intéressant, pour ce que l'on souhaite faire, de calculer l'action 
de $\phi$ sur ce module. 
Ainsi, plutôt qu'un algorithme qui prend en entrée un module $\Mk$, un 
vecteur $x$ et renvoie le $\Max$ du module engendré par $\Mk$ et $x$ 
(après une extension éventuelle des scalaires), on utilisera avec 
plaisir un algorithme qui prend en entrée une matrice \code{PhiBK} 
donnant l'action d'un opérateur $\phi$ agissant sur un module $\Mk$ et 
un vecteur $x$ et qui renvoie une nouvelle matrice \code{PhiBK} donnant 
cette fois-ci l'action de $\phi$ sur le $\Max$ du module engendré par 
$\Mk$ et $x$ (encore une fois, éventuellement, après extension des 
scalaires).

C'est exactement ce que fait l'algorithme \ref{algo:changebase}, 
présenté page \pageref{algo:changebase}. En plus de cela, il prend en 
compte toutes les remarques et améliorations qui ont été mises en 
lumière précédemment et met à profit la remarque sur la forme 
particulière $\varpi_N^{-b_i} B_i$ que prennent les vecteurs de base que 
l'on calcule pour faire en sorte de toujours travailler uniquement sur 
les anneaux $\E^+_\nu$, et non sur les $\E^+_{\nu,b}$.

\begin{algorithm}[h]
  \SetKwInOut{NB}{\it NB :}
  \SetKwInOut{Nota}{\it Notation :}
  \SetKwInOut{Input}{Entrée :}
  \SetKwInOut{Output}{Sortie :}
  \SetKwInOut{Void}{}

  \Nota{Si $V$ est un vecteur, on désigne par $V[i]$ sa $i$-ième composante}
  \Void{Si $A$ est une matrice, on désigne par :}
  \Void{\hspace{0.5cm} $A[i,j]$ son coefficient en $i$-ième ligne et $j$-colonne}
  \Void{\hspace{0.5cm} $A[i,\undef]$ sa $i$-ième ligne}
  \Void{\hspace{0.5cm} $A[\undef,j]$ sa $j$-ième colonne}

  \BlankLine

  \NB{Tous les calculs dans cet algorithme sont effectués à précision $u^N$}

  \BlankLine
  \BlankLine

  \Input{$\star$ un vecteur $b = (b_1 \, \ldots \, b_d)$ d'entiers relatifs}
  \BlankLine
  \Void{$\star$ une matrice $\code{PhiBK} \in M_d(\E^+_\nu)$ donnant 
    l'action, dans la base canonique $(e_i)_{1 \leq i \leq d}$,}
  \Void{\hphantom{$\star$} d'un opération semi-linéaire 
    $\phi : (\E^+_{\nu,N})^d \to (\E^+_{\nu,N})^d$}
  \BlankLine
  \Void{$\star$ un vecteur $x \in (\E^+_{\nu,N})^d$ donné sous la forme
    $x = \varpi_N^{-a} X$}
  \Void{\hphantom{$\star$} où $a$ est un entier et $X$ un vecteur 
    colonne à coefficients dans $\E^+_\nu$}
  \BlankLine
  \Void{$\star$ les paramètres habituels $c$, $N$ et $\nu$}

  \BlankLine
  \BlankLine

  \Output{$\star$ un vecteur $b' = (b'_1 \, \ldots \, b'_d)$ d'entiers relatifs}
  \BlankLine
  \Void{$\star$ la matrice $\code{PhiBK}'$ de $\phi$ écrite dans la base 
    des $e'_i$ où les $\varpi_N^{-b'_i} e'_i$}
  \Void{\hphantom{$\star$} forment une base de
    $\Max_{\nu + c/N,N}\big( \big< x, \, \varpi_N^{-b_i} e_i \, 
     (1 \leq i \leq d)\big>\big)$}

  \BlankLine
  \BlankLine

  $\nu' \leftarrow \nu + \frac c N$\;
  \While{il existe $i$ tel que $v_{\nu'}(X[i]) + b_i < a$}{
    $i \leftarrow$ indice tel que $v_{\nu'}(X[i]) + b_i$ est minimal$;$\,
    $v_i \leftarrow v_{\nu'}(X[i]) + b_i$\label{lgn:whilestart}\;
    $j \leftarrow$ indice distinct de $i$ tel que $v_{\nu'}(X[j]) + b_j$ est minimal$;$\,
    $v_j \leftarrow v_{\nu'}(X[j]) + b_j$\;
    $\delta \leftarrow \min(a,v_j) - v_i;$\,
    $b_i \leftarrow b_i + \delta$\label{lgn:updatebi}\;
    \If{$v_j < a$}{
      \lIf{$\deg_{\nu'}(X[i]) < \deg_{\nu'}(X[j])$}{échanger $i$ et $j$}\label{lgn:swapij}\;
      $(q,r) \leftarrow $ quotient et reste de la division euclidienne
      (dans $\E^+_\nu$) de $X[i]$ par $X[j]$\label{lgn:diveuc}\;
      $X[i] \leftarrow r$\label{lgn:updateXi}\;
      $\code{PhiBK}[\undef,i] \leftarrow \code{PhiBK}[\undef,i] + \phi(q) \cdot \code{PhiBK}[\undef,j]$\;
      $\code{PhiBK}[i,\undef] \leftarrow \code{PhiBK}[i,\undef] - q \cdot \code{PhiBK}[j,\undef]$\label{lgn:whileend}\;
    }
  }
  \Return $b,\code{PhiBK}$\;
\caption{{\sc ChangeBase$(b,\code{PhiBK},a,X,c,N,\nu)$}
\hfill {\it Hypothèse : $N\nu$ est entier}
\label{algo:changebase}}
\end{algorithm}

\subsubsection{Une idée simple}
\label{subsec:iteridee}

On revient à présent à la situation du début du \S \ref{subsec:etape2}. 
Cependant, dans cette partie introductive destinée simplement à 
présenter les idées sous-jacente à l'algorithme, on suppose pour 
simplifier que $\nu = 0$. On suppose donc donnés un $\Sk$-module 
$\Mk^\GeLa$, ainsi que des entiers $c$ et $c_0$ vérifiant 
$\phi(\Mk^\GeLa) \subset p^{-c} \cdot \Mk^\GeLa$ et $p^{c_0} \cdot 
\Mk^\GeLa \subset \Mk \subset \Mk^\GeLa$ pour un certain module de 
Breuil-Kisin $\Mk$.

\medskip

Puisque $\Mk$ doit être un module de Breuil-Kisin, il est particulier 
stable par $\phi$ et donc, s'il contient $\Mk^\GeLa$, il contient 
nécessairement aussi $\phi(\Mk^\GeLa)$ et donc par suite la somme 
$\Mk^\GeLa + \big<\phi(\Mk^\GeLa)\big>_\Sk$. Comme en outre $\Mk$ est 
libre, il contient nécessairement le $\Max$ de cette somme d'après le 
théorème \ref{theo:iwasawa}. En répétant l'argument, on trouve que $\Mk$ 
contient tous les sous-modules $\Mk^{(s)} \subset \E^+ \otimes_\Sk 
\Mk^\GeLa$ définis par récurrence par :
\begin{equation}
\label{eq:defMi}
\Mk^{(1)} = \Mk^\GeLa \quad ; \quad
\Mk^{(s+1)} = \Max\Big(\Mk^{(s)} + \big<\phi(\Mk^{(s)})\big>_\Sk\Big).
\end{equation}
Comme tous les $\Mk^{(s)}$ sont inclus dans $p^{-c_0} \: \Mk^\GeLa$ qui 
est libre de rang fini sur l'anneau $\Sk$ qui est noethérien, la suite 
des $\Mk^{(s)}$, qui est manifestement croissante, est stationnaire. 
Soit $\Mk^{(\infty)}$ sa limite ; à l'évidence, c'est un 
sous-$\Sk$-module libre de $\Dk$ qui est stable par $\phi$. La 
proposition suivante montre que cela suffit à en faire un module de 
Breuil-Kisin.

\begin{prop}
\label{prop:Eurauto}
Soient $\Mk$ un module de Breuil-Kisin de hauteur $\leq r$ sur $\Sk$ et
$\Dk = \E^+ \otimes_\Sk \Mk$. Soit $\Mk'$ un sous-$\Sk$-module de 
$\Dk$ libre de rang fini, stable par $\phi$ tel que $\Dk = \E^+
\otimes_\Sk \Mk'$.
Alors $\Mk'$ est un module de Breuil-Kisin de hauteur $\leq r$.
\end{prop}

\begin{proof}
On suppose pour commencer que $\Mk$ est de rang $1$ ; il en est alors de 
même de $\Mk'$. On note $x$ (resp. $x'$) une base de $\Mk$ (resp. de 
$\Mk'$) sur $\Sk$ et on note $s$ (resp. $s'$) l'unique élément de $\Sk$ 
tel que $\phi(x) = sx$ (resp. $\phi(x') = s'x'$). On introduit également 
l'élément $a \in \E^+$ défini par l'égalité $x' = ax$. Il est inversible 
dans $\E^+$ (puisque $\E^+ \otimes_\Sk \Mk = \E^+ \otimes_\Sk \Mk' = 
\Dk$) et s'écrit donc sous la forme $a = p^n b$ avec $n \in \Z$ et $b$
inversible dans $\Sk$.
Par ailleurs, de $\phi(x') = \phi(ax) = \phi(a) \phi(x) = \phi(a) s x$, 
on déduit la relation $s' = \frac{\phi(a)} a \: s = \frac{\phi(b)} b \: 
s$. Or, comme $\Mk$ est un module de Breuil-Kisin de hauteur $\leq r$, il est
clair que $s$ divise $E(u)^r$. Ainsi $s'$ divise lui aussi $E(u)^r$, ce
qui implique réciproquement que $\Mk'$ est un module de Breuil-Kisin de hauteur
$\leq r$.

Si maintenant $\Dk$ est de rang $d$, on considère sa $d$-ième puissance 
extérieure $\Lambda^d \Dk$. À l'intérieur de ce $\E^+$-module de rang 
$1$ vivent les réseaux $\Lambda^d \Mk$ et $\Lambda^d \Mk'$. Le premier 
est un module de Breuil-Kisin de hauteur $\leq dr$, d'où il résulte, \emph{via} 
le premier point, qu'il en est de même du second. Cela implique 
facilement que $\Mk'$, lui-même, est un module de Breuil-Kisin de hauteur $\leq 
dr$. Il faut à présent montrer qu'il est en fait de hauteur $\leq r$. 
Pour cela, on fixe une $\Sk$-base $(e'_1, \ldots, e'_d)$ de $\Mk'$ et 
on considère un élément $x \in \Mk'$. Puisque $\Mk$ est de hauteur $\leq 
r$ et que les $e'_i$ forment une $\E^+$-base de $\Dk = \Mk[1/p]$, on 
peut écrire le produit $E(u)^r x$ sous la forme :
$$E(u)^r x = a_1 \phi(e'_1) + a_2 \phi(e'_2) + \cdots + a_d \phi(e'_d)$$
pour des éléments $a_i \in \E^+$. Pour conclure, il suffit de démontrer
que tous les $a_i$ sont dans $\Sk$. Or, puisque l'on a prouvé que $\Mk'$ 
est de hauteur $\leq dr$, on sait que $E(u)^{rd} x$ appartient au 
$\Sk$-module engendré par les $\phi(e'_1)$, ce qui revient à dire que 
tous les $E(u)^{r(d-1)} a_i$ appartiennent à $\Sk$. En prenant les
valuations de Gauss, on trouve $r(d-1) v_0(E(u)) + v_0(a_i) \geq 0$,
soit encore $v_0(a_i) \geq 0$ car $v_0(E(u)) = 0$. Ainsi $a_i \in \Sk$
pour tout $i$ et on a bien démontré ce qui avait été annoncé.
\end{proof}

À ce stade, on a compris ce que l'on a à faire pour calculer un module 
de Breuil-Kisin $\Mk$ à l'intérieur de $\Dk$ : on calcule les 
$\Mk^{(n)}$ définis par la formule récurrence \eqref{eq:defMi} jusqu'à 
ce que celle-ci stationne et, à ce moment, le module $\Mk^{(\infty)}$ 
obtenu est un module de Breuil-Kisin comme souhaité. On peut en outre 
borner explicitement le rang à partir duquel la suite des $\Mk^{(n)}$ 
stationne, comme le précise le lemme suivant.

\begin{lemme}
\label{lem:long0}
Sous les hypothèses précédentes, on a $\Mk^{(\infty)} = \Mk^{(d)}$.
\end{lemme}

\begin{proof}
Soit $\O_\E$ le complété $p$-adique de $\Sk[1/u]$. La valuation de Gauss 
$v_0$ s'étend à $\O_\E$ et fait de ce dernier un anneau de valuation 
discrète complet de corps résiduel $k((u))$. De même, il est clair 
que le Frobenius $\phi$ se prolonge en un endomorphisme de $\O_\E$. 
Par ailleurs, il suit 
du théorème 3.12 (appliqué avec $\nu = 0$) de \cite{carlub} qu'en posant 
$M^{(n)} = \O_\E \otimes_\Sk \Mk^{(n)}$, on a les deux propriétés 
suivantes valables pour tout $n \geq 0$ :
\begin{enumerate}[(i)]
\item l'égalité $\Mk^{(n)} = \Mk^{(n+1)}$ est équivalente à $M^{(n)} = 
M^{(n+1)}$ ;
\item $\Mk^{(n+1)} = \Mk^{(n)} + (\phi \otimes \phi)(\Mk^{(n)})$.
\end{enumerate}
Pour tout entier $n$, notons $f_n : \frac{M^{(n)}}{p\:M^{(n)}} \to 
\frac{M^{(\infty)}}{p\:M^{(\infty)}}$ l'application déduite de 
l'inclusion $M^{(n)} \subset M^{(\infty)}$ ; il s'agit d'une application 
$k((u))$-linéaire entre $k((u))$-espaces vectoriels de dimension $d$. De 
la suite d'inclusions $M^{(n)} \subset M^{(n+1)} \subset M^{(\infty)}$,
on déduit que $f_n$ se factorise par $f_{n+1}$ et, par suite, que l'image
de $f_{n+1}$ contient celle de $f_n$. On a ainsi la suite d'inclusions :
$$\im f_0 \subset \im f_1 \subset \cdots \subset \im f_{d+1}.$$
Comme tous les espaces ci-dessus sont des sous-$k((u)$-espaces vectoriels
de $\frac{M^{(\infty)}}{p\:M^{(\infty)}}$ qui est de dimension $d$, il
existe nécessairement un entier $\ell \leq d$ tel que $\im f_\ell = 
\im f_{\ell+1}$. En revenant aux définitions, cela signifie que, pour 
cet entier $\ell$, on a $M^{(\ell+1)} \subset M^{(\ell)} + p 
M^{(\infty)}$. En appliquant $\phi \otimes \phi$ à cette égalité puis
en sommant, on obtient $M^{(\ell+2)} \subset M^{(\ell + 1)} + p 
M^{(\infty)} \subset M^{(\ell)} + p M^{(\infty)}$. En répétant 
l'argument, on démontre par récurrence que, pour tout entier $n$, 
l'inclusion $M^{(\ell+n)} \subset M^{(\ell)} + p M^{(\infty)}$ est 
vérifiée. En passant à la limite, on en déduit que $M^{(\infty)} \subset 
M^{(\ell)} + p M^{(\infty)}$. Ceci implique que $M^{(\infty)} \subset
M^{(\ell)}$ puis que ces deux espaces coïncident étant donné que 
l'inclusion réciproque est vraie par construction. Comme $\ell \leq 
d$, on en déduit l'énoncé du lemme.
\end{proof}

\begin{rem}
\label{rem:fastiter}
Plutôt que de calculer les $\Mk^{(n)}$ un par un à l'aide de la formule 
itérative \eqref{eq:defMi}, il est bien plus efficace de procéder 
comme suit. On pose $\phi_0 = \phi$, $\Mk_{(0)} = \Mk^\GeLa$ et on
définit par récurrence :
\begin{equation}
\label{eq:defMi2}
\phi_{m+1} = \phi_m \circ \phi_m \quad ; \quad
\Mk_{(m+1)} = \Max\Big(\Mk_{(m)} + \big<\phi_m(\Mk_{(m)})\big>_\Sk\Big).
\end{equation}
Il est facile de montrer que $\Mk_{(m)} = \Mk^{(2^m)}$ pour tout entier 
$i \geq 0$. Ainsi comme on sait par le lemme \ref{lem:long0} que la 
suite des $\Mk^{(n)}$ stationne au plus après $d$ étapes, la limite sera 
atteinte au bout de $\log_2 d$ étapes pour la suite des $\Mk_{(m)}$.
Ceci constitue un gain important pour la complexité. 
\end{rem}

\paragraph{Les complications}

Les principales complications viennent du fait que, malheureusement, on 
ne peut pas faire comme si tout se passait sur $\Sk$. En effet,
\begin{enumerate}[(i)]
\item on ne dispose en général pas du module $\Mk^\GeLa$ --- qui, 
rappelons-le, sert de point de départ pour les itérations que l'on 
souhaite faire --- mais uniquement d'un $\Mk^\GeLa_\nu$ pour un certain 
$\nu > 0$ (que l'on peut, certes, choisir), et
\item comme cela a été expliqué au \S \ref{subsec:algomax}, on 
ne peut pas calculer un $\Max$ en restant à $\nu$ constant à cause des 
problèmes de précision.
\end{enumerate}
On est ainsi amené à reprendre tout ce qui précède en essayant de
remplacer partour $\Sk$ par $\Sk_\nu$. Or, déjà sans aller très loin,
on voit apparaître de nombreux problèmes ; en vrac
\begin{itemize}
\item la notion de $\Max$ n'a pas été définie sur les anneaux $\Sk_\nu$ : 
il va donc falloir, comme cela est expliqué au \S \ref{subsec:algomax} 
travailler sur des extensions $\Sk_{\nu,b}$ puis redescendre à 
$\Sk_\nu$ en prenant des intersections ;
\item il n'est pas vrai, en général, que l'intersection de $\Dk_\nu$
avec un $\Sk_{\nu,b}$-module libre est libre sur $\Sk_\nu$ ; il faudra
donc travailler davantage pour obtenir un module de Breuil-Kisin ;
\item afin de pouvoir appliquer le théorème \ref{theo:surconv}, il
faut que la pente à laquelle on aboutit finalement reste 
$< \frac{p-1} {per}$ ; comme la pente est modifiée après chaque
calcul de $\Max$, il faudra faire attention à contrôler le nombre
de ces calculs.
\end{itemize}
Ces problèmes sont résolus dans les numéros suivants. Précisément, dans 
le \S \ref{subsec:iterfrob} ci-après, on explique comment à partir de la 
donnée de $\Mk^\GeLa_\nu$ (pour un $\nu > 0$) et de l'entier $c_\nu$ 
correspondant, on peut construire, en itérant le Frobenius, un 
$\Sk_{\nu',b}$-module libre $\Mk_{\nu',b}$ stable par $\phi$ (pour un 
certain $b$ et un certain $\nu' > \nu$). L'intersection $\Sk_{\nu',b} 
\cap (\E^+_{\nu'} \otimes_{\Sk_\nu} \Mk^\GeLa_\nu)$ --- que l'on sait 
calculer (voir \S \ref{subsec:algomax}) --- apparaît alors comme un 
candidat raisonnable pour être un module de Breuil-Kisin $\Mk_{\nu'}$. 
Hélas, en général, cela ne fonctionne pas tel quel car rien n'assure 
qu'elle est libre sur $\Sk_{\nu'}$. Nous résolvons ce problème, ainsi 
que quelques autres, dans le \S \ref{subsec:liberte}.

\subsubsection{Itération du Frobenius}
\label{subsec:iterfrob}

À partir de maintenant, on oublie le cas d'école \og $\nu = 0$ \fg\ et on
se place à nouveau dans le cas \og $\nu > 0$ \fg : précisément, on 
suppose donnés un $\Sk_\nu$-module $\Mk_\nu^\GeLa$ muni d'une
application semi-linéaire $\phi : \Mk_\nu^\GeLa \to p^{-c} \cdot 
\Mk_\nu^\GeLa$ pour un certain entier $c = O(dr \cdot \log_p^2(\frac 1 
\nu))$ connu. On suppose en outre que l'on connaît un entier $c_\nu$ 
pour lequel on a la garantie qu'il existe un module de Breuil-Kisin de 
hauteur $\leq r$ compris entre $p^{-c_\nu} \: \Mk^\GeLa_\nu$ et 
$\Mk^\GeLa_\nu$. On rappelle que l'application $\phi$ dont il a été
question ci-dessus est donnée par l'intermédiaire de sa matrice (dans
une certaine base) notée \code{PhiBK} et que celle-ci est connue
modulo $u^N$ pour un certain entier $N$ que l'on peut choisir comme
on le souhaite.

À partir de maintenant, on notera $c_\GeLa$ à la place de $c_\nu$ ; cela 
permettra, malgré les apparences, d'éviter quelques confusions : en 
effet, dans la suite, on sera amené à considérer de multiples autres 
pentes que $\nu$, alors que la constance $c_\GeLa$, elle, sera toujours 
la même.

\paragraph{La suite des $\Mk^{(n)}$ revisitée}

Pour pouvoir appliquer les algorithmes rappelés au \S
\ref{subsec:algomax} à la situation 
présente, on introduit deux nouveaux paramètres entiers $D$ et $N$, que 
l'on choisira judicieusement plus tard. Conformément aux notations du \S 
\ref{subsec:algomax}, l'entier $N$ désignera la précision $u$-adique à 
laquelle les calculs seront effectués. Pour le moment, on conserve une 
certaine liberté sur le choix de $D$ et $N$ et on impose seulement que 
$N$ soit un dénominateur de $\nu$ (\emph{i.e.} que $N\nu$ soit entier), 
que $D$ soit un dénominateur de $\nu'$ et que $D$ divise $N$. On va 
adapter la définition de la suite des $\Mk^{(n)}$ afin que l'action du 
Frobenius sur ceux-ci puisse être calculée à l'aide de l'algorithme 
\ref{algo:changebase}. En fait, plutôt qu'une seule suite, on va en 
définir trois : $(\Mk_{\nu_n,N}^{(n)})$ avec $\nu_n = \nu + \frac {2nc} 
N$, $(\Mk_{\nu_n,D}^{(n)})$ et $(\Mk_{\nu',D}^{(n)})$. Dans l'écriture 
précédente, l'exposant $(n)$ correspond, bien entendu, au rang dans la 
suite tandis que les indices indiquent sur quel anneau le module 
correspondant est défini. On notera en particulier que $N$ est un 
dénominateur commun à tous les $\nu_n$ ; les anneaux $\Sk_{\nu_n,N}$ 
sont donc des anneaux de séries formelles en une variable, au même titre 
d'ailleurs que l'anneau $\Sk_{\nu',D}$ d'après l'hypothèse faite sur 
$D$. Pour l'initialisation, on pose :
$$\Mk_{\nu,N}^{(0)} = \Sk_{\nu,N} \otimes_{\Sk_\nu} \Mk^\GeLa_\nu 
\quad ; \quad
\Mk_{\nu,D}^{(0)} = \Sk_{\nu,D} \otimes_{\Sk_\nu} \Mk^\GeLa_\nu
\quad ; \quad
\Mk_{\nu',D}^{(0)} = \Sk_{\nu',D} \otimes_{\Sk_\nu} \Mk^\GeLa_\nu.$$
Si on suppose maintenant que les suites précédentes sont construites 
jusqu'au rang $n$, on distingue les deux cas suivants :
\begin{itemize}
\item soit le module $\Mk_{\nu',D}^{(n)}$ est stable par $\phi$ et, à ce 
moment, on arrête le processus,
\item soit il existe un vecteur $x_n\in \phi(\Mk_{\nu_n,D}^{(n)})$, $x_n 
\not\in \Mk_{\nu',D}^{(n)}$ et on définit alors 
\begin{eqnarray*}
\Mk_{\nu_{n+1},N}^{(n+1)} & = & \Max_{\nu_{n+1},N} \big(
\Sk_{\nu_{n+1}, N} \otimes_{\Sk_{\nu_n,N}} \Mk_{\nu_n,N}^{(n)} + 
x_n \: \Sk_{\nu_{n+1}, N}\big) \\
\Mk_{\nu_{n+1},D}^{(n+1)} & = & \Mk_{\nu_{n+1},N}^{(n+1)} \cap \big(
\E^+_{\nu_{n+1},D} \otimes_{\Sk_\nu} \Mk^\GeLa_\nu\big) \\
\Mk_{\nu',D}^{(n+1)} & = & \big(\Sk_{\nu',N} \otimes_{\Sk_{\nu_{n+1},N}}
\Mk_{\nu_{n+1},N}^{(n+1)} \big) \cap 
\big( \E^+_{\nu',D} \otimes_{\Sk_\nu} \Mk^\GeLa_\nu\big).
\end{eqnarray*}
\end{itemize}

\smallskip

\noindent On remarquera que, dans ce qui précède, on a fait l'hypothèse 
implicite que $\nu_{n+1} \leq \nu'$ ; en effet, dans le cas contraire, 
les produits tensoriels que l'on a écrits n'ont aucun sens. Dans la 
suite, on choisira les paramètres $D$ et $N$ de sorte que le processus 
s'arrête toujours avant que cette inégalité ne soit violée. On notera 
également qu'il n'est \emph{a priori} pas évident que les deux cas que 
l'on a séparés recouvrent tous les possibles ; en effet, il se pourrait 
très bien, à première vue, que $\phi (\Mk_{\nu_n,D}^{(n)})$ soit inclus 
dans $\Mk_{\nu',D}^{(n)}$ sans pour autant que 
$\phi(\Mk_{\nu',D}^{(n)})$ le soit. Toutefois, d'après le lemme 
\ref{lem:basenupD} ci-après, cela ne peut se produire.

Avec ce qui a été rappelé au \S \ref{subsec:algomax}, une récurrence 
immédiate montre que, tant qu'il est bien défini, le module 
$\Mk_{\nu_n,N}^{(n)}$ admet une base sur $\Sk_{\nu_n,N}$ composée de 
vecteurs de la forme $\varpi_N^{-b_i^{(n)}} B_i^{(n)}$ où les 
$b_i^{(n)}$ sont des entiers, les $B_i^{(n)}$ appartiennent à 
$\E^+_{\nu_n} \otimes_{\Sk_\nu} \Mk^\GeLa_\nu$ et où on rappelle que 
$\varpi_N \in \bar K$ est un élément vérifiant $\varpi_N^N = p$.

\begin{lemme}
\label{lem:basenupD}
On conserve les notations précédentes. Soit $n$ est un entier pour lequel 
$\Mk_{\nu_n,N}^{(n)}$ est défini. Alors le module $\Mk_{\nu',D}^{(n)}$ est 
libre sur $\Sk_{\nu',D}$ et admet pour base la famille des 
$\varpi_D^{-a_i^{(n)}} B_i^{(n)}$ ($1 \leq i \leq d$) où $a_i^{(n)}$ 
désigne la partie entière de $\frac D N \: b_i^{(n)}$.

En particulier, $\Mk_{\nu_n,D}^{(n)}$ engendre $\Mk_{\nu',D}^{(n)}$ comme 
$\Sk_{\nu',D}$-module.
\end{lemme}

\begin{proof}
Pour la première assertion, il s'agit de montrer que si $b$ est un 
entier, alors $\varpi_N^{-b} \Sk_{\nu',N} \cap \E^+_{\nu',D} = 
\varpi_D^{-a} \Sk_{\nu',D}$ où $a$ est la partie entière de $\frac{bD} 
N$. Or, si un élément $x = \sum_{i \in \N} a_i u^i$ appartient à cette 
intersection, on doit avoir $\val(a_i) \in \frac 1 D \Z$ et $\val(a_i) 
\geq \nu' i - \frac b N$ pour tout $i$. Le produit $D \cdot \val(a_i)$ 
doit alors être un entier $\geq D \nu' i - \frac {bD} N$. Comme on a 
supposé que $D$ est un dénominateur de $\nu'$, on en déduit que $D \nu' 
i$ est entier, puis que $D \cdot \val(a_i) \geq D\nu' i - a$. Cela 
signifie exactement que $x \in \varpi_D^{-a} \Sk_{\nu',D}$.

La seconde assertion est maintenant claire puisque, d'une part, 
$\Mk_{\nu_n,D}^{(n)} \subset \Mk_{\nu',D}^{(n)}$ et, d'autre part, tous 
les $\varpi_D^{b_i^{(n)}} B_i^{(n)}$ sont éléments de 
$\Mk_{\nu_n,D}^{(n)}$.
\end{proof}

\begin{lemme}
\label{lem:long}
Le processus itératif que l'on a défini précédemment prend fin au plus
au bout de $c_\GeLa \cdot dD$ étapes.
\end{lemme}

\begin{rem}
La borne ci-dessus peut paraître bien pâle par rapport à ce que nous
avions prouvé dans le cas d'école \og $\nu = 0$ \fg\ (voir lemme
\ref{lem:long0}). Toutefois, nous n'avons pas réussi à adapter la
démonstration de ce lemme à cette nouvelle situation ; de nombreux
problèmes se posent et, en particulier, celui de la non-stabilité 
de $\Sk_{\nu',D}[(\frac u {\varpi_D^{D \nu'}})^{-1}]$ par $\phi$.
Pire encore, des simulations numériques montrent qu'il n'est pas
vrai que, dans le cas $\nu > 0$, le processus prend nécessairement
fin en moins de $d$ étapes ; toutefois, la borne donnée par le 
lemme \ref{lem:long} ne nous paraît pas optimale.
\end{rem}

\begin{proof}
Soit $\ell$ le plus grand entier tel que $\Mk_{\nu',D}^{(\ell)}$ soit 
défini. On veut démontrer que $\ell \leq c_\GeLa \cdot dD$. Pour cela, 
on va mettre à profit la suite d'inclusions suivante :
$$\Mk_{\nu',D}^{(0)} \subsetneq \Mk_{\nu',D}^{(1)} \subsetneq 
\Mk_{\nu',D}^{(2)} \subsetneq \cdots \subsetneq \cdots \subsetneq 
\Mk_{\nu',D}^{(\ell)} \subset p^{-c_\GeLa} \: \Mk_{\nu',D}^{(0)}$$
qui, après passage au déterminant, donne une suite d'inclusions analogue
sur les $\Lambda^d \: \Mk_{\nu',D}^{(n)}$ sauf que le coefficient $p^{-c_\GeLa}$
est remplacé par $p^{-d c_\GeLa}$ dans le dernier module.
Soit $x$ une base de $p^{-dc_\GeLa} \: \Lambda^d \: \Mk_{\nu',D}^{(0)}$ sur 
$\Sk_{\nu',D}$, et soient $s_1, \ldots, s_\ell$ des éléments de 
$\Sk_{\nu',D}$ tels que $s_n x$ soit une base de $\Lambda^d \:
\Mk_{\nu'}^{(n)}$. On a alors $s_0 = p^{c_\GeLa \cdot d} = \varpi_D^{c_\GeLa \cdot dD}$ et 
$s_{n+1}$ divise $s_n$ pour tout $n \in \{0, \ldots, \ell-1\}$. Or 
$\varpi_D$ engendre un idéal premier dans $\Sk_{\nu',D}$. Il en résulte 
qu'à multiplication par des éléments inversibles près, tous les $s_n$ 
sont des puissances de $\varpi_D$. Les exposants qui apparaissent 
forment une suite d'entiers strictement décroissante qui commence à 
$c_\GeLa \cdot dD$ et se termine à $0$. Clairement, sa longueur est donc majorée par 
$c_\GeLa \cdot dD + 1$.
\end{proof}

Ainsi, si l'on choisit $N$ supérieur ou égal à $\frac{2c \:c_\GeLa \cdot 
dD}{\nu'-\nu}$, on a $\nu_n \leq \nu'$ pour tout $n \leq c_\GeLa \cdot dD$, 
et le lemme ci-dessus implique que le processus itératif que l'on a 
défini précédemment ne peut s'arrêter en raison d'une violation de 
l'inégalité $\nu_n \leq \nu'$. Autrement dit, il s'arrête nécessairement 
lorsque l'on a trouvé un $\Mk_{\nu',D}^{(n)}$ stable par $\phi$.

\paragraph{L'algorithme proprement dit}

Le procédé itératif précédent peut, sans problème, être transformé en un 
véritable algorithme qui calcule la matrice donnant l'action de $\phi$ 
sur $\Mk_{\nu_n,N}^{(n)}$ et s'arrête lorsque $\Mk_{\nu',D}^{(n)}$ est
stable par $\phi$. En effet, du fait que $\Mk_{\nu,N}^{(0)}$ est stable 
par $p^{-c} \cdot \phi$, on déduit que l'on peut utiliser l'algorithme 
\ref{algo:changebase} pour calculer l'action de $\phi$ sur 
$\Mk_{\nu_1,N}^{(1)}$ et, en examinant cet algorithme, on montre que
$\Mk_{\nu_1,N}^{(1)}$ est, lui aussi, stable par $p^{-c} \cdot \phi$. On
peut ainsi itérer le procédé et calculer, comme annoncé, l'action
de $\phi$ sur chacun des $\Mk_{\nu_n,D}^{(n)}$. Comme en outre, le
processus précédent prend fin au bout d'au plus $c_\GeLa \cdot dD$ 
étapes (par le lemme \ref{lem:long}), la pente $\nu$ se dépasse jamais
la valeur critique $\nu'$. 

Par ailleurs, pour tester la stabilité de $\Mk_{\nu',D}^{(n)}$, on peut 
procéder comme suit : (1)~à partir du lemme \ref{lem:basenupD} et la 
matrice donnant l'action de $\phi$ sur $\Mk_{\nu_n,N}^{(n)}$ calculée 
précedemment, on détermine la matrice de l'action de $\phi$ sur 
$\Mk_{\nu',D}^{(n)}$ et (2)~on vérifie que cette matrice est à 
coefficients dans $\Sk_{\nu'}$ ce qu'il est possible de tester en ne
regardant que les termes en $u^i$ pour $i < N$ étant donné que $\nu' 
\geq \nu_n + \frac c N$\footnote{Cela pourrait éventuellement ne pas se 
produire si $n = c_\GeLa \cdot dD$. Mais, dans ce cas\footnotemark, la 
borne du lemme \ref{lem:long} apporte la garantie que 
$\Mk^{(n)}_{\nu',D}$ est automatiquement stable par $\phi$ et on peut 
donc se dispenser de faire le test.},

\footnotetext{En fait, ce cas ne se produit jamais. En effet, toujours 
d'après le lemme \ref{lem:long}, si l'on arrive à calculer $\Mk^{(n)} 
_{\nu',D}$ pour $n = cdD$, nécessairement $\Mk^{(n)}_{\nu'} = p^{-c} \: 
\Mk^{(0)}_{\nu',D}$ et celui-ci est stable par $\phi$. Ainsi 
$\Mk^{(0)}_{\nu,D'}$ serait lui aussi stable par $\phi$ et on n'aurait 
dû s'arrêter avant même la première itération.}

L'algorithme \ref{algo:iterfrob} résume, de façon concise ce que l'on 
vient de dire.

\begin{algorithm}
  \SetKwInOut{Input}{Entrée :}
  \SetKwInOut{Output}{Sortie :}
  \SetKwInOut{Void}{}

  \Input{$\star$ des nombres rationnels $0 < \nu < \nu'$}
  \BlankLine
  \Void{$\star$ une matrice $\code{PhiBK} \in M_d(\E^+_\nu)$}
  \Void{\hphantom{$\star$} donnant l'action, dans la base canonique, d'un opérateur $\phi$-semi-linéaire $\phi$ sur $(\E^+_\nu)^d$}
  \BlankLine
  \Void{$\star$ une constante $c \geq 0$ telle que $p^c \cdot \code{PhiBK} \in M_d(\Sk_\nu)$}
  \BlankLine
  \Void{$\star$ une constante $c_\GeLa \geq 0$ telle qu'il existe}
  \Void{\hphantom{$\star$} un module de Breuil-Kisin $\Mk_\nu$ compris entre $p^{-c} \: (\E^+_\nu)^d$ et $(\E^+_\nu)^d$}

  \BlankLine

  \Output{$\star$ la matrice de $\phi$ agissant sur un module de Breuil-Kisin 
    $\Mk_{\nu'}$ sur $\Sk_{\nu'}$}
  \Void{\hphantom{$\star$} tel que $p^{-c} \: \Sk_{\nu'}^d \subset \Mk_{\nu'} \subset \Sk_{\nu'}^d$}

  \BlankLine
  \BlankLine

  $D \leftarrow$ un dénominateur de $\nu'$\;
  $N \leftarrow$ plus petit multiple de $D$ supérieur ou égal à $\frac{c_\GeLa \cdot dD}{\nu'-\nu}$\;
  $a \leftarrow (0, \ldots, 0);$\,
  $b \leftarrow (0, \ldots, 0)$\;
  \While{il existe $(i,j)$ tel que $D \cdot v_\nu(\code{PhiBK}[i,j]) < a[j] - a[i]$}{
    $(i,j) \leftarrow$ un tel couple\;
    $b,\code{PhiBK} \leftarrow$ 
    \textsc{ChangeBase}$(b, \code{PhiBK}, \frac N D a[j], \code{PhiBK}[\undef,j], c, N, \nu)$\;
    \lFor{$i$ allant de $1$ à $d$}{$a[i] \leftarrow \text{Floor}(\frac{b[i] D}N)$}\;
    $\nu \leftarrow \nu + \frac c N$\;
  }
  \Return{$a,\code{PhiBK}$}\;
\caption{{\sc IterationFrobenius$(\nu,\nu',\code{PhiBK},c)$}
\label{algo:iterfrob}}
\end{algorithm}

\begin{rem}
Dans la remarque \ref{rem:fastiter}, nous avions présenté une idée 
permettant, semble-t-il, d'aboutir à une convergence beaucoup plus 
rapide vers la solution et consistant à itérer non pas l'opérateur 
$\phi$ lui-même mais ses puissances. Toutefois, lorsque l'on essaie
de l'appliquer à notre situation précise, on se heurte à un certain
nombre de difficultés que les auteurs de l'article n'ont pas réussi
à surmonter. Le souci majeur est que la \og construction $\Max$ \fg\ 
ne commute pas avec les sommes lorsque les pentes changent.
\end{rem}

\subsubsection{La liberté rendue}
\label{subsec:liberte}

À ce niveau, on a réussi à calculer un $d$-uplet d'entiers $a = (a_1, 
\ldots, a_d)$ et une matrice \code{PhiBK} à coefficients dans $\Sk_\nu$ 
pour lesquels on est assuré qu'il existe une $\E^+_{\nu'}$-base $(e_1, 
\ldots, e_d)$ de $\Dk_{\nu'} = \E^+_{\nu'} \otimes_{\E^+} \Dk$ (où on
rappelle que $\Dk$ désigne le module de Breuil-Kisin sur $\E^+$ associé au
$(\phi,N)$-module filtré $D$ avec lequel on est parti) telle que
\code{PhiBK} soit la matrice de $\phi$ dans la base $(\varpi_D^{-a_1}
e_1, \ldots, \varpi_D^{-a_d} e_d)$. En particulier, l'espace 
$$\Mk_{\nu',D} = \Sk_{\nu',D} \varpi_D^{-a_1} e_1 \oplus \cdots
\oplus \Sk_{\nu',D} \varpi_D^{-a_d} e_d$$
est stable par $\phi$. Toutefois, pour obtenir un module de Breuil-Kisin, on ne 
souhaite pas un espace défini sur $\Sk_{\nu',D}$ mais, bel et bien, un 
module \emph{libre} défini sur $\Sk_{\nu'}$. Il reste donc encore à 
redescendre $\Mk_{\nu',D}$ en un $\Sk_{\nu'}$-module et, pour cela, il 
est naturel de considérer l'intersection $\Mk_{\nu'} = \Mk_{\nu',D} 
\cap \Dk$ qui est à l'évidence, elle aussi, stable par $\phi$. Par
contre, \emph{a priori}, rien n'indique qu'elle est libre. Toutefois si 
$\nu' = \frac 1 D$ --- ce que l'on supposera à partir de maintenant --- 
on sait (voir les rappels du \S \ref{subsec:algomax}) que cette 
intersection est engendrée par les vecteurs $p^{q_i+1} e_i$ et $p^{q_i} 
u^{r_i} e_i$ ($1 \leq i \leq d$) où $q_i$ et $r_i$ désignent 
respectivement le quotient et le reste de la division euclidienne de 
$-a_i$ par $D$. Or, on peut manifestement écrire :
$$p^{q_i} u^{r_i} = p^{q_i+1} \cdot \frac{u^{r_i}} p \in p^{q_i+1}
\Sk_{1/r_i}$$
ce qui prouve que si $\nu''$ est supérieur ou égal à $\nu'$ et à tous 
les $\frac 1 {r_i}$ dès que $r_i \neq 0$, alors le module $\Mk_{\nu''} = 
\Sk_{\nu''} \otimes_{\Sk_{\nu'}} \Mk_{\nu'}$ est libre sur $\Sk_{\nu''}$ 
et admet pour base la famille des $u^{q_i+\epsilon_i} e_i$ où 
$\epsilon_i$ vaut $0$ si $r_i = 0$ et $1$ sinon.
Toutefois, cela n'est pas encore satisfaisant car si l'un des $r_i$ vaut 
$1$, l'argument précédent nous oblige à choisir un $\nu'' \geq 1$ et 
l'on ne peut alors plus appliquer le théorème de surconvergence des 
modules de Breuil-Kisin. Il faudrait donc pouvoir réussir à s'assurer que tous 
les restes $r_i$ non nuls sont suffisamment grands ; c'est, en quelque 
sorte, le contenu du lemme suivant.

\begin{lemme}
\label{lem:entiert}
Il existe un entier $t$ tel que, pour tout $i \in \{1, \ldots, d\}$,
le reste de la division euclidienne de $-(t+a_i)$ par $D$ soit égal à $0$
ou $\geq \frac D d$.
\end{lemme}

\begin{proof}
On note $\rho_1 < \cdots < \rho_d$ les restes des divisions euclidiennes 
des $-a_i$ par $D$ rangés par ordre croissant et on pose $\rho_{d+1} = 
\rho_1 + D$. Il est alors clair qu'il existe un $i$ tel que $\rho_{i+1} 
\geq \rho_i + \frac D d$ et l'entier $t = \rho_i$ vérifie la propriété 
du lemme.
\end{proof}

Soit $t$ un entier satisfaisant la condition du lemme. Si on remplace le 
$d$-uplet $a$ par $t + a = (t + a_1, \ldots, t + a_d)$, le module 
$\Mk_{\nu',D}$ est remplacé par $\varpi_D^{-t} \Mk_{\nu',D}$ et reste 
donc stable par $\phi$. Il en va donc de même de $\Mk_{\nu'}$ et de 
$\Mk_{\nu''}$ et, comme précédemment, ce dernier est libre sur 
$\Sk_{\nu''}$. La différence est que, maintenant, grâce à la 
modification que l'on a faite, on peut choisir $\nu'' = d\nu'$. 

Pour conclure, il reste encore à vérifier que $\Mk_{\nu''}$ est bien un 
module de Breuil-Kisin, c'est-à-dire qu'il est bien de hauteur finie. 
Cela se fait simplement en reprenant la première partie de la 
démonstration de la proposition \ref{prop:Eurauto} : on 
obtient\footnote{Nous n'avons pas réussi à adapter la deuxième partie de 
la preuve ; nous ne savons donc pas si $\Mk_{\nu''}$ est nécessairement 
de hauteur $\leq r$. Pouvoir montrer cela améliorerait la complexité 
finale de l'algorithme.}, ce faisant, que $\Mk_{\nu''}$ est de hauteur 
$\leq dr$. Ainsi si $\nu''$ est strictement inférieur à 
$\frac{p-1}{perd}$, le théorème de surconvergence des modules de 
Breuil-Kisin s'applique et affirme que le $\Mk_{\nu''}$ que l'on vient 
de calculer provient par extension des scalaires d'un module de 
Breuil-Kisin $\Mk$ sur $\Sk$. En particulier, $\Mk_{\nu''}$ correspond 
bien à un réseau stable par $G_\infty$ dans la représentation 
semi-stable associée à au $(\phi,N)$-module filtré $D$ duquel on est 
parti !

\subsubsection{L'algorithme sous forme synthétique}
\label{subsec:etape2synthese}

On rappelle qu'à l'issue de l'étape 1 (\emph{cf} \S \ref{subsec:etape1}),
on a calculé :
\begin{enumerate}[(i)]
\item une matrice \code{PhiBK} connue à précision $u$-adique $u^N$ 
donnant l'action du Frobenius $\phi$ sur le module de Breuil-Kisin
$\Dk_\nu$ ;
\item une constante $c$ tel que $p^c \cdot \code{PhiBK}$ soit à 
coefficients dans $\Sk_\nu$ ;
\item \label{item:cnu} une constante $c_\nu = c_\GeLa$ pour laquelle on a la 
garantie qu'il existe un module de Breuil-Kisin $\Mk_\nu$ de hauteur 
$\leq r$ compris entre $\Mk_\nu^\GeLa$ et $p^{-c} \cdot \Mk_\nu^\GeLa$
sachant que $\Mk_\nu^\GeLa$ désigne le $\Sk_\nu$-module ayant pour
base celle dans laquelle est écrite la matrice \code{PhiBK}.
\end{enumerate}
Les paramètres $N$ et $\nu$ qui sont apparus ci-dessus sont 
respectivement un entier strictement positif et un nombre rationnel
strictement positif qui peuvent être choisis comme on le souhaite.
Dans la suite, conformément aux contraintes qui ont été dégagées 
précédemment, on choisira un entier $D$ strictement supérieur à 
$\frac{perd^2}{p-1}$ et on prendra :
\begin{equation}
\label{eq:defnu}
N = 4 c \: c_\GeLa \cdot D^2
\quad \text{et} \quad
\nu = \frac 1 {2D}.
\end{equation}
On posera également $\nu' = \frac 1 D = 2 \nu$.
Dans les numéros précédents, nous avons présenté une méthode 
algorithmique pour calculer la matrice donnant l'action de $\phi$
sur un module de Breuil-Kisin $\Mk_\nu$ satisfaisant à la condition
de l'alinéa \eqref{item:cnu} ci-dessus, qui peut se résumer ainsi :

\begin{enumerate}[(1)]
\item Appliquer l'algorithme \ref{algo:iterfrob} avec les paramètres
précédents et noter $a, \code{PhiBK}$ le couple renvoyé
\item Appliquer l'algorithme \ref{algo:liberte} ci-après.
\end{enumerate}
\begin{algorithm}
  {\it // On calcule un entier $t$ vérifiant la condition du lemme
  \ref{lem:entiert}}

  $(\rho_1, \ldots, \rho_d)$ $\leftarrow$ restes des divisions
  euclidiennes de $-a_i$ par $D$ triés par ordre croissant\;
  \For{$i$ allant de $1$ à $d-1$}{
    \lIf{$\rho_{i+1} \geq \rho_i + \frac D d$}{
      $t$ $\leftarrow$ $\rho_i;$\, {\bf break}\;}
  }

  \BlankLine

  {\it // On calcule la matrice \code{PhiBK} de $\phi$ agissant sur
  $\Mk_{\nu''}$}

  \For{$i$ allant de $1$ à $d$}{
    $(q_i, r_i)$ $\leftarrow$ quotient et reste de la division
    euclidienne de $-(a_i+t)$ par $D$\;
    \lIf{$r_i > 0$}{$q_i$ $\leftarrow$ $q_i + 1$\;}
  }
  $D$ $\leftarrow$ la matrice diagonale dont les éléments
  diagonaux sont $p^{q_1}, \ldots, p^{q_d}$\;
  \code{PhiBK} $\leftarrow$ $D \cdot \code{PhiBK} \cdot D^{-1}$\;
  \Return \code{PhiBK}
\caption{\sc 
Liberte$(a, \code{PhiBK})$\label{algo:liberte}}
\end{algorithm}
La matrice \code{PhiBK} renvoyée par l'algorithme \ref{algo:liberte} est 
la matrice de l'action de $\phi$ sur un module de Breuil-Kisin $\Mk_{d 
\nu'}$ de hauteur $\leq rd$. Comme en outre, par notre choix de $\nu'$, 
on a $d \nu' < \frac{p-1}{perd}$, le théorème de surconvergence des
modules de Kisin s'applique.

Étant donné que l'algorithme \ref{algo:changebase} n'effectue des 
divisions euclidiennes qu'entre éléments de $\E^+_\nu$ (pour différents 
$\nu$) qui sont des polynômes, il est en outre facile d'analyser les 
pertes de précision $p$-adiques engendrées par la méthode ci-dessus : on 
trouve que si les coefficients de la matrice \code{PhiBK} initiale sont 
connus modulo $(p^M \: \Sk_\nu + u^N \: \Sk_\nu)$, alors ceux de la 
matrice \code{PhiBK} calculée sont connus au pire modulo $(p^{M-2} \: 
\Sk_{\nu''} + u^N \: \Sk_{\nu''})$. En effet, on remarque dans un 
premier temps que l'algorithme \ref{algo:changebase} n'entraîne aucune 
perte de précision, dans le sens où si les coefficients de
$$\varpi_N^{-a} X \quad \text{et} \quad
\Delta(b,N)^{-1} \cdot \code{PhiBK} \cdot \Delta(b,N) \quad \text{avec}
\quad \Delta(b,N) = \Diag(\varpi_N^{b_1}, \ldots, \varpi_N^{b_d})$$
calculées à partir des entrées sont connus modulo $(p^M \: \Sk_\nu + 
u^N \: \Sk_\nu)$, alors ceux de la matrice $\Delta(b,N)^{-1} \cdot 
\code{PhiBK} \cdot \Delta(b,N)$ calculée à partir de la sortie sont 
connus modulo $(p^M \: \Sk_{\nu'} + u^N \: \Sk_{\nu'})$ où, dans cette 
phrase et dans cette phrase uniquement, $\nu$ et $\nu'$ désignent les 
rationnels pris en entrée par l'algorithme \ref{algo:changebase}. Il 
résulte de ceci que les coefficients de la matrice $\Delta(a,D)^{-1} 
\cdot \code{PhiBK} \cdot \Delta(a,D)$ calculée à partir de la sortie de 
l'algorithme \ref{algo:iterfrob} sont connus modulo $(p^M \: \Sk_{\nu'} 
+ u^N \: \Sk_{\nu'})$ avec, cette fois-ci, à nouveau, $\nu' = \frac 1 
D$. On en déduit enfin la propriété annoncée en remarquant que les $q_i$ 
calculés dans l'algorithme \ref{algo:liberte} diffèrent d'au plus $1$ 
des $\frac {-a_i} D$.

\subsection{Étape 3 : Réduction modulo $p$ et semi-simplification}
\label{subsec:etape3}

À l'issue de l'étape précédente, nous avons calculé $\Mk_{\nu''} = 
\Sk_{\nu''} \otimes_\Sk \Mk$ pour un certain nombre rationnel $\nu'' < 
\frac{perd}{p-1}$. Le théorème de surconvergence des modules des modules 
de Breuil-Kisin nous assure que cela est suffisant pour retrouver $\Mk$. 
Toutefois, bien que la démonstration de ce théorème soit assez 
constructive, il n'est pas aisé, dans la pratique, de calculer $\Mk$ à 
partir de $\Mk_{\nu''}$. Par chance, nous n'en avons pas besoin ! En 
effet, en se rappelant de surcroît que $\nu'' = \frac 1 {dD}$ est 
l'inverse d'un nombre entier, l'égalité $\Mk_{\nu''} = \Sk_{\nu''} 
\otimes_\Sk \Mk$ montre que la connaissance de $\Mk_{\nu''}$ suffit à 
déterminer le quotient 
$$\Mk/(p \Mk + u^{1/\nu''} \Mk) \simeq \Mk_{\nu''}/(p \:\Mk_{\nu''} + 
u^{1/\nu''} \Mk_{\nu''}).$$
Or, étant donné que $\frac 1{\nu''} = dD \geq \frac{perd}{p-1}$, il est 
facile de démontrer par ailleurs (voir par exemple lemme 5 de 
\cite{caruso-phimod}) que deux modules de Breuil-Kisin de hauteur $\leq 
rd$ sur $\Sk/p\Sk \simeq k[[u]]$ qui deviennent égaux après réduction 
modulo $u^{1/\nu''}$ sont isomorphes. Ainsi, en pratique, connaissant 
$\Mk_{\nu''}$, il suffit de réduire ce module modulo $(p, u^{1/\nu''})$ 
puis de le relever n'importe comment sur $k[[u]]$.

À partir de là, calculer la semi-simplifiée de la représentation 
associée fait l'objet du chapitre 3 de la thèse de Le Borgne 
\cite{leborgne}. Nous ne nous attarderons donc pas davantage sur ce 
point et nous contentons de renvoyer le lecteur à cette référence.

\subsection{Étude de la complexité}
\label{subsec:complexite}

Nous terminons cette section par une étude sommaire de la complexité de 
l'algorithme que nous venons de présenter. Nous nous proposons, plus
précisément, de montrer que celle-ci est polynômiale en tous les
paramètres pertinents du problème, qui sont :
\begin{itemize}
\item l'entier $e$ qui est l'indice de ramification absolue de $K$ ;
\item l'entier $f$ qui est le degré de $k$ (que l'on suppose fini, on
rappelle) sur son sous-corps premier $\F_p$ ;
\item la valuation $p$-adique de la différente de $K/K_0$, notée $\delta$ ;
\item l'entier $r$ qui correspond (quitte à twister correctement $V$)
à la différence entre les deux poids de Hodge-Tate extrèmes de $V$.
\end{itemize}

\medskip

Avant de poursuivre, remarquons que, d'après les précisions qui ont été 
faites au \S \ref{subsec:repr} sur la représentation des éléments, la 
complexité des opérations arithmétiques élémentaires --- à savoir 
l'addition, la soustraction et la multiplication --- dans les anneaux 
$\O_{K_0} \simeq \Z_q$, $\O_K$, $K_0 \simeq \Q_q$, $K$, $\Sk_\nu$ et 
$\E^+_\nu$ est polynômiale en la taille de l'entrée, et même 
quasi-linéaire\footnote{Au moins, dans la cas des éléments de $\Sk_\nu$ 
et $\E^+_\nu$, si l'on suppose que les coeffcients des séries à 
multiplier sont tous connus avec à peu près la même précision.} si
on utilise des algorithmes basées sur la transformée de Fourier
rapide. Ceci vaut également pour la division euclidienne dans
$\Sk_\nu$ ; on renvoie le lecteur aux \S\S 2.3 et 4.2 de \cite{carlub} 
à ce propos.

\subsubsection{Coût de l'algorithme {\sc ChangeBase}}
\label{subsec:coutchangebase}

L'algorithme {\sc ChangeBase} (algorithme \ref{algo:changebase}, page 
\pageref{algo:changebase}) joue un rôle essentiel dans la deuxième étape 
de notre algorithme. Dans cette partie, nous étudions sa complexité. On 
note $M$ un entier pour lequel tous les coefficients des matrices
$$\varpi_N^{-a} X \quad \text{et} \quad 
\Delta(b,N)^{-1} \cdot \code{PhiBK} \cdot \Delta(b,N) 
\quad \text{avec} \quad 
\Delta(b,N) = \Diag(\varpi_N^{b_1}, \ldots, \varpi_N^{b_d})$$
sont connus modulo $p^M \cdot \Sk_\nu$ ou, ce qui revient au même, un
entier qui majore tous les $N_i$ de la représentation PAG (voir \S
\ref{subsec:repr}) des coefficients des matrices ci-dessus.

\medskip

Le calcul de la division euclidienne de la ligne \ref{lgn:diveuc} se 
fait en temps polynômial en $N(M+c)$ ; en effet, la représentation PAG 
des éléments $X[i]$ et $X[j]$ que l'on divise entre eux ont une taille 
qui ne dépasse pas $O(N(M+c))$. Il est alors clair que chaque exécution 
de la boucle \emph{tant que} (lignes \ref{lgn:whilestart} à 
\ref{lgn:whileend}) prend aussi un temps polynômial en $N(M+c)$.

Il ne reste donc plus qu'à majorer le nombre d'exécutions de cette
boucle. Pour cela, on considère le triplet $(-v, \delta, n)$ où :
\begin{itemize}
\item $v$ désigne le minimum des $v_{\nu'}(X[i]) + b_i$ pour $i$
variant dans $\{1, \ldots, d\}$,
\item $\delta$ désigne le maximum des $\deg_{\nu'}(X[i])$ pour $i$
parcourant les indices tels que $v_{\nu'}(X[i]) + b_i = v$,
\item $n$ désigne le nombre d'indices $i$ tels que $v = v_{\nu'}(X[i]) 
+ b_i$ et $d = \deg_{\nu'}(X[i])$
\end{itemize}
et on remarque qu'à chaque itération de la boucle, celui-ci diminue 
strictement pour l'ordre lexicographique donnant le poids le plus fort
à la première coordonnée. En effet, on constate tout 
d'abord que $v$ ne peut qu'augmenter, et donc $-v$ ne peut que diminuer.
De plus, en reprenant les notations de l'algorithme, on s'aperçoit que 
si le minimum des $v_{\nu'}(X[i]) + b_i$ est atteint pour un unique 
indice $i$, alors $b_i$ augmente strictement à la ligne 
\ref{lgn:updatebi} et, par suite, que $v$ augmente strictement dans ce 
cas. Si le minimum des $v_{\nu'}(X[i]) + b_i$ est atteint pour au moins 
deux indices $i$ et $j$, c'est-à-dire si $v_i = v_j$, on s'arrange à la 
ligne \ref{lgn:swapij} pour que $\deg_{\nu'}(X[i]) \geq 
\deg_{\nu'}(X[j])$ puis on remplace à la ligne \ref{lgn:updateXi} le 
coefficient $X[i]$ par le reste de sa division euclidienne par $X[j]$, 
faisant ainsi chuter son degré de Weierstrass à un entier $< 
\deg_{\nu'}(X[j])$. Ainsi, en supposant que $v$ n'augmente pas, si 
$X[i]$ et $X[j]$ avait même degré de Weierstrass, le nombre $\delta$ reste
inchangé alors que $n$ diminue tandis que si on avait $\deg_{\nu'}(X[i]) > 
\deg_{\nu'}(X[j])$, le nombre $\delta$ diminue. Dans tous les cas, le 
triplet $(-v, d, n)$ diminue donc strictement pour l'ordre 
lexicographique.

Pour conclure, il suffit de majorer le nombre de valeurs distinctes
que peut prendre le triplet $(-v, d, n)$. La coordonnée $d$ (resp.
$n$) est un entier qui varie entre $0$ et $N-1$ (resp. entre $1$ et 
$d$) ; elle peut donc prendre $N$ (resp. $d$) valeurs. Le nombre $v$,
quant à lui, est un rationnel dont le dénominateur est divisible par
le dénominateur de $\nu$, à savoir $N$. Par ailleurs, il ne
peut descendre en dessous de $-a$ (c'est la condition d'arrêt de 
l'algorithme) et, par définition de l'entier $c$, il ne peut excéder
$c-a$. Ainsi, il varie dans un ensemble de cardinal $cN$. En mettant
tout ensemble, on en déduit que le triplet $(-v, d, n)$ prend au
maximum $N^2 \cdot cd$ valeurs et, par suite, que la boucle \emph{tant
que} de l'algorithme \ref{algo:changebase} est exécutée au maximum
$N^2 \cdot cd$.

Il résulte de ce qui précède que la complexité de l'algorithme 
\ref{algo:changebase} est polynômiale en les paramètres $N$, $M$, $c$,
$d$ et, également, $e$ et $f$ qui contrôlent la \og taille \fg\ de $K$.

\subsubsection{Examen de la précision $p$-adique}

Avant de pouvoir conclure, il nous reste à estimer la précision 
$p$-adique avec laquelle l'entrée de notre algorithme --- c'est-à-dire 
le $(\phi,N)$-module filtré $D$ de départ donné, on le rappelle, par le 
quadruplet $(\code{Phi}, \code{N}, \code{H}, \code{F})$ --- doit être 
connue pour que tout le calcul puisse fonctionner correctement.

À cet effet, on remarque qu'à l'entrée de la troisième étape, le module 
$\Mk_{\nu''}$ que l'on a calculé est réduit modulo $(p \: \Mk_{\nu''} +
u^{1/\nu''} \: \Mk_{\nu''})$. 
Ainsi, il suffit de connaître les coordonnées de la matrice \code{PhiBK} 
donnant l'action de $\phi$ sur $\Mk_{\nu''}$ modulo $(p \: \Sk_{\nu''}
+ u^N \: \Sk_{\nu''})$ car $N \geq \frac 1 {\nu''}$. Ainsi, l'entier $M$ 
du \S \ref{subsec:etape2synthese} peut être choisi égal à $3$ étant 
donné que l'on a vu que le résultat calculé à l'issue de l'étape 2 est 
connu modulo $(p^{M-2} \: \Sk_{\nu''} + u^N \: \Sk_{\nu''})$.
On en déduit que l'entier $M$ du \S \ref{subsec:etape1synthese}, 
c'est-à-dire la précision de l'entrée que l'on cherche à évaluer, peut
être choisi égal $M_0 + 3$ où $M_0$ est le nombre qui a été défini au
\S \ref{subsec:etape1synthese}. En particulier, il résulte de la 
formule \eqref{eq:precreletape1} que $M_0$ dépend de façon polynômiale
des paramètres $d$, $r$ et $\log_p(\frac 1 \nu)$. En outre, d'après
la formule \eqref{eq:defnu}, la quantité $\frac 1 \nu$ est contrôlée 
polynomialement par les paramètres $c$, $c_\GeLa$, $e$, $r$ et $d$.
Or, dans l'introduction du \S \ref{subsec:etape2}, on a vu que $c$ et 
$c_\GeLa$ sont eux-mêmes controlés polynomialement par $e$, $r$, $d$ et 
$\delta$ (qui désigne, on rappelle, la valuation $p$-adique de la différente 
de $K$ à $K_0$). On trouve ainsi que $M_0$, et donc également $M_0 + 3$, 
est inférieur à un certain polynôme dépendant des variables $e$, $r$,
$d$ et $\delta$.

\subsubsection{Un algorithme de complexité polynômiale}

À la lumière de ce qui précède, il n'est pas difficile de conclure que 
l'algorithme de calcul de la semi-simplifiée modulo $p$ que nous avons 
présenté dans cet article a une complexité polynômiale en $e$, $f$, $r$,
$d$ et $\delta$.
En effet, revenant à la synthèse du \S \ref{subsec:etape1synthese}, on
démontre immédiatement que l'étape 1 de notre algorithme a une complexité
polynômiale en $e$, $f$, $r$, $d$, $N$, $n \approx \log_p(\frac 1 \nu)$ et
$M_0$ et donc, d'après les dépendances qui ont été explicitées ci-dessus, 
a également une complexité polynômiale en $e$, $f$, $r$, $d$ et $\delta$. 
De manière similaire, en étudiant la synthèse du \S 
\ref{subsec:etape2synthese}, on obtient une complexité polynômiale en
$e$, $f$, $r$, $d$ et $\delta$ pour l'étape 2 de notre algorithme. Enfin,
les résultats de \S III.2.5.2 de \cite{leborgne} montrent que l'étape
3 de notre algorithme a également une complexité polynômiale en $e$,
$f$, $r$ et $d$. Mettant tout ensemble, on conclut.



\begin{thebibliography}{99}
\bibitem{breuil}
 C. Breuil, 
 \emph{Schémas en groupes et corps des normes},
 prépublication (1998),
 \url{http://www.math.u-psud.fr/~breuil/PUBLICATIONS/groupesnormes.pdf}

\bibitem{breuil2}
 C. Breuil, 
 \emph{Une application du corps des normes},
 Compositio Math. {\bf 117} (1999), 189--203

\bibitem{breuil-mezard}
 C. Breuil, A. Mézard,
 \emph{Multiplicités modulaires et représentations de 
 $\text{GL}_2(\Z_p)$ et de $\Gal(\bar \Q_p/\Q_p)$ en $\ell=p$},
 Duke Math. J. {\bf 115} (2002), 205--310

\bibitem{caruso-phimod}
 X. Caruso,
 \emph{Sur la classification de quelques $\varphi$-modules simples},
 Mosc. Math. J. {\bf 9} (2009), 562--568

\bibitem{caruso}
 X. Caruso,
 \emph{Représentations galoisiennes $p$-adiques et $(\varphi,\tau)$-modules},
 prépublication (2010)

\bibitem{caruso-LU}
 X. Caruso,
 \emph{Random matrices over a DVR and LU factorization},
 prépublication (2012)

\bibitem{carliu}
 X. Caruso, T. Liu,
 \emph{Some bounds for ramification of $p^n$-torsion semi-stable 
 representations},
 J. of Algebra {\bf 325} (2011), 70--96

\bibitem{carlub}
 X. Caruso, D. Lubicz,
 \emph{Linear Algebra over $\Z_p[[u]]$ and related rings},
 prépublication (2012)

\bibitem{colmez-fontaine}
 P. Colmez, J.-M. Fontaine,
 \emph{Construction des représentations $p$-adiques semi-stables},
 Invent. Math. {\bf 140} (2000), 1--43

\bibitem{ast-fontaine}
  J.-M. Fontaine, \emph{Le corps des p\'eriodes $p$-adiques}, 
  Ast\'erisque {\bf 223}, Soc. math. France (1994), 59--111

\bibitem{genlaf}
 A. Génestier, V. Lafforgue,
 \emph{Structures de Hodge-Pink pour les $\varphi/\Sk$-modules de 
 Breuil et Kisin},
 à paraître à Compositio Math.

\bibitem{kisin}
 M. Kisin,
 {\it Crystalline representations and F-crystals},
 Algebraic Geometry and Number Theory,
 Drinfeld 50th Birthday volume,
 459--496

\bibitem{leborgne}
 J. Le Borgne,
 \emph{Représentations galoisiennes et $\varphi$-modules : aspects algorithmiques}
 thèse de doctorat (2012)

\bibitem{savitt}
 D. Savitt,
 \emph{On a Conjecture of Conrad, Diamond, and Taylor},
 Duke Math. J. {\bf 128} (2005), no. 1, 141-197

\end{thebibliography}
\end{document}